\documentclass{aomart}
\usepackage{graphicx}
\usepackage{amsmath}
\usepackage{amsaddr}
\usepackage{amssymb}
\usepackage{textcomp}
\usepackage{mathrsfs}

\usepackage{chngcntr}
\usepackage{apptools}
\usepackage{verbatim}

\newtheorem{theorem}{Theorem}

\newtheorem{lemma}{Lemma}

\theoremstyle{definition}

\newtheorem{problem}{Problem}

\theoremstyle{remark}
\newtheorem{remark}{Remark}
\newtheorem{corollary}{Corollary}
\newtheorem{exercise}{Exercise}

\numberwithin{equation}{section}

\newcommand\restr[2]{{
		\left.\kern-\nulldelimiterspace 
		#1 
		\vphantom{\big|} 
		\right|_{#2} 
}}

\numberwithin{theorem}{subsection}
\numberwithin{lemma}{subsection}
\numberwithin{remark}{subsection}
\numberwithin{exercise}{subsection}
\numberwithin{corollary}{subsection}

\title[Inertial manifolds and foliations]
{Inertial manifolds and foliations for asymptotically compact cocycles in Banach spaces}

\author{Mikhail Anikushin}
\address{Department of Applied Cybernetics, Faculty of Mathematics and Mechanics, St. Petersburg University, 28 Universitetskiy prospekt, Peterhof, 198504, Russia}
\email{demolishka@gmail.com}

\makeatletter 
\fancypagestyle{firstpage}{%
	\fancyhf{}%
	\if@aom@manuscript@mode
	\lhead{\begin{picture}(0,0)%
			\put(-21,-25){\usebox{\@aom@linecount}}%
	\end{picture}}
	\fi
	\chead{}%
	\cfoot{\footnotesize\thepage}}%
\doinumber{}
\makeatother

\begin{document}


\date{\today}




\begin{abstract}
We study asymptotically compact nonautonomous dynamical systems given by abstract cocycles in Banach spaces. Our main assumptions are given by a squeezing property in a quadratic cone field (given by a family of indefinite quadratic Lyapunov-like functionals) and asymptotic compactness. Under such conditions it is possible to reconstruct foliations as in the theory of normally hyperbolic manifolds. Our approach allows to unify many ``practical'' theories of local and nonlocal invariant manifolds, such as stable/unstable/center manifolds and inertial manifolds, and, moreover, it often leads to optimal conditions for their existence. 

We give applications for semilinear parabolic equations and neutral delay equations, where the Frequency Theorem is used to construct constant cone fields.


\end{abstract}

\maketitle

\tableofcontents

\section{Introduction}
\subsection{Preliminary historical background and motivation}
The concept of an inertial manifold provides the most desired type of finite-dimensional asymptotic behavior for infinite-dimensional dynamical systems. These manifolds, being finite-dimensional invariant Lipschitz or even differentiable manifolds, attract all trajectories at an exponential rate by trajectories lying on the manifold. Thus, any trajectory has a ``fast part'', which exponentially decays, and a ``slow part'', which corresponds to a trajectory from the inertial manifold. If a differential equation is given, then the dynamics on the inertial manifold can be described by a system of ordinary differential equations (the so-called inertial form). Usually, inertial manifolds are normally hyperbolic that allows us to expect them to be stable under small perturbations of the system coefficients and this may be essential, for example, in numerical simulations.

To the best of our knowledge, the closest (to the concept of an inertial manifold) idea was firstly proposed by E.~Hopf in \cite{Hopf1948} from consideration of several examples arising in hydrodynamics. The term ``inertial manifold'' was introduced in the seminal paper of C.~Foias, G.R.~Sell and R.~Temam \cite{FoiasSellTemam1988} and related to the inertial scale from the Kolmogorov theory of hydrodynamic turbulence.

This work is devoted to the construction of inertial manifolds for nonautonomous dynamical systems given by abstract asymptotically compact cocycles in Banach spaces. Our main assumption is based on the squeezing property w.~r.~t. a family of indefinite quadratic Lyapunov functionals (representing a field of quadratic cones). We will show that many existing ``practical'' theories of invariant manifolds (especially, inertial manifolds) can be embedded into this context. This work is an extended and reworked version of the main results from author's PhD thesis defended\footnote{See \url{https://disser.spbu.ru/zashchita-uchenoj-stepeni-spbgu/550-anikushin-mikhail-mikhajlovich.html}.} at St.~Petersburg University \cite{Anikushin2021Thesis}.

Historically, our basic aim was to unite the frequency-domain approach presented in the works of R.A.~Smith (see further) and the Frequency Theorem of V.A.~Yakubovich and A.L.~Likhtarnikov \cite{Likhtarnikov1977,Likhtarnikov1976}, including its recent developments due to A.V.~Proskurnikov \cite{Proskurnikov2015} and the present author \cite{Anikushin2020FreqDelay,Anikushin2020FreqParab}. More specifically, our interest was concerned with the results of R.A.~Smith on developments of the Poincar\'{e}-Bendixson theory \cite{Smith1994PB2,Smith1994PB1,Smith1992,Smith1987OrbStab,Smith1980}, Poincar\'{e} index theorem \cite{Smith1981IndexTh1,Smith1984PBindex}, Poincar\'{e} isolation theorem \cite{Smith1984IsolatedOrbits} and convergence theorems for periodic equations \cite{Smith1986Massera,Smith1990ConvDelay}. As we will see, his frequency conditions guarantee the existence of two-dimensional inertial (or slow) manifolds that makes his main results geometrically clear, elegant and being applicable to a broader context. As a byproduct, we establish connections with many other works (to be discussed below).

It should be noted that R.A.~Smith used quadratic functionals in his works on ODEs (for example, \cite{Smith1987OrbStab,Smith1986Massera}) and  even in the earliest paper concerned with delay equations \cite{Smith1980}, but he abandoned such an approach in future works for delay and parabolic equations \cite{Smith1990ConvDelay,Smith1992,Smith1994PB1,Smith1994PB2} probably being unable to provide adequate conditions for their existence. He also seemed to be unfamiliar with the Frequency Theorem, which provides such conditions.

\subsection{Dialectics between the Lyapunov-Perron and Hadamard methodologies as a historical perspective}
Before we proceed further, let us mention that there are two known approaches for the construction of invariant manifolds. Namely, the so-called Lyapunov-Perron method, in which one constructs the invariant manifold as a fixed point of an integral operator derived from the variation of constants formula. The other one is the so-called Hadamard method\footnote{Named after J.~Hadamard, who used this method to construct the unstable manifold near a hyperbolic fixed point of a diffeomorphism of the plane \cite{Hadamard1901}.} or the Graph Transform technique, in which one iterates an in some sense admissible submanifold under the dynamics and expect the iterations to converge to the desired invariant manifold. Both methods allow to show that a certain linear structure (semidichotomy, dichotomy or trichotomy) persists under small nonlinear perturbations. Besides the standard constructions of stable/unstable/center manifolds in a neighborhood of a stationary point, these methods are widely used in the hyperbolic theory (see L.~Barreira and Ya.B.~Pesin \cite{BarreiraPesin2007} and references therein) or in the theory of normally hyperbolic invariant manifolds (see P.W.~Bates, K.~Lu and C.~Zeng \cite{BatesLuZeng2000,BatesLuZeng1998} references therein).

However, the application of both methods in the theory of inertial manifolds demands obtaining non-local and in some sense optimal conditions for the persistence of linear structures. At this place (as time has shown, nontrivial) modifications of the methods are required. It turned out that for standard problems, where the class of perturbations is given by nonlinearities with a fixed Lipschitz constant, the Fourier transform plays a crucial role. 

For the Lyapunov-Perron method, the Fourier transform was firstly applied by V.M.~Popov \cite{Popov1961} to obtain conditions for the absolute stability and since then this method has been known for mathematical engineers as the method of \textit{a priori integral estimates}. This method was developed by R.A.Smith in his works on delay and parabolic equations \cite{Smith1994PB2,Smith1994PB1,Smith1992,Smith1990ConvDelay} as an alternative to his approach via quadratic functionals. In the theory of inertial manifolds such an approach is mostly known due to M.~Miklav\v{c}i\v{c} \cite{Miclavcic1991}, who was motivated by the perturbation theory from functional analysis.

On the other hand, the work of V.M.~Popov opened the problem on the equivalence between the method of a priori integral estimates and the method of Lyapunov functionals (for the problems studied by V.M.~Popov such functionals are known as the Lur'e-Postnikov functionals) in the theory of absolute stability. This problem stimulated the appearance of the Frequency Theorem and the theorem on the losslessness of the $S$-procedure proved by V.A.~Yakubovich \cite{Gelig1978}, who thus obtained a positive solution to the problem. In this regard, the papers of R.A.~Smith, where quadratic functionals are used, represent a proper modification of the Hadamard method\footnote{It seems that R.A.~Smith reinvented the Hadamard method in accordance with the famous quote by D.V.~Anosov: ``Every five years or so, if not more often, someone ``discovers'' the theorem of Hadamard and Perron, proving it by Hadamard's method of proof or by Perron's''.} for purposes of inertial manifolds. Although, he was unfamiliar or unable to prove\footnote{Apparently, he wondered about the connections of his works and works on inertial manifolds only in his last work \cite{Smith1994PB2}.} most properties of inertial manifolds (such as the exponential tracking, $C^1$-differentiability, normal hyperbolicity or the existence $C^{1}$-foliations), while only implicitly utilizing some of their consequences, this provide an essential contribution to the theory developed in the present paper.

Thus, one may ask for an equivalence between the Lyapunov-Perron and Hadamard methods in the nonlocal context. For example, are there quadratic functionals under the conditions used by R.A.~Smith \cite{Smith1992,Smith1994PB2} or M.~Miclav\v{c}i\v{c} \cite{Miclavcic1991}, which allow to apply the Hadamard method? From the extensions of the Frequency Theorem presented in our works \cite{Anikushin2020FreqDelay,Anikushin2020FreqParab} and results of the present paper the answer is positive for Smith's frequency inequality and the ``quadratic'' version of M.~Miclav\v{c}i\v{c}'s conditions (see \cite{Anikushin2020FreqParab} for a discussion).

Despite both methods are equivalent at least in the case of perturbations describing by quadratic constraints, deriving optimal conditions from the Lyapunov-Perron method was a challenging task for many authors. For example, even in the case of semilinear parabolic problems with a self-adjoint linear part, first conditions for the existence of inertial manifolds were non-optimal and ineffective as a consequence of rough estimates. This includes the papers of C.~Foias, G.R.~Sell and R.~Temam\cite{FoiasSellTemam1988}, R.~Man\'{e} \cite{Mane1977}, G.R.~Sell and Y.~You \cite{SellYou1992}, R.~Rosa and R.~Temam \cite{RosaTemam1996}, J.C.~Robinson \cite{Robinson1996AsympComl} and many others. It took several years before optimal conditions (known as the \textit{Spectral Gap Condition}) were established by M.~Miclav\v{c}i\v{c} \cite{Miclavcic1991} (via the Lyapunov-Perron method) and A.V.~Romanov \cite{Romanov1994} (via the Hadamard method). See Subsection \ref{SUBSEC: SemilinearParabolicEquations} for further discussions on semilinear parabolic equations. We also refer to the survey of S.~Zelik \cite{Zelik2014} for more discussions in the context finite-dimensional reduction for semilinear parabolic equations.

In the case of delay equations, the list can be extended by the works of Yu.A.~Ryabov \cite{Ryabov1967}, R.D.~Driver \cite{Driver1968SmallDelays}, C.~Chicone \cite{Chicone2003}, S.~Chen and J.~Shen \cite{ChenShen2021}, L.~Boutet~de~Monvel, I.D.~Chueshov and A.V.~Rezounenko \cite{BoutetChueshovRezounenko1998}. See Subsection \ref{SUBSEC: DelayEquations} for further discussions on delay equations. See also our papers \cite{AnikushinNDE2022,Anikushin2020Semigroups,Anikushin2021DiffJ} for more discussions.

In this regard, the Frequency Theorem not only immediately leads to optimal conditions even for non-self-adjoint linear parts, but also states in what sense such conditions are optimal, thanks to the theorem on the losslessliness of $S$-procedure (see \cite{Anikushin2020FreqParab} for a discussion). In addition, this approach also separates geometry and analytics, clearly highlighting the properties of linear and nonlinear problems required for verification. Note that some authors (for example, A.~Kostianko and S.~Zelik \cite{KostiankoZelik2019Kwak}; V.V.~Chepyzhov, A.~Kostianko and S.~Zelik \cite{ChepyzhovKostiankoZelik2019}) understand the optimality in a stronger sense, but, in fact, what they prove is that the frequency inequality is optimal in that stronger sense.

We also note the attempt of N.~Koksch and S.~Siegmund \cite{KokschSiegmund2002} to generalize inertial manifolds theory for nonautonomous dynamical systems by developing the Hadamard method. However, verification of their abstract assumptions even in the case of semilinear parabolic problems with self-adjoint linear parts lead to non-optimal conditions. Moreover, what seems to be more restrictive, is the considered class of non-quadratic cones, which is given by a naive generalization of the self-adjoint case. Since optimal conditions are derived through the Frequency Theorem and quadratic cones, this may possibly explain why their generalization did not succeed in applications. However, it is interesting to obtain a proper generalization of our results for general cones and establish their existence under the non-quadratic conditions of M.~Miclav\v{c}i\v{c} \cite{Miclavcic1991} that will require a ``non-quadratic'' version of the Frequency Theorem.

It is shown in the works of A.~Kostianko and S.~Zelik \cite{KostiankoZelik2015}, A.~Kostianko et. al \cite{KostiankoZelikSA2020} that the \textit{Spatial Averaging Principle}, which was proposed by J.~Mallet-Paret and G.R.~Sell \cite{MalletParetSell1988IM} to construct inertial manifolds for certain scalar parabolic problems in 2D and 3D domains in the case when the Spectral Gap Condition is not satisfied, also lead to a certain squeezing property w.~r.~t. a quadratic functional. In this case the exponent depends on the point and the resulting inertial manifold may be nonuniformly normally hyperbolic (see Subsection \ref{EQ: SPEspatialaveraginge}). Examples of nonuniformly hyperbolic inertial manifolds are constructed by A.V.~Romanov \cite{Romanov2016}.

In \cite{Anikushin2020FreqParab} we posed a question concerned with the connection between the Spatial Averaging Principle and the Frequency Theorem for nonstationary problems. In finite dimensions, R.~Fabbri, R.~Johnson and C. N\'{u}\~{n}ez \cite{Fabbri2003FreqTh} obtained a generalization of the Frequency Theorem for nonstationary problems, which can be used to construct quadratic cone fields. For such systems an analog of the frequency inequality is that a certain Hamiltonian system associated with the optimization problem studied in the Frequency Theorem admits an exponential dichotomy and a nonoscillation condition. For dissipative systems, the theorem can be used to obtain the existence of proper spectral gaps in the Sacker-Sell spectrum of the linearization cocycle over the global attractor. Although, the existence of a proper gap is not sufficient to embed the global attractor into a locally invariant manifold and there exist topological obstacles concerned with nontrivial intersections of strongly stable manifolds (see C.~Bonatti and S.~Crovisier \cite{BonattiCrovisier2016}). We shall develop with this approach in more details somewhere.

\subsection{Specificity of our approach}

As we just noted, the general approach for the construction of inertial manifolds as in the work of C.~Bonatti and S.~Crovisier \cite{BonattiCrovisier2016}, which decomposes the problem into the studying of spectral properties and the elimination of topological obstacles, is still awaiting developments that will make it applicable to differential equations. Moreover, such an approach does not allow to include some of the results discussed above (see also Subsection \ref{SEC: AbstractApplic}) since they are highly based on the trivial topology of the inertial manifold constructed as a Lipschitz graph over a spectral subspace. In addition, we find it more promising for studying concrete models to combine constant cone fields constructed by the Frequency Theorem with a proper change of variables simplifying the geometry.

For our purposes, it is sufficient to work directly with a nonlinear cocycle in a given Banach space. This significantly simplifies the proofs compared with the ones used in the theory of normally hyperbolic manifolds \cite{BatesLuZeng2000,BatesLuZeng1998}. Namely, our main aim is to construct what we call horizontal and vertical leaves, which are just the equivalence classes corresponding to the ordering of trajectories w.~r.~t. a quadratic cone field. The fact that such an ordering is equivalent to an asymptotic relation between the trajectories follows almost immediately from the squeezing property in the quadratic cone field. We also note that, despite it was known since the beginning that inertial manifolds are normally hyperbolic, almost no one tried to construct the foliations around it. We only mention here the work of R.~Rosa and R.~Temam \cite{RosaTemam1996} as a rare exception. It is worth noting that it is also more convenient to study the robustness of inertial manifolds w.~r.~t. perturbations under the method of construction. Thus, our results become independent from the abstract theories of normally hyperbolic manifolds. This is especially important considering the fact that we do not know such theories for abstract nonuniformly hyperbolic and noncompact invariant manifolds in Banach spaces. We restrict ourselves by proving only the $C^{1}$-differentiability of leaves, which is sufficient for the purposes of inertial manifolds (and usually, this is the maximum that we can achieve).

This paper is organized as follows.

In Section \ref{SUBSEC: DefinitionsAndHypotheses} we expound and discuss our main assumptions. Namely, we give definitions of semicocycles, cocycles and associated skew-product semiflows (Subsection \ref{SUBSE: DefinitionCocycle}); the definition of the squeezing property in a quadratic cone field (Subsection \ref{SUBSEC: SqueezingQuadraticConeFields}); the definition of asymptotic compactness and Lipschitzity (Subsection \ref{SUBSEC: AsymptoticCompactnessLipschitzity}); and we add few remarks concerned with the introduced notions (Subsection \ref{SUBSEC: Motivation}). 

Section \ref{SEC: ConstructionIM} contains our main results. Namely, the construction of horizontal leaves (Subsection \ref{SUBSEC: HorizontalLeaves}), vertical leaves (Subsection \ref{SUBSEC: VerticalFoliation}), their $C^{1}$-differentiability (Subsection \ref{SUBSEC: DifferentiabilityLeaves}) and Subsection \ref{SUBSEC: AdmissibleProjectors} is devoted to remarks concerned with admissible projectors. 

In Section \ref{SEC: InertialManifolds} the developed theory is applied to construct inertial (Subsection \ref{SUBSEC: AmenableSets}) and slow (Subsection \ref{SUBSEC: SlowManifolds}) manifolds. We establish the exponential tracking (Subsection \ref{SUBSEC: ExponentialTracking}), the normal hyperbolicity (Subsection \ref{SUBSEC: NormalHyperbolicity}), invertibility and stability properties of the reduced system (Subsection \ref{SUBSEC: StabilityReducedDynamics}) and study robustness of inertial manifolds (Subsection \ref{SUBSEC: RobustnessManifolds}). 

Section \ref{SEC: AbstractApplic} contains applications for the problems concerned with low-dimensional dynamics. Namely, for the problems of exponential stability and almost automorphy in almost periodic cocycles (Subsection \ref{SUBSEC: ExpStabAndAAAP}), the Lyapunov stability convergence theorems in periodic cocycles (Subsection \ref{SUBSEC: ConvergencePerCoc}), the Poincar\'{e}-Bendixson theory (Subsection \ref{SUBSEC: PoincareBendixsonTheory}). 

In Section \ref{SEC: Applications} we give applications for ODEs (Subsection \ref{SUBSEC: ODEs}), neutral delay equations (Subsection \ref{SUBSEC: DelayEquations}), semilinear parabolic equations (Subsection \ref{SUBSEC: SemilinearParabolicEquations}).
\section{Preliminaries: definitions, main hypotheses and auxiliary lemmas}
\label{SUBSEC: DefinitionsAndHypotheses}

\subsection{Semicocycles, cocycles and skew-product semiflows}
\label{SUBSE: DefinitionCocycle}
Consider a complete metric space $\mathcal{X}$. Let $\mathbb{T}$ be either $\mathbb{R}_{+}$ or $\mathbb{R}$ and consider a one-parameter family of maps $\varphi^{t} \colon \mathcal{X} \to \mathcal{X}$, where $t \in \mathbb{T}$, such that
\begin{enumerate}
	\item[1.] $\varphi^{0}(x) = x$ and $\varphi^{t+s}(x)=\varphi^{t}(\varphi^{s}(x))$ for all $t,s \in \mathbb{T}$ and $x \in \mathcal{X}$.
	\item[2.] The map $(t,x) \mapsto \varphi^{t}(x)$ is continuous as a map from $\mathbb{T} \times \mathcal{X}$ to $\mathcal{X}$.
\end{enumerate}
For the sake of brevity, we often denote the family $\{\varphi^{t}\}_{t \in \mathbb{T}}$ by $\varphi$ or $(\mathcal{X},\varphi)$. In the case $\mathbb{T} = \mathbb{R}_{+}$ (resp. $\mathbb{T} = \mathbb{R}$) we say that $\varphi$ is a \textit{semiflow} (resp. \textit{flow}) on $\mathcal{X}$.

Now let $\mathcal{Q}$ be another complete metric space and consider a flow $\vartheta$ on $\mathcal{Q}$. Let $\mathbb{E}$ be a Banach space over $\mathbb{R}$ or $\mathbb{C}$. A given two-parameter family of mappings $\psi^{t}(q,\cdot) \colon \mathbb{E} \to \mathbb{E}$, where $q \in \mathcal{Q}$ and $t \geq 0$, is called a \textit{cocycle} in $\mathbb{E}$ over the driving system $(\mathcal{Q},\vartheta)$ if the following properties are satisfied:
\begin{enumerate}
	\item[1.] $\psi^{0}(q,v)=v$ and $\psi^{t+s}(q,v)=\psi^{t}(\vartheta^{s}(q),\psi^{s}(q,v))$ for all $t,s\geq 0$, $q \in \mathcal{Q}$ and $v \in \mathbb{E}$.
	\item[2.] The map $(t,q,v) \mapsto \psi^{t}(q,v)$ is continuous as a map from $\mathbb{R}_{+} \times \mathcal{Q} \times \mathbb{E}$ to $\mathbb{E}$.
\end{enumerate}
Similarly, the cocycle will be denoted by $(\psi,\vartheta)$ or just by $\psi$. If we let $\vartheta$ from the previous definition to be only a semiflow, then $\psi$ will be called a \textit{semicocycle} in $\mathbb{E}$ over $(\mathcal{Q},\vartheta)$.

A huge advantage of the cocycle is the possibility to define complete trajectories, which are essential for our construction of inertial manifolds. Namely, a continuous function $v(\cdot) \colon \mathbb{R} \to \mathbb{E}$ is a \textit{complete trajectory} over $q \in \mathcal{Q}$ if $v(t+s)=\psi^{t}(\vartheta^{s}(q),v(s))$ for all $t \geq 0$ and $s \in \mathbb{R}$. In this case we also say that $v(\cdot)$ is \textit{passing through $v(0)$} over $q$. For results below the following assumption will be used.
\begin{description}
	\item[\textbf{(BA)}] For any $q \in \mathcal{Q}$ there is a bounded in the past complete trajectory $w^{*}_{q}(\cdot)$ over $q$ and there exists a constant $M_{b}>0$ such that 
	\begin{equation}
		\label{EQ: BAUniformBound}
		\sup_{q \in \mathcal{Q}}\sup_{t \leq 0}\|w^{*}_{q}(t)\|_{\mathbb{E}} =: M_{b} < +\infty.
	\end{equation}
\end{description}

With each semicocycle we can associate a semiflow called the \textit{skew-product semiflow} $\pi^{t} \colon \mathcal{Q} \times \mathbb{E} \to \mathcal{Q} \times \mathbb{E}$ given by $\pi^{t}(q,v) = (\vartheta^{t}(q), \psi^{t}(q,v))$ for all $t \geq 0$, $q \in \mathcal{Q}$ and $v \in \mathbb{E}$.

\subsection{Squeezing in quadratic cone fields and orthogonal projectors}
\label{SUBSEC: SqueezingQuadraticConeFields}

Let $\langle v, f \rangle$ be the dual pairing between $v \in \mathbb{E}$ and $f \in \mathbb{E}^{*}$. We call an operator\footnote{Throughout this paper by $\mathcal{L}(\mathbb{F}_{1};\mathbb{F}_{2})$ we denote the space of bounded linear operators between Banach spaces $\mathbb{F}_{1}$ and $\mathbb{F_{2}}$, which is endowed with the operator norm. If $\mathbb{F}_{1} = \mathbb{F}_{2}$, we usually write $\mathcal{L}(\mathbb{F}_{1})$.} $P \in \mathcal{L}(\mathbb{E};\mathbb{E}^{*})$ \textit{symmetric} if $\langle v_{1}, Pv_{2} \rangle = \overline{\langle v_{2}, Pv_{1} \rangle}$ for all $v_{1},v_{2} \in \mathbb{E}$. For a linear subspace $\mathbb{L} \subset \mathbb{E}$ we say that $P$ is \textit{positive} (resp. \textit{negative}) on $\mathbb{L}$ if $\langle v, P v \rangle > 0$ (resp. $\langle v, Pv \rangle < 0$) for all non-zero vectors $v \in \mathbb{L}$.

\begin{description}
	\item[\textbf{(H1)}] For every $q \in \mathcal{Q}$ there exists a symmetric operator $P(q) \in \mathcal{L}(\mathbb{E};\mathbb{E}^{*})$ and a splitting of $\mathbb{E}$ into a direct sum as $\mathbb{E} = \mathbb{E}^{+}(q) \oplus \mathbb{E}^{-}(q)$ such that $P(q)$ is positive on $\mathbb{E}^{+}(q)$ and negative on $\mathbb{E}^{-}(q)$. Moreover, the norms of $P(q)$ in $\mathcal{L}(\mathbb{E};\mathbb{E}^{*})$ are uniformly bounded
	\begin{equation}
		\label{EQ: OperatorsBound}
		\sup_{q \in \mathcal{Q}} \|P(q)\| = M_{P} < +\infty.
	\end{equation}
	\item[\textbf{(H2)}] For some integer $j \geq 0$ we have $\operatorname{dim}\mathbb{E}^{-}(q)=j$ for all $q \in \mathcal{Q}$.
\end{description}

Under \textbf{(H1)} we define the quadratic forms $V_{q}(v):=\langle v, P(q)v \rangle$ for $q \in \mathcal{Q}$ and $v \in \mathbb{E}$.

\begin{remark}
	\label{REM: Orthogonality}
	Consider an operator $P \in \mathcal{L}(\mathbb{E};\mathbb{E}^{*})$ and suppose that $\mathbb{E}$ splits into a direct sum $\mathbb{E} = \mathbb{E}^{+} \oplus \mathbb{E}^{-}$ such that $P$ is positive on $\mathbb{E}^{+}$ and negative on $\mathbb{E}^{-}$. Suppose 
	$\mathbb{E}^{-}$ is finite-dimensional. Let $V(v):=\langle v,Pv \rangle$ be the quadratic form of $P$. We define the $V$-orthogonal projector $\Pi \colon \mathbb{E} \to \mathbb{E}^{-}$ as follows. Since $P$ is symmetric and negative on $\mathbb{E}^{-}$ the bilinear form $(v_{1},v_{2}) \mapsto -\langle v_{1}, v_{2} \rangle$ is an inner product on $\mathbb{E}^{-}$. Let $v \in \mathbb{E}$ then $\langle \cdot, Pv \rangle$ is a linear functional on $\mathbb{E}^{-}$ and, by the Riesz representation theorem, there is a unique element $\Pi v \in \mathbb{E}^{-}$ such that $\langle \cdot, Pv \rangle = \langle \cdot, P\Pi v \rangle$. From this and since the norm $\sqrt{-V(\cdot)}$ is equivalent to $\| \cdot \|_{\mathbb{E}}$ on $\mathbb{E}^{-}$, we have that $\Pi$ is bounded. Moreover, $\mathbb{E} = \operatorname{Ker}\Pi \oplus \mathbb{E}^{-}$ and $V(v)=V(v^{+}) + V(v^{-})$, where $v=v^{+} + v^{-}$ is the unique decomposition with $v^{+} \in \operatorname{Ker}\Pi$ and $v^{-} \in \mathbb{E}^{-}$. It is also clear that $P$ is positive on $\operatorname{Ker}\Pi$. Thus, under \textbf{(H1)} and \textbf{(H2)} we can always construct new subspaces $\mathbb{E}^{+}(q)$ and $\mathbb{E}^{-}(q)$, which are $V_{q}$-orthogonal (in the given sense) and the corresponding $V_{q}$-orthogonal projector $\Pi_{q}$ is bounded. For further investigations we need an estimate for the norm $\| \Pi_{q}\|$ as follows. Let us consider the constants
	\begin{equation}
		\label{EQ: ConstantProjectorNorm}
		C_{q} = C_{q}(P)  := \inf_{\|\zeta\| = 1, \zeta \in \mathbb{E}^{-}(q)} \left(-V_{q}(\zeta)\right)^{1/2}.
	\end{equation}
    Then it is not hard to see that $\| \Pi_{q}\| \leq C^{-2}_{q}\|P(q)\|$.
\end{remark}

With the above remark we will also use the following assumption, which will be essential for our proof of the exponential tracking property.
\begin{description}
	\item[\textbf{(PROJ)}] The norms of $V_{q}$-orthogonal projectors $\Pi_{q}$ are uniformly bounded
	\begin{equation}
		\sup_{q \in \mathcal{Q}} \| \Pi_{q} \| =: M_{\Pi} < +\infty.
	\end{equation}
\end{description}

The following assumption will be used to show that the fibers of inertial manifolds depend continuously on $q \in \mathcal{Q}$.
\begin{description}
	\item[\textbf{(CD)}] The operators $P(q)$ from \textbf{(H1)} depend continuously on $q \in \mathcal{Q}$ in the norm of $\mathcal{L}(\mathbb{E};\mathbb{E}^{*})$ and the projectors $\Pi^{d}_{q} \colon \mathbb{E} \to \mathbb{E}^{-}(q)$ w.~r.~t. the decomposition $\mathbb{E} = \mathbb{E}^{+}(q) \oplus \mathbb{E}^{-}(q)$ from \textbf{(H1)} depend continuously on $q \in \mathcal{Q}$ in the norm of $\mathcal{L}(\mathbb{E})$.
\end{description}
Note that within \textbf{(CD}) we do not assume that the projectors $\Pi^{d}_{q}$ are $V_{q}$-orthogonal. In practice, the spaces $\mathbb{E}^{+}(q)$ and $\mathbb{E}^{-}(q)$ are obtained from the exponential dichotomy of some linear problem. Thus we refer to $\Pi^{d}_{q}$ as the \textit{dichotomy projector} over $q$.

Let us show that \textbf{(CD)} implies that the $V_{q}$-orthogonal projectors $\Pi_{q}$ (constructed in Remark \ref{REM: Orthogonality}) are also depend continuously on $q \in \mathcal{Q}$. This is contained in the following lemma.

\begin{lemma}
	\label{LEM: VOrtogonalProjectorsContinuity}
	Under \textbf{(H1)} let \textbf{(CD)} be satisfied. Then the $V_{q}$-orthogonal projectors $\Pi_{q}$  depend continuously on $q \in \mathcal{Q}$ in the norm of $\mathcal{L}(\mathbb{E})$.
\end{lemma}
\begin{proof}
	Recall that $\Pi_{q} \colon \mathbb{E} \to \mathbb{E}^{-}(q)$ is defined by the relation $\langle w, P(q)v \rangle = \langle w, P(q)\Pi_{q} v \rangle$ for all $w \in \mathbb{E}^{-}(q)$ and $v \in \mathbb{E}$. Suppose we are given with a sequence $q_{k} \in \mathcal{Q}$, where $k=1,2,\ldots$, converging to some $q_{0} \in \mathcal{Q}$. 
	
	Let us show that for our purposes it is sufficient to show that $V_{q_{0}}(\Pi_{q_{k}}w - \Pi_{q_{0}}w) \to 0$ uniformly in $w \in \mathbb{E}$ such that $\|w\|=1$. Since $P(q)$ continuously depend on $q$, the map $\mathcal{Q} \times \mathbb{E} \ni (q,v) \mapsto V_{q}(v) \in \mathbb{R}$ is continuous. The set $\mathcal{K} \subset \mathcal{Q}\times \mathbb{E}$ consisting of all $(q,w)$ such that $q \in \{ q_{1},q_{2},\dots \} \cup \{ q_{0} \}$ and $w \in \mathbb{E}^{-}(q)$ with $\|w\|_{\mathbb{E}}=1$ is compact since $\mathbb{E}^{-}(q)$ depend continuously on $q$ in the sense of \textbf{(CD)}. Thus, there exists $C_{V}>0$ such that
	\begin{equation}
		\label{EQ: VorthogonalEnergyAsUnifNorm}
		-V_{q}(v) \geq C_{V} \cdot \| v \|^{2}_{\mathbb{E}}
	\end{equation} 
	for all $q \in \{ q_{1},q_{2},\dots \} \cup \{ q_{0} \}$ and $v \in \mathbb{E}^{-}(q)$. From \eqref{EQ: VorthogonalEnergyAsUnifNorm} it is also follows that the norms of $\Pi_{q}$ are bounded uniformly in $q \in \{ q_{1},q_{2},\dots \} \cup \{ q_{0} \}$ since $\| \Pi_{q}\| \leq C^{-2}_{V}\|P(q)\|$ (see \eqref{EQ: ConstantProjectorNorm}). In particular, for all $w \in \mathbb{E}$ with $\|w\|_{\mathbb{E}}=1$
	\begin{equation}
		\label{EQ: PiOrtProjectorEstimate}
		-V_{q_{0}}(\Pi^{d}_{q_{0}}(\Pi_{q_{k}}w - \Pi_{q_{0}}w)) \geq C_{V} \cdot \| \Pi^{d}_{q_{0}}(\Pi_{q_{k}}w - \Pi_{q_{0}}w )\|^{2}_{\mathbb{E}}.
	\end{equation}
	Since $\Pi_{q_{k}}w = \Pi^{d}_{q_{k}} \Pi_{q_{k}} w$, we have, thanks to \textbf{(CD)}, $\Pi^{d}_{q_{0}}(\Pi_{q_{k}}w - \Pi_{q_{0}}w) = \Pi_{q_{k}}w - \Pi_{q_{0}}w + o(1)$. From this and \eqref{EQ: PiOrtProjectorEstimate} it follows that if $V_{q_{0}}(\Pi_{q_{k}}w - \Pi_{q_{0}}w) \to 0$ then $\|\Pi_{q_{k}}w - \Pi_{q_{0}}w \|_{\mathbb{E}} \to 0$ (uniformly in the considered $w$).
	
	Now we have
	\begin{equation}
		\begin{split}
			V_{q_{0}}(\Pi_{q_{k}}w - \Pi_{q_{0}}w) &= \langle \Pi_{q_{k}} w, P(q_{0})\Pi_{q_{k}} w \rangle - \langle \Pi_{q_{k}} w, P(q_{0}) \Pi_{q_{0}}w \rangle - \langle \Pi_{q_{0}} w, P(q_{0}) \Pi_{q_{k}}w \rangle \\ &+ \langle \Pi_{q_{0}}w,P(q_{0})\Pi_{q_{0}}w\rangle.
		\end{split}    	
	\end{equation}
	From \textbf{(CD)}, the definition of $\Pi_{q}$ and the uniform boundedness of the norms of $\Pi_{q}$ we have the following relations as $k \to +\infty$, which are uniform in $w \in \mathbb{E}$ with $\|w\|_{\mathbb{E}}=1$:
	\begin{equation}
		\begin{split}
			\langle \Pi_{q_{k}} w, P(q_{0})\Pi_{q_{k}} w \rangle &= \langle \Pi_{q_{k}}w, P(q_{0})w  \rangle + o(1),\\
			\langle \Pi_{q_{k}} w, P(q_{0}) \Pi_{q_{0}}w \rangle &= \langle \Pi_{q_{k}}w, P(q_{0}) w \rangle + o(1),\\
			\langle \Pi_{q_{0}} w, P(q_{0}) \Pi_{q_{k}}w \rangle &= \langle \Pi_{q_{0}} w, P(q_{0})w \rangle + o(1),\\
			\langle \Pi_{q_{0}}w,P(q_{0})\Pi_{q_{0}}w\rangle &= \langle \Pi_{q_{0}}w,P(q_{0})w\rangle.
		\end{split}
	\end{equation}
	Thus, $V_{q_{0}}(\Pi_{q_{k}}w - \Pi_{q_{0}}w) = o(1)$ that shows the required continuity.
\end{proof}

An immediate corollary of Lemma \ref{LEM: VOrtogonalProjectorsContinuity} is the following.
\begin{corollary}
	\label{COR: OrtProjectosBounded}
	Let \textbf{(CD)} be satisfied and let $\mathcal{Q}$ be compact. Then we have \textbf{(PROJ)} satisfied and the constants $C_{q}$ from \eqref{EQ: ConstantProjectorNorm} are uniformly (in $q \in \mathcal{Q}$) bounded from above and below.
\end{corollary}

Let $\mathbb{E}$ be continuously embedded into some Hilbert space $\mathbb{H}$. We identify elements of $\mathbb{E}$ and $\mathbb{H}$ under such an embedding. 
\begin{remark}
	It is not essential for the abstract part of the theory from Section \ref{SEC: ConstructionIM} that $\mathbb{H}$ is a Hilbert space, but it will be convenient for discussions and applications to think so (since in all known situations $\mathbb{H}$ is a Hilbert space). The weaker norm $|\cdot|_{\mathbb{H}}$ and the Hilbert space structure are connected with the construction of quadratic functionals via the Frequency Theorem \cite{Anikushin2020FreqDelay,Anikushin2020FreqParab}.
\end{remark}

Our main assumption is the following squeezing property w.~r.~t. the family of quadratic forms $V_{q}$.
\begin{description}
	\item[\textbf{(H3)}] There exist constants $\nu^{-} < \nu^{+}, \delta_{P}>0$ and $\tau_{P} \geq 0$ and a Borel measurable function $\nu_{0}(\cdot) \colon \mathcal{Q} \to \mathbb{R}_{+}$ such that $\nu^{-} \leq \nu_{0}(\cdot) \leq \nu^{+}$ and for $\nu(t,q) := \int_{0}^{t} \nu_{0}(\vartheta^{s}(q)) ds$ we have the inequality
	\begin{equation}
		\label{EQ: PropH3inequality}
		\begin{split}
		e^{2 \nu(r, q)} V_{\vartheta^{r}(q)}(\psi^{r}(q,v_{1})-\psi^{r}(q,v_{2})) &- e^{2 \nu(l,q)}V_{\vartheta^{l}(q)}(\psi^{l}(q,v_{1})-\psi^{l}(q,v_{2})) \leq \\ &\leq -\delta_{P} \int_{l}^{r} e^{2 \nu(s,q)} |\psi^{s}(q,v_{1})-\psi^{s}(q,v_{2})|^{2}_{\mathbb{H}}ds	
		\end{split}	
	\end{equation}
    satisfied for all $q \in \mathcal{Q}$, $v_{1},v_{2} \in \mathbb{E}$ and $r, l \geq 0$ such that $r-l \geq \tau_{P}$.
\end{description}
To construct inertial manifolds one restricts the exponent sign as $\nu^{-} > 0$. However, this is not essential for many other results.

\begin{remark}
	Let the operators $P(q)$ and the spaces $\mathbb{E}^{+}(q)$ and $\mathbb{E}^{-}(q)$ be independent of $q$. So, $P(q)=P$ and $V_{q}=V$. For specialists more recognizable are the differential versions of \textbf{(H3)} given by
	\begin{equation}
		\label{EQ: ExampleDiffFormSmith}
		\frac{d}{ds}\left[ e^{2 \nu(t,q) s} V\left(\psi^{s}(q, v_{1}) - \psi^{s}(q,v_{2}) \right) \right] \leq -\delta_{P} e^{2\nu(t,q) s} |\psi^{s}(q, v_{1}) - \psi^{s}(q,v_{2})|^{2}_{\mathbb{H}}.
	\end{equation}
    or, after using the Leibniz rule,
    \begin{equation}
    	\label{EQ: ExampleDiffFormConeConstant}
    	\begin{split}
             \frac{d}{ds} V(\psi^{s}(q, v_{1}) - \psi^{s}(q,v_{2}) ) + 2\nu_{0}(\vartheta^{s}(q)) V(\psi^{s}(q, v_{1}) - \psi^{t}(q,v_{2}) )  \leq \\ \leq -\delta_{P} |\psi^{s}(q, v_{1}) - \psi^{s}(q,v_{2})|^{2}_{\mathbb{H}}.
    	\end{split}
    \end{equation}
    For the case of semilinear parabolic equations, where $\mathbb{E}=\mathbb{H}$, and for special quadratic forms, a similar to \eqref{EQ: ExampleDiffFormConeConstant} inequality posed for the semicocycle generated by the linearized (variational) system is known as the \textit{strong cone property in a differential form} introduced by A.~Kostianko and S.~Zelik \cite{KostiankoZelikSA2020,KostiankoZelik2015}. For the case of ODEs and with constant $\nu_{0}$ this condition was used by R.A.~Smith \cite{Smith1986Massera,Smith1987OrbStab}.
\end{remark}


Let $v \colon [T,+\infty) \to \mathbb{R}$, where $T \in \mathbb{R}$, be a continuous function. We say that $v(\cdot)$ is a \textit{trajectory} of the cocycle over $q$ defined for $t \geq T$ if $v(t+s)=\psi^{t}(\vartheta^{s}(q),v(s))$ holds for all $t \geq 0$ and $s \geq T$. Let $v_{1}(\cdot)$ and $v_{2}(\cdot)$ be two trajectories over $q$ defined for $t \geq T$. Then under \textbf{(H3)} for all $r - l \geq \tau_{P}$ and $l \geq T$ we have
\begin{equation}
	\label{EQ: H3trajectoryform}
	\begin{split}
		e^{2\nu(r,q)}V_{\vartheta^{r}(q)}(v_{1}(r)-v_{2}(r)) - e^{2\nu(l,q)}V_{\vartheta^{l}(q)}(v_{1}(l)-v_{2}(l)) \leq \\ \leq -\delta \int_{l}^{r}e^{2\nu(s,q)} |v_{1}(s)-v_{2}(s)|^{2}_{\mathbb{H}}ds.
	\end{split}
\end{equation}
We usually use \textbf{(H3)} and similar to it conditions in this form.
\begin{remark}
	\label{REM: TrajectoriesOfSemicocycles}
	It will be also convenient (especially in Subsection \ref{SUBSEC: VerticalFoliation} and further references to it) to consider trajectories over $q$ defined for $t \geq T$ with $T < 0$ in the case when $(\mathcal{Q},\vartheta)$ is a semiflow. This is possible if there exists (not necessarily unique) a backward extension of $q$ up to time $T$, i.~e. a point $q_{0} \in \vartheta^{-T}(q)$. Then we can speak of a trajectory over $q$ defined for $t \geq T$ having in mind that a backward extension for $q$, say $\hat{q}(t)=\vartheta^{t-T}(q_{0})$, is fixed. It is also convenient to denote it by $\vartheta^{s}(q) := \hat{q}(s)$ for $s \geq T$. In such a notation \eqref{EQ: H3trajectoryform} still holds.
\end{remark}

\begin{remark}
	Further we will relax \textbf{(H3)} a bit to include the Spatial Averaging Principle (see Subsection \ref{SUBSEC: SemilinearParabolicEquations}) into our general theory.
\end{remark}

\subsection{Asymptotic compactness and uniform Lipschitzity}
\label{SUBSEC: AsymptoticCompactnessLipschitzity}

Recall that the Kuratowski measure of noncompactness in $\mathbb{E}$ is a nonnegative function $\alpha_{K}(\cdot)$ defined on bounded subsets $\mathcal{B}$ in $\mathbb{E}$, which is given by
\begin{equation}
	\alpha_{K}(\mathcal{B}) := \inf\{ r > 0 \ | \ \text{there is a finite cover of } \mathcal{B} \text{ with diameter } \leq r  \}.
\end{equation}
It is clear that $\alpha(\mathcal{B}) \leq \operatorname{diam}\mathcal{B}$, where $\operatorname{diam}\mathcal{B}$ denotes the diameter of $\mathcal{B}$. We also recall some of standard properties, which are easy to check (see I.~Chueshov \cite{Chueshov2015}; J.~K.~Hale \cite{Hale2000}):
\begin{description}
	\item[\textbf{(MNC1)}] $\alpha_{K}(\mathcal{B}) = 0$ if and only if $\mathcal{B}$ is precompact.
	\item[\textbf{(MNC2)}] $\alpha_{K}(\mathcal{B}_{1} \cup \mathcal{B}_{2}) \leq \operatorname{max}\{\alpha_{K}(\mathcal{B}_{1}), \alpha_{K}(\mathcal{B}_{2})\}$,
	\item[\textbf{(MNC3)}] $\alpha_{K}(\mathcal{B}_{1} + \mathcal{B}_{2}) \leq \alpha_{K}(\mathcal{B}_{1}) + \alpha_{K}(\mathcal{B}_{2})$,
	\item[\textbf{(MNC4)}] $\alpha_{K}(\mathcal{B}) \leq \operatorname{diam}\mathcal{B}$.
\end{description}
Property \textbf{(MNC1)} is a reformulation of the well-known Hausdorff criterion. The above properties of $\alpha_{K}(\cdot)$ are the only required for further considerations, so one can replace the Kuratowski measure of noncompactness by an arbitrary measure of noncompactness, which satisfies these four properties.

Our second main assumption is the following \textit{asymptotic compactness} property.
\begin{description}
	\item[\textbf{(ACOM)}] There exists a constant $\gamma^{+}$ such that for every compact set $\mathcal{C} \subset \mathcal{Q}$ there is a function $\gamma_{0} \colon \mathbb{R}_{+} \to \mathbb{R}_{+}$ such that $\gamma_{0}(t) \to 0$ as $t \to +\infty$ and for all $t \geq 0$ we have
	\begin{equation}
		\label{EQ: AsymptoticCompactnessEstimate}
		\alpha_{K}\left( \psi^{t}(\mathcal{C},\mathcal{B}) \right) \leq \gamma_{0}(t) e^{-\gamma^{+} t} \alpha_{K}(\mathcal{B})
	\end{equation}
    for any bounded set $\mathcal{B} \subset \mathbb{E}$.
\end{description}
In contrast to some known results, which utilize only the decay of the Kuratowski measure, it turns out that for our purposes it is essential to emphasize the order ($\gamma^{+} = \nu^{+}$) of the decay. In the following remark it is shown that this assumption is ``almost necessary'' for the existence of inertial manifolds.
\begin{remark}
	Let us for simplicity discuss the case $\nu_{0}(q) = \nu_{0} = \operatorname{const} > 0$. Below we will show that there exists a continuous nonlinear map $\Pi^{c}_{q} \colon \mathbb{E} \to \mathbb{E}$ such that $S(t,q,v) := \psi^{t}(q,v) - \psi^{t}(q,\Pi^{c}_{q}(v))$ satisfies $\| S(t,q,v) \|_{\mathbb{E}}= O(e^{-\nu_{0} t})$ as $t \to +\infty$ uniformly in $q$ from compact sets and $v$ from bounded subsets. Thus, we have the decomposition
	\begin{equation}
		\label{EQ: RemAsymptoticComp1}
		\psi^{t}(q,v) = S(t,q,v) + \psi^{t}(q,\Pi^{c}_{q}(v)).
	\end{equation}
    The image of $\Pi^{c}_{q}$ is the inertial manifold fibre over $q$, which is a finite-dimensional manifold. Thus, the second term on the right-hand side of \eqref{EQ: RemAsymptoticComp1} is a compact map. From this, \textbf{(MNC3)} and \textbf{(MNC4)} we have that
    \begin{equation}
    	\label{EQ: CocycleDecomposition}
    	\alpha_{K}\left( \psi^{t}(q,\mathcal{B}) \right) = O(e^{-\nu_{0} t}) \text{ as } t \to +\infty
    \end{equation}
    uniformly in $q \in \mathcal{Q}$ from compact subsets and any $\mathcal{B}$ lying in a ball of fixed radius. In applications, we can always vary $\nu_{0}$ a bit (and this will be essential for the normal hyperbolicity property), so the exponent $\nu_{0}$ can be changed to $\nu_{0} + \varepsilon$ for a sufficiently small $\varepsilon>0$. So, the decomposition in \eqref{EQ: RemAsymptoticComp1} ``almost implies'' \textbf{(ACOM)}. Note also that for linear cocycles it is necessary in the exact sense.
\end{remark}
\begin{remark}
    In applications, the true squeezing exponent for the Kuratowski measure is usually larger than $\nu_{0}$. For example, for damped hyperbolic equations given by
    \begin{equation}
    	\label{EQ: ExampleAcomHyperbolic}
    	\partial^{2}_{t} v + \varkappa \partial_{t}v - \Delta v + f(v) = h
    \end{equation}
    in appropriate spaces and with proper assumptions on the functions $f$ and $h$, it can be taken arbitrarily close to the friction $\varkappa > 0$. Although, we usually take $\nu_{0} \ll \varkappa$ to increase the ``spectral gap'' that allows to construct inertial manifolds for a wider class of nonlinearities $f$. Note that to show \textbf{(ACOM)} is satisfied for \eqref{EQ: ExampleAcomHyperbolic}, it is used a similar to \eqref{EQ: CocycleDecomposition} decomposition, where the role of the squeezing operators $S(t,q,v)$ is played by the infinite-dimensional part of the linear semigroup generated by \eqref{EQ: ExampleAcomHyperbolic} with $f=0$ and $h=0$. Analogous results can be obtained for neutral delay equations \cite{AnikushinNDE2022,Hale1977}.
\end{remark}

An important case of \textbf{(ACOM)} is given by the (eventual) \textit{uniform compactness} of the cocycle as in the following assumption.
\begin{description}
	\item[\textbf{(UCOM)}] There exists $\tau_{ucom}>0$ such that the set $\psi^{\tau_{ucom}}(\mathcal{P},\mathcal{B})$ is precompact (in $\mathbb{E}$) for every precompact subset $\mathcal{P} \subset \mathcal{Q}$ and bounded subset $\mathcal{B} \subset \mathbb{E}$.
\end{description}
This is the case encountered in the study of delay equations and parabolic equations in bounded domains.

We use \textbf{(ACOM)} mainly in the situation covered by the following lemma.
\begin{lemma}
	\label{LEM: GeneralCompactnessLemma}
	Let \textbf{(ACOM)} be satisfied. Suppose that $v_{k}(\cdot)$, where $k=1,2,\ldots$, is a sequence of trajectories over $q_{k} \in \mathcal{Q}$, which are defined for $t \geq T_{k}$ with $T_{k} \to -\infty$. Let $q_{k}$ converge to some $q$ as $k \to +\infty$. Suppose also that there are a sequence $t_{l}$ and a family of bounded sets $\mathcal{B}_{l}$ in $\mathbb{E}$, where $l=1,2,\ldots$, such that
	\begin{enumerate}
		\item[1)] $t_{l} \to -\infty$ as $l \to +\infty$.
		\item[2)] $v_{k}(t_{l}) \in \mathcal{B}_{l}$ for all sufficiently large $k$ (depending on $l$ such that $T_{k} \leq t_{l}$) and $\operatorname{diam}\mathcal{B}_{l} \leq D e^{-\gamma^{+} t_{l}}$ for some constant $D$ (independent of $l$).
	\end{enumerate}
	Then there exists a complete trajectory $v^{*}(\cdot)$ over $q$ and a subsequence $v_{k_{m}}(\cdot)$, where $m=1,2,\ldots$, such that $v_{k_{m}}(t) \to v^{*}(t)$ as $m \to +\infty$ for every $t \in \mathbb{R}$.
\end{lemma}
\begin{proof}
	Let $l$ be fixed and let $N_{1}$ be such that $T_{k} \leq t_{l}$ for any $k=N_{1},N_{1}+1,\ldots$. We have to show that the sequence $v_{k}(t_{l})$, where $k=1,2,\ldots$, is precompact in $\mathbb{E}$.  Let $m$ be a fixed number and $N_{2} \geq N_{1}$ be such that $T_{k} \leq t_{l+m}$ for all $k=N_{2},N_{2}+1,\ldots$. Put $\mathcal{C} = \{ q_{N}, q_{N+1},\ldots \} \cup \{ q \}$ and $s_{m} = t_{l} - t_{l+m} > 0$. From \textbf{(ACOM)}, \textbf{(MNC2)}, \textbf{(MNC4)} and item 2) of the lemma we have
	\begin{equation}
		\label{EQ: CompactnessLemmaMainEstimate}
		\begin{split}
			\alpha_{K}\left( \bigcup_{k=N_{1}}^{\infty} \{ v_{k}(t_{l}) \} \right) = \alpha_{K}\left( \bigcup_{k=N_{2}}^{\infty} \{ v_{k}(t_{l}) \} \right) \leq \alpha_{K}\left( \psi^{s_{m}}(\mathcal{C},\mathcal{B}_{l+m}) \right) \leq \\ \leq \gamma_{0}(s_{m}) e^{-\gamma^{+} s_{m}} \cdot \alpha_{K}\left( \mathcal{B}_{l+m} \right) \leq \gamma_{0}(s_{m}) D e^{-\gamma^{+} t_{l}}.
		\end{split}
	\end{equation}
	Since $s_{m}$ tends to $+\infty$ when $m$ tends to $+\infty$, from \eqref{EQ: CompactnessLemmaMainEstimate} and \textbf{(MNC1)} we get that the sequence $v_{k}(t_{l})$ is precompact for every $l$.
	
	Now, using Cantor's diagonal procedure, we may get a subsequence $v_{k_{m}}(\cdot)$, where $m=1,2,\ldots$, such that for every $l$ the sequence $v_{k_{m}}(t_{l})$ converges to some $\overline{v}_{l}$ as $m \to +\infty$. Consider the trajectories
	\begin{equation}
		\overline{w}_{l}(t):=\psi^{t-t_{l}}(\vartheta^{t_{l}}(q), \overline{v}_{l}) \text{ for } t \geq t_{l}.
	\end{equation}
	Since $v_{k_{m}}(t_{l}) \to \overline{v}_{l}$, we have the convergence $v_{k_{m}}(t) \to \overline{w}_{l}(t)$ for all $t \in [t_{l},+\infty)$. From this it follows that $w_{l}(t)$ and $w_{l-1}(t)$ coincide for all $t \in [t_{l},+\infty)$. Thus, the equality
	\begin{equation}
		v^{*}(t):=\overline{w}_{l}(t), \text{ for arbitrary } l \text{ such that } t \geq t_{l}
	\end{equation}
	correctly defines a complete trajectory of the cocycle such that $v_{k_{m}}(t) \to v^{*}(t)$ for every $t \in (-\infty,+\infty)$. The lemma is proved.
\end{proof}

The following property is usually linked with smoothing-like properties of differential equations. Note that it is always satisfied if $\mathbb{E}=\mathbb{H}$ since in this case we may put $\tau_{S} = 0$.
\begin{description}
	\item[\textbf{(S)}] There exists $\tau_{S} \geq 0$ and $C_{S}>0$ such that
	\begin{equation}
		\| \psi^{\tau_{S}}(q,v_{1})-\psi^{\tau_{S}}(q,v_{2}) \|_{\mathbb{E}} \leq C_{S} |v_{1}-v_{2}|_{\mathbb{H}}.
	\end{equation}
\end{description}

We will usually require the following uniform Lipschitz property, which strengthens \textbf{(S)}, to be satisfied. 
\begin{description}
	\item[\textbf{(ULIP)}] There exists $\tau_{S} \geq 0$ such that for any $T>0$ there exists a constant $L_{T}>0$ such that for all $q \in \mathcal{Q}$ and $v_{1},v_{2} \in \mathbb{E}$ we have
	\begin{equation}
		\|\psi^{t}(q,v_{1})-\psi^{t}(q,v_{2})\|_{\mathbb{E}} \leq L_{T}|v_{1}-v_{2} |_{\mathbb{H}} \text{ for } t \in [t_{S},t_{S}+T]
	\end{equation}
	and also
	\begin{equation}
		\|\psi^{t}(q,v_{1})-\psi^{t}(q,v_{2})\|_{\mathbb{E}} \leq L_{T}\|v_{1}-v_{2} \|_{\mathbb{E}} \text{ for } t \in [0,T].
	\end{equation}
\end{description}

\subsection{A few remarks}
\label{SUBSEC: Motivation}

Let us work within this subsection under the assumptions \textbf{(H1)},\textbf{(H2)} and \textbf{(H3)}. Here we try to expound several basic principles, which should provide a geometric insight onto these assumptions and present ideas, which will be used and developed in the next sections.

We say that two points $v_{1},v_{2} \in \mathbb{E}$ are \textit{pseudo-ordered} (resp. \textit{strictly pseudo-ordered}) over $q \in \mathcal{Q}$ if $V_{q}(v_{1}-v_{2}) \leq 0$ (resp. $V_{q}(v_{1}-v_{2}) < 0$). Sometimes we omit mentioning $q$ if it is clear from the context of which fiber we are referring to.

Let us consider the $V_{q}$-orthogonal projector $\Pi_{q}$ defined in Remark \ref{REM: Orthogonality} (however, the orthogonality is not essential for the following lemma).
\begin{lemma}[Injectivity of projectors on pseudo-ordered sets]
	\label{LEM: OrderProjectorPrinciple}
	Let two different points $v_{1},v_{2} \in \mathbb{E}$ be pseudo-ordered over some $q \in \mathcal{Q}$. Then $\Pi_{q}v_{1} \not= \Pi_{q} v_{2}$.
\end{lemma}
\begin{proof}
	Indeed, if $\Pi_{q}v_{1} = \Pi_{q} v_{2}$ then $v_{1} - v_{2} \in \mathbb{E}^{+}(q)$ and, consequently, $V_{q}(v_{1}-v_{2}) \geq 0$. Thus $V_{q}(v_{1}-v_{2}) = 0$ and $v_{1}=v_{2}$.
\end{proof}

The set $\mathcal{C}_{q} := \{ v \in \mathbb{E} \ | \ V_{q}(v) \leq 0 \}$ is a quadratic cone of rank $j$, i.~e. it is a closed set such that $\alpha v \in \mathcal{C}_{q}$ for all $\alpha \in \mathbb{R}$ and $v \in \mathcal{C}_{q}$ and the maximal dimension of a linear subspace $\mathbb{L} \subset \mathbb{E}$, which entirely lies in $\mathcal{C}_{q}$, is $j$.

The following lemma shows that the pseudo-ordering provided by these cones is strongly preserved under the dynamics.
\begin{lemma}[Cone invariance]
	\label{LEM: MonotonicityPrinciple}
	Suppose two different points $v_{1},v_{2} \in \mathbb{E}$ are pseudo-ordered over some $q \in \mathcal{Q}$. Then their trajectories $\psi^{t}(q,v_{1})$ and $\psi^{t}(q,v_{2})$ are strictly pseudo-ordered over $\vartheta^{t}(q)$ for all $t \geq \tau_{V}$.
\end{lemma}
\begin{proof}
	From \textbf{(H3)} we have for $t \geq \tau_{V}$
	\begin{equation}
		e^{2\nu t}V_{\vartheta^{t}(q)}(\psi^{t}(q,v_{1}) - \psi^{t}(q,v_{2})) \leq V_{q}(v_{1}-v_{2}) - \delta \int_{0}^{t}e^{2\nu s}|\psi^{s}(q,v_{1}) - \psi^{s}(q,v_{2})|^{2}_{\mathbb{H}}ds
	\end{equation}
	that immediately implies the conclusion of the lemma.
\end{proof}

\begin{lemma}[Squeezing property]
	\label{LEM: SqueezingPrinciple}
	Suppose two different points $v_{1},v_{2} \in \mathbb{E}$ and $q \in \mathcal{Q}$ are given such that $V_{q}(\psi^{T}(q,v_{1}) - \psi^{T}(q,v_{2}) \geq 0$ for some $T \geq \tau_{V}$. Then for all $t \in [\tau_{V},T]$ we have
	\begin{equation}
		\label{EQ: SqueezingPrincipleIntegral}
		\delta^{-1} V_{q}(v_{1}-v_{2}) \geq \int_{0}^{t}e^{2\nu s}| \psi^{s}(q,v_{1}) - \psi^{s}(q,v_{2}) |^{2}_{\mathbb{H}}ds.
	\end{equation}
	Moreover, in the case \textbf{(ULIP)} satisfied and $T \geq \tau_{S}+1$, we have for all $t \in [\max\{\tau_{V},\tau_{S}+1\};T]$ the estimate
	\begin{equation}
		\label{EQ: SqueezingPrincipleExponential}
		\| \psi^{t}(q,v_{1}) - \psi^{t}(q,v_{2}) \|^{2}_{\mathbb{E}} \leq \delta^{-1} (C'_{\tau_{S}+1})^{2} e^{2\nu (\tau_{S}+1)} V_{q}(v_{1}-v_{2}) \cdot e^{-2\nu t}.
	\end{equation}
\end{lemma}
\begin{proof}
	Since $V_{q}(\psi^{T}(q,v_{1}) - \psi^{T}(q,v_{2})) \geq 0$, by Lemma \ref{LEM: MonotonicityPrinciple} we have that $V_{q}(\psi^{t}(q,v_{1}) - \psi^{t}(q,v_{2})) \geq 0$ for all $t \in [\tau_{V},T]$. Thus \eqref{EQ: SqueezingPrincipleIntegral} immediately follows from \textbf{(H3)}. Now from the mean value theorem and \textbf{(ULIP)} we get
	\begin{equation}
		\begin{split}
			\int_{0}^{t}e^{2\nu s}| \psi^{s}(q,v_{1}) - \psi^{s}(q,v_{2}) |^{2}_{\mathbb{H}}ds \geq \\ \geq \int_{t-\tau_{S}-1}^{t-\tau_{S}}e^{2\nu s}| \psi^{s}(q,v_{1}) - \psi^{s}(q,v_{2}) |^{2}_{\mathbb{H}}ds = e^{2\nu s_{0}}| \psi^{s_{0}}(q,v_{1}) - \psi^{s_{0}}(q,v_{2}) |^{2}_{\mathbb{H}}ds \geq \\ \geq e^{-2\nu(\tau_{S}+1)}(C'_{\tau_{S}+1})^{-2} e^{2\nu t} \| \psi^{t}(q,v_{1}) - \psi^{t}(q,v_{2}) \|^{2}_{\mathbb{E}},
		\end{split}
	\end{equation}
	for some $s_{0} \in [t - \tau_{S}-1, t -\tau_{S}]$. From this and \eqref{EQ: SqueezingPrincipleIntegral} we immediately get \eqref{EQ: SqueezingPrincipleExponential}.
\end{proof}

Lemmas \ref{LEM: OrderProjectorPrinciple}, \ref{LEM: MonotonicityPrinciple} and \ref{LEM: SqueezingPrinciple} present three general principles, the ideas of which will be developed throughout this paper.

For one-dimensional cones (i.~e. for $j=1$) there are two natural convex parts of $\mathcal{C}_{q}$, say $\mathcal{C}^{+}_{q}$ and $-\mathcal{C}^{+}_{q}$. Then the relation $v_{1} \prec v_{2}$ iff $v_{2} - v_{1} \in \mathcal{C}^{+}_{q}$ defines a partial order in $\mathbb{E}$ (the convexity implies transitivity). This is no longer true for cones of higher rank. However, as we will see in Subsection \ref{SUBSEC: HorizontalLeaves}, the negative or positive pseudo-ordering relations become an equivalence relation (in particular, a transitive relation) if we apply them for trajectories of the cocycle.

\section{Recovery of foliations}
\label{SEC: ConstructionIM}

Basic ideas behind our construction of inertial manifolds (or, more generally, foliations) are inspired by R.A.~Smith's paper on non-autonomous ODEs \cite{Smith1986Massera}, where the operators $P(q)$ (and the corresponding sign spaces) were independent of $q \in \mathcal{Q}$ and the exponent $\nu_{0}$ was constant. In \cite{Anikushin2020Red} the present author generalized Smith's approach for cocycles in Hilbert spaces (i.~e. when $\mathbb{E}=\mathbb{H}$). Some notes on possible extensions of results from \cite{Anikushin2020Red} for Banach spaces, motivated by considering delay equations, are presented in \cite{Anikushin2020Semigroups}. Consideration of the fibre-dependent exponent $\nu_{0}$ is motivated by the works of A.~Kostianko and S.~Zelik \cite{KostiankoZelik2015} and A.~Kostianko et. al. \cite{KostiankoZelikSA2020} on the Spatial Averaging Principle. 

As in \cite{Smith1986Massera}, to construct inertial manifolds (or, more generally, what we call \textit{horizontal leaves}) we use the Graph Transform technique, which is also frequently used in hyperbolic dynamics (see, for example, \cite{BarreiraPesin2007}) or in the theory of normally hyperbolic invariant manifolds (see, for example, \cite{BatesLuZeng2000, BatesLuZeng1998}). This approach is concerned with iterates (pullback iterates in the nonautonomous case) of certain submanifolds, which converge to the desired invariant manifold. In our opinion, an essential contribution of R.A.~Smith was to consider a general class of quadratic cones. This is important for applications of the theory, where such cones are constructed via the Frequency Theorem.

\subsection{Admissible manifolds and horizontal leaves}
\label{SUBSEC: HorizontalLeaves}

Throughout this section we will make use of a bit weaker form of \textbf{(H3)}. We also state it for a different family of operators $Q=Q(q)$ for purposes of further developments. This hypothesis can be stated as follows.
\begin{description}
	\item[$\textbf{(H3)}^{-}_{w}$] There exists a family of operators $Q(q) \in \mathcal{L}(\mathbb{E};\mathbb{E}^{*})$, which satisfies \textbf{(H1)} with $P(q)$ changed to $Q(q)$, $\mathbb{E}^{-}(q)$ and $\mathbb{E}^{+}(q)$ changed to $\mathbb{E}^{-}_{Q}(q)$, $\mathbb{E}^{+}_{Q}(q)$ and \textbf{(H2)} with $j$ changed to $j_{Q}$. Moreover, for some constants $-\infty < \alpha^{-} < \alpha^{+} < +\infty$, $\delta_{Q}>0$, $\tau_{Q} \geq 0$ and a Borel measurable function $\alpha_{0} \colon \mathcal{Q} \to \mathbb{R}$ such that $\alpha^{-} \leq \alpha_{0}(\cdot) \leq \alpha^{+}$ and for $\alpha(t;q) := \int_{0}^{t}\alpha_{0}(\vartheta^{s}(q))ds$ and $V^{Q}_{q}(v):=\langle v, Q(q)v \rangle$ we have the inequality
	\begin{equation}
		\label{EQ: H3WeakMinus}
		\begin{split}
		e^{2\alpha(r;q)}V^{Q}_{\vartheta^{r}(q)}(\psi^{r}(q,v_{1})-\psi^{r}(q,v_{2})) - V^{Q}_{q} (v_{1} - v_{2}) \leq \\ \leq -\delta_{Q} \int_{0}^{r} e^{2\alpha(s;q)}|\psi^{s}(q,v_{1})-\psi^{s}(q,v_{2})|^{2}_{\mathbb{H}}ds 
		\end{split}
	\end{equation}
    satisfied for all $q \in \mathcal{Q}$, $r \geq \tau_{Q}$ and $v_{1},v_{2} \in \mathbb{E}$ such that $V^{Q}_{q}(v_{1}-v_{2}) \leq 0$.
\end{description}

\begin{remark}
	We have to emphasize the two changes made in \textbf{(H3)}. Firstly, we require the function $\nu_{0}(\cdot)$ (that is $\alpha_{0}(\cdot)$ now) to be bounded only (not necessarily positive). In further sections, we apply results of the present section for both positive $\alpha^{-}>0$ and negative $\alpha^{+} < 0 $ cases. Secondly, we require \eqref{EQ: PropH3inequality} (that is \eqref{EQ: H3WeakMinus} now) to be satisfied only for negatively pseudo-ordered initial points $v_{1}$ and $v_{2}$. 
\end{remark}

\begin{remark}
	\label{REM: NotationQomited}
	For simplicity of notation, throughout this subsection we write $V_{q}$, $\mathbb{E}^{-}(q)$ and $\mathbb{E}^{+}(q)$ instead of $V^{Q}_{q}$, $\mathbb{E}^{-}_{Q}(q)$ and $\mathbb{E}^{+}_{Q}(q)$ respectively. For the $V^{Q}_{q}$-orthogonal projector and its complementary one, which should have been denoted, for example, as $\Pi^{Q}_{q}$ and $\Pi^{Q,+}_{q}$ respectively, we use the notations $\Pi_{q}$ and $\Pi^{+}_{q}$ respectively.
\end{remark}

\subsubsection{Dynamics of admissible manifolds}

At first, we shall emphasize that the results of this subsection are proved for \underline{semicocycles}, i.e. allow the family of maps $\vartheta^{t} \colon \mathcal{Q} \to \mathcal{Q}$ to be only a semiflow on $\mathcal{Q}$. This will be especially important in the study of vertical foliations in Section \ref{SUBSEC: VerticalFoliation}.

In what follows we strictly use the $V_{q}$-orthogonality of the spaces $\mathbb{E}^{+}(q)$ and $\mathbb{E}^{-}(q)$, which was described in Remark \ref{REM: Orthogonality}. Let us recall that this means the relation $V_{q}(v)=V_{q}(v^{+})+V_{q}(v^{-})$ holds for all $v \in \mathbb{E}$, where $v=v^{+}+v^{-}$ is the unique decomposition with $v^{+} \in \mathbb{E}^{+}(q)$ and $v^{-} \in \mathbb{E}^{-}(q)$. We consider the corresponding $V_{q}$-orthogonal projector given by $\Pi_{q}(v):=v^{-}$.

For a real number $\varkappa$ (we are interested in $\varkappa \in (0,1)$) and $q \in \mathcal{Q}$ let us consider the quadratic form $V^{(\varkappa)}_{q}(v):= V_{q}(v^{+}) + \varkappa^{2} V_{q}(v^{-})$, where, as usual, $v=v^{+} + v^{-}$ is the unique decomposition with $v^{+} = \Pi^{+}_{q}v \in \mathbb{E}^{+}(q)$ and $v^{-}=\Pi_{q}v \in \mathbb{E}^{-}(q)$. We start with the following lemma showing that $\textbf{(H3)}^{-}_{w}$ is preserved for the family of quadratic forms  $V^{(\varkappa)}_{q}$ if $\varkappa \in (0,1)$ is chosen sufficiently close to $1$. This idea is due to A.V.~Romanov \cite{Romanov1994}.

\begin{lemma}[Cone Perturbation Lemma]
	\label{LEM: ConePerturbation}
	Suppose $\textbf{(H3)}^{-}_{w}$, \textbf{(ULIP)} and \textbf{(PROJ)} (for the $V_{q}$-orthogonal projectors $\Pi_{q}$) are satisfied. Then there exists $\varkappa_{0} \in (0,1)$ such that for any $\varkappa \in [\varkappa_{0},1]$ we have the inequality
	\begin{equation}
		\label{EQ: PerturbedConeCondition}
		\begin{split}
			e^{2\alpha(r;q)}V^{(\varkappa)}_{\vartheta^{r}(q)}(\psi^{r}(q,v_{1})-\psi^{r}(q,v_{2})) - V^{(\varkappa)}_{q}(v_{1}-v_{2}) \leq \\ \leq -\frac{\delta_{Q}}{2} \int_{0}^{r} e^{2\alpha(s;q)}|\psi^{s}(q,v_{1})-\psi^{s}(q,v_{2})|^{2}_{\mathbb{H}}ds.
		\end{split}	
	\end{equation}
	satisfied for all $q \in \mathcal{Q}$, $r \geq \max\{\tau_{Q}, \tau_{S}+1\}$ and $v_{1},v_{2} \in \mathbb{E}$ such that $V_{q}(v_{1}-v_{2}) \leq 0$.
\end{lemma}
\begin{proof}
	In terms of \eqref{EQ: PerturbedConeCondition} let $w_{i}(s):=\psi^{s}(q,v_{i})$ for $i=1,2$ and $s \geq 0$. It is easy to see that \eqref{EQ: H3WeakMinus} is equivalent to the inequality
	\begin{equation}
		\begin{split}
			e^{2\alpha(r;q)}V^{(\varkappa)}_{\vartheta^{r}(q)}\left(w_{1}(r)-w_{2}(r)\right) - V^{(\varkappa)}_{q} \left(w_{1}(0)-w_{2}(0)\right) \leq \\ \leq (1-\varkappa^{2}) V_{q}\left( \Pi_{q} w_{1}(0)-\Pi_{q}w_{2}(0)\right) -\\- (1-\varkappa^{2})e^{2\alpha(r;q)} V_{\vartheta^{r}(q)}\left( \Pi_{\vartheta^{r}(q)} w_{1}(r)-\Pi_{\vartheta^{r}(q)}w_{2}(r)\right) - \\ - \delta_{Q}\int_{0}^{r}e^{2\alpha(s;q)}|w_{1}(s)-w_{2}(s)|^{2}_{\mathbb{H}}ds.
		\end{split}
	\end{equation}
	Since $V_{q}( \Pi_{q} w_{1}(0)-\Pi_{q}w_{2}(0)) \leq 0$, it is sufficient to show that if $1-\varkappa^{2}$ is small then we have
	\begin{equation}
		\begin{split}
			(1-\varkappa^{2})e^{2\alpha(r;q)} V_{\vartheta^{r}(q)}\left( \Pi_{\vartheta^{r}(q)} w_{1}(r)-\Pi_{\vartheta^{r}(q)}w_{2}(r)\right) + \\ + \frac{\delta_{Q}}{2}\int_{0}^{r}e^{2\alpha(s;q) s}|w_{1}(s)-w_{2}(s)|^{2}_{\mathbb{H}}ds \geq 0.
		\end{split}
	\end{equation}
	Since for some constants $M_{Q},M_{\Pi} >0$ we have $\|Q(q)\| \leq M_{Q}$ (thanks to \textbf{(H1)}) and $\| \Pi_{q} \| \leq M_{\Pi}$ (thanks to \textbf{(PROJ)}), we get
	\begin{equation}
		\begin{split}
			\left|(1-\varkappa^{2})e^{2\alpha(r;q)} V_{\vartheta^{r}(q)}\left( \Pi_{\vartheta^{r}(q)} w_{1}(r)-\Pi_{\vartheta^{r}(q)}w_{2}(r)\right)\right| \leq \\ \leq (1-\varkappa^{2}) e^{2\alpha(r;q)} M_{Q} M^{2}_{\Pi} \cdot \| w_{1}(r) -w_{2}(r) \|^{2}_{\mathbb{E}}.
		\end{split}	
	\end{equation}
	By the mean value theorem, \textbf{(ULIP)} and since $r \geq \tau_{S} + 1$, for some $s_{0} \in [r-\tau_{S}-1,r-\tau_{S}]$ we have
	\begin{equation}
		\label{EQ: ConePerturbationLipIntegralEstimate}
		\begin{split}
			\int_{0}^{r}e^{2\alpha(s;q)}|w_{1}(s)-w_{2}(s)|^{2}_{\mathbb{H}}ds \geq \int_{r-\tau_{S}-1}^{r-\tau_{S}}e^{2\alpha(s;q)}|w_{1}(s)-w_{2}(s)|^{2}_{\mathbb{H}}ds = \\ = e^{2\alpha(s_{0};q)}|w_{1}(s_{0})-w_{2}(s_{0})|^{2}_{\mathbb{H}} \geq \\ \geq e^{-2\max\{ |\alpha^{-}|, |\alpha^{+}| \}(\tau_{S}+1)}e^{2\alpha(r;q)} \left(L_{\tau_{S}+1}\right)^{2} \cdot \| w_{1}(r) - w_{2}(r) \|^{2}_{\mathbb{E}}.
		\end{split}    	
	\end{equation}
	Thus, we have \eqref{EQ: PerturbedConeCondition} satisfied for all $\varkappa \in (0,1)$ such that
	\begin{equation}
		(1-\varkappa^{2}) M_{P} M^{2}_{\Pi} < \frac{\delta_{Q}}{2}  e^{-2\max\{ |\alpha^{-}|, |\alpha^{+}| \}(\tau_{S}+1)}\left(L_{\tau_{S}+1}\right)^{2}.
	\end{equation}
	The proof is finished.
\end{proof}
\begin{remark}
	Note that in the proof of Lemma \ref{LEM: ConePerturbation} we were dealing only with the inequality in \eqref{EQ: H3WeakMinus} and the pseudo-ordering assumption $V_{q}(v_{1}-v_{2}) \leq 0$ was not essential. Thus, the same perturbation result holds under \textbf{(H3)} and other similar conditions.
\end{remark}

In the following lemma we assume the agreement from Remark \ref{REM: TrajectoriesOfSemicocycles} that when speaking of a trajectory over $q$ defined for $t \geq T$ with $T<0$ we also assume that a backward extension for $q$ is fixed and denoted by $\vartheta^{s}(q)$ for $s \geq T$.
\begin{lemma}
	\label{LEM: BackwardEstimate}
	Suppose $\textbf{(H3)}^{-}_{w}$ and \textbf{(ULIP)} are satisfied. Let $w_{1}(\cdot)$ and $w_{2}(\cdot)$ be two trajectories over $q_{0}$, which are defined for $t \geq t_{1}$. Suppose that $t_{2} - t_{1} \geq \max\{ \tau_{Q}, \tau_{S}+1 \}$ and $V_{\vartheta^{t_{1}}(q)}(w_{1}(t_{1})-w_{2}(t_{1})) \leq 0$. Then for any $t_{0} \in [t_{1} + \max\{ \tau_{Q},\tau_{S}+1 \}, t_{2}]$ we have
	\begin{equation}
		\label{EQ: BackwardEstimateLemma}
		\begin{split}
			\| w_{1}(t_{0}) - w_{2}(t_{0}) \|^{2}_{\mathbb{E}} \leq \\ \leq \delta^{-1}_{Q} \| Q(\vartheta^{t_{2}}(q_{0}) ) \| \cdot L^{2}_{\tau_{S}+1} \cdot e^{2\max\{ |\alpha^{+}|,|\alpha^{-}| \} (\tau_{S}+1)} \cdot e^{2\alpha(t_{2};q_{0}) - 2 \alpha(t_{0};q_{0})} \cdot \\ \cdot \| \Pi_{\vartheta^{t_{2}}(q_{0})}w_{1}(t_{2}) - \Pi_{\vartheta^{t_{2}}(q_{0})}w_{2}(t_{2}) \|^{2}_{\mathbb{E}}.
		\end{split}
	\end{equation}
\end{lemma}
\begin{proof}
	From $\textbf{(H3)}^{-}_{w}$ we have
	\begin{equation}
		\begin{split}
			\delta^{-1}_{Q} \cdot \|Q(\vartheta^{t_{2}}(q_{0}))\| \cdot \| \Pi_{\vartheta^{t_{2}}(q_{0}) }w_{1}(t_{2}) - \Pi_{\vartheta^{t_{2}}(q_{0})} w_{2}(t_{2}) \|^{2}_{\mathbb{E}} \geq \\ \geq \int_{t_{1}}^{t_{2}}e^{2\alpha(s;q_{0})}|w_{1}(s)-w_{2}(s)|^{2}_{\mathbb{H}}ds \geq \int_{t_{0}-\tau_{S}-1}^{t_{0}}e^{2\alpha(s;q_{0})}|w_{1}(s)-w_{2}(s)|^{2}_{\mathbb{H}}ds.
		\end{split}
	\end{equation}
	Applying the mean value theorem to the last integral and using \textbf{(ULIP)} in the same way we did in \eqref{EQ: ConePerturbationLipIntegralEstimate}, we get \eqref{EQ: BackwardEstimateLemma}.
\end{proof}

Let $\mathcal{M} \subset \mathbb{E}$ be a subset. We say that $\mathcal{M}$ is \textit{strongly negatively pseudo-ordered} over a given $q \in \mathcal{Q}$ if for some $\varkappa \in (0,1)$ and all $v_{1},v_{2} \in \mathcal{M}$ we have $V^{(\varkappa)}_{q}(v_{1}-v_{2}) \leq 0$. We call $\mathcal{M}$ an \textit{admissible set} over $q$ if it is strongly negatively pseudo-ordered and the map $\Pi_{q} \colon \mathcal{M} \to \mathbb{E}^{-}(q)$ is a homeomorphism. Note also that any strongly negatively pseudo-ordered set is strictly negatively pseudo-ordered. Sometimes we will say that $\mathcal{M}$ is $\varkappa$-admissible to emphasize the choice of $\varkappa$.

We call an admissible over $q \in \mathcal{Q}$ set $\mathcal{M}$ \textit{Lipschitz admissible} if the inverse to $\Pi_{q} \colon \mathcal{M} \to \mathbb{E}^{-}(q)$ is Lipschitz. Let $L(\mathcal{M})$ be the Lipschitz constant of the inverse map. We will sometimes say that $\mathcal{M}$ is $L(\mathcal{M})$-Lipschitz admissible (or $L(\mathcal{M})$-Lipschitz $\varkappa$-admissible) to emphasize the chosen parameters.

The following lemma describes dynamics of admissible sets (similar properties are also known in hyperbolic dynamics \cite{BarreiraPesin2007}). It is of special interest in the case of non-invertible dynamics (for example, given by parabolic problems or delay equations) since it reveals non-invariant structures which are homeomorphically mapped onto each other. Note also that its proof heavily relies on the Brouwer theorem on invariance of domain.
\begin{figure}
	\centering
	\includegraphics[width=1\linewidth]{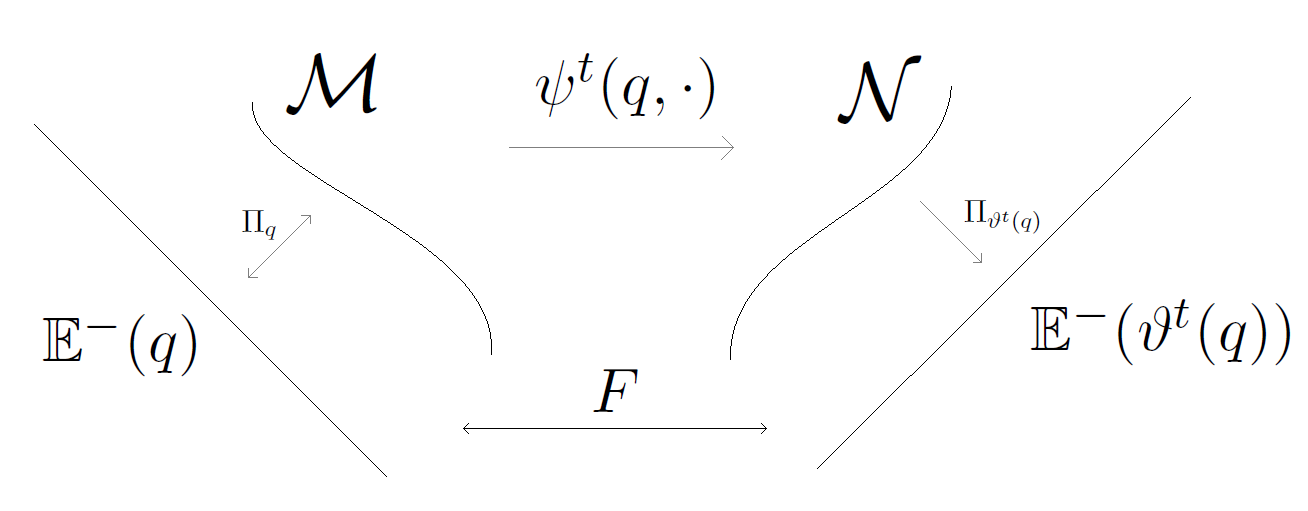}
	\caption{An illustation to Lemma \ref{LEM: GraphTransform}. Here arrows show directions for which the continuity is given a priori or is proved in the lemma. Clearly, this is sufficient to obtain continuity in the other (not drawn) directions.}
	\label{FIG: GraphTransformMN}
\end{figure}
\begin{lemma}
	\label{LEM: GraphTransform}
	Under the hypotheses of Lemma \ref{LEM: ConePerturbation}, let $\mathcal{M}$ be a $\varkappa$-admissible set over $q \in \mathcal{Q}$ for some $\varkappa \geq \varkappa_{0}$ (here $\varkappa_{0}$ is given by Lemma \ref{LEM: ConePerturbation}). Then for any $t \geq \max\{ \tau_{Q},\tau_{S}+1 \}$ the set $\mathcal{N} := \psi^{t}(q,\mathcal{M})$ is $C_{Lip}(\vartheta^{t}(q))$-Lipschitz $\varkappa$-admissible over $\vartheta^{t}(q)$ and the map $\psi^{t}(q,\cdot) \colon \mathcal{M} \to \mathcal{N}$ is a homeomorphism. Here $C_{Lip}(\vartheta^{t}(q))$ is given by
	\begin{equation}
		\label{EQ: LipschitzConstantOverQ}
		C_{Lip}(q) := \left( \delta^{-1}_{Q} \| Q(q) \| \right)^{1/2} \cdot L_{\tau_{S}+1} \cdot e^{\max\{ |\alpha^{+}|, |\alpha^{-}| \}(\tau_{S}+1)}.
	\end{equation}
\end{lemma}
\begin{proof}
	Since $\mathcal{M}$ is $\varkappa$-admissible, for arbitrary points $v_{1},v_{2} \in \mathcal{M}$ we have $V^{(\varkappa)}_{q}(v_{1}-v_{2}) \leq 0$. Then, by \eqref{EQ: PerturbedConeCondition} we have
	\begin{equation}
		e^{2\alpha(t;q)}V^{(\varkappa)}_{\vartheta^{t}(q)}(\psi^{t}(q,v_{1})-\psi^{t}(q,v_{2})) \leq V^{(\varkappa)}_{q}(v_{1}-v_{2}) \leq 0.
	\end{equation}
	Thus, $\mathcal{N}$ is also strongly negatively pseudo-ordered. Let $\Pi^{-1}_{q} \colon \mathbb{E}^{-}(q) \to \mathcal{M}$ denote the inverse to $\Pi_{q} \colon \mathcal{M} \to \mathbb{E}^{-}(q)$ map, which is a homeomorphism by definition. Note that for our purposes it is sufficient to show that the map $F \colon \mathbb{E}^{-}(q) \to \mathbb{E}^{-}(\vartheta^{t}(q))$, which is given by $F(\zeta) = \Pi_{\vartheta^{t}(q)}\psi^{t}(q,\Pi^{-1}_{q} \zeta)$ for $\zeta \in \mathbb{E}^{-}(q)$, is a homeomorphism (see Fig. \ref{FIG: GraphTransformMN}). Suppose that $\Pi_{q} v_{1} = \zeta_{1}$  and $\Pi_{q} v_{2} = \zeta_{2}$. Thus, $F(\zeta_{i}) = \Pi_{\vartheta^{t}(q)}\psi^{t}(q,v_{i})$ for $i = 0,1$. From this and \eqref{EQ: H3WeakMinus} we get
	\begin{equation}
		\label{EQ: AdmissibleSetsLemmaMain}
		\begin{split}
			e^{2\alpha(t;q)} \| Q(q) \| \cdot \| F(\zeta_{1}) - F(\zeta_{2}) \|^{2}_{\mathbb{E}} \geq -e^{2\alpha(t;q)}V_{\vartheta^{t}(q)}(\psi^{t}(q,v_{1})-\psi^{t}(q,v_{2})) \\ \geq -V_{q}(v_{1}-v_{2}) = -V^{(\varkappa)}_{q}(v_{1}-v_{2}) - (1-\varkappa^{2})V_{q}(\zeta_{1}-\zeta_{2}) \geq \\ \geq -(1-\varkappa^{2})V_{q}(\zeta_{1}-\zeta_{2}) \geq (1-\varkappa^{2})C^{2}_{q}\|\zeta_{1} - \zeta_{2}\|^{2}_{\mathbb{E}},
		\end{split}
	\end{equation}
	where $C_{q}>0$ is given by \eqref{EQ: ConstantProjectorNorm}. Since $\mathcal{M}$ is strongly negatively pseudo-ordered, it is strictly negatively pseudo-ordered, i.~e. we have $V_{q}(v_{1}-v_{2}) < 0$. This and the second inequality in \eqref{EQ: AdmissibleSetsLemmaMain} shows that the map $\psi^{t}(q,\cdot) \colon \mathcal{M} \to \mathcal{N}$ is injective. By similar reasonings concerned with $\mathcal{N}$, the map $\Pi_{\vartheta^{t}(q)} \colon \mathcal{N} \to \mathbb{E}^{-}(\vartheta^{t}(q))$ is injective. Thus, the map $F$ is a continuous injective map between $j$-dimensional spaces. By the Brouwer theorem on invariance of domain, the image $F(\mathbb{E}^{-}(q))$ is open in $\mathbb{E}^{-}(q)$ and $F$ is a homeomorphism onto the image. Moreover, the relation between the first and last term in \eqref{EQ: AdmissibleSetsLemmaMain} shows that the image must be also closed. Therefore $F(\mathbb{E}^{-}(q))=\mathbb{E}^{-}(\vartheta^{t}(q))$ and $F$ is a homeomorphism.
	
	To show that $\mathcal{N}$ is Lipschitz admissible we use Lemma \ref{LEM: BackwardEstimate} with $q_{0} = \vartheta^{t}(q)$, $t_{2} = 0$, $t_{1} = -t$ and $t_{0} = 0$.
\end{proof}

\begin{exercise}
	\label{EXERCISE: GraphTransformPlane}
	Consider the special case encountered in Lemma \ref{LEM: GraphTransform} given by $\mathcal{M} = \mathbb{F} + v$ for some $v \in \mathbb{E}$ and a strictly pseudo-ordered (over $q$) linear subspace. Show that if we only require that $\textbf{(H3)}^{-}_{w}$ is satisfied then for all $t \geq \tau_{Q}$ the set $\mathcal{N} = \psi^{t}(q,\mathcal{M})$ is strictly pseudo-ordered and the map $\Pi_{\vartheta^{t}(q)} \colon \mathcal{N} \to \mathbb{E}^{-}(\vartheta^{t}(q))$ is a homeomorphism.
	
	\textit{Hint:} Show that the map $F$ from the lemma is a homeomorphism by noting that an analog of the last estimate in \eqref{EQ: AdmissibleSetsLemmaMain} is valid due to the equivalence of norms given by $\sqrt{-V_{q}(\cdot)}$ and $\|\cdot\|_{\mathbb{E}}$ in the finite-dimensional space $\mathbb{F}$.
\end{exercise}

\subsubsection{Negatively pseudo-ordered manifolds along complete orbits}
\label{SUBSUBSECTION: NegativePseudoOrderedOrbits}
We say that two complete trajectories $v_{1}(\cdot)$ and $v_{2}(\cdot)$ over some $q \in \mathcal{Q}$ are \textit{pseudo-ordered} if the points $v_{1}(t)$ and $v_{2}(t)$ are pseudo-ordered over $\vartheta^{t}(q)$ for all $t \in \mathbb{R}$, i.~e. $V_{\vartheta^{t}(q)}(v_{1}(t)-v_{2}(t)) \leq 0$ for all $t \in \mathbb{R}$. Note that due to the monotonicity it is sufficient to verify the latter property for a sequence $t=t_{k}$ tending to $-\infty$. We have already mentioned that the pseudo-ordering relation in $\mathbb{E}$ over each $q$ is not transitive. However, as the following lemma shows, there is a some sort of dynamical transitivity. Namely, its above given extension to the set of complete trajectories (over a given $q \in \mathcal{Q}$) is transitive and, consequently, defines an equivalence relation. 
\begin{lemma}
	\label{LEM: CompleteTrajectoriesOrdering}
	Suppose $\textbf{(H3)}^{-}_{w}$ and \textbf{(S)} are satisfied. Let $v_{i}(\cdot)$, $i=1,2,3$, be complete trajectories over some $q \in \mathcal{Q}$. Then we have
	\begin{enumerate}
		\item[1)] The trajectories $v_{1}(\cdot)$ and $v_{2}(\cdot)$ are pseudo-ordered if and only if for some (consequently, all) $t \in \mathbb{R}$ we have finiteness (and, consequently, the estimate) of the integral
		\begin{equation}
			\label{EQ: OrderedCompleteTrajEstimate}
			\int_{-\infty}^{t}e^{2\alpha(s;q)} |v_{1}(s)-v_{2}(s)|^{2}_{\mathbb{H}}ds \leq -\delta^{-1}_{Q}e^{2\alpha(t;q)}V_{\vartheta^{t}(q)}(v_{1}(t)-v_{2}(t)) <+\infty.
		\end{equation}
	    \item[2)] If $v_{1}(\cdot)$ and $v_{2}(\cdot)$ are pseudo-ordered and $v_{2}(\cdot)$ and $v_{3}(\cdot)$ are pseudo-ordered, then $v_{1}(\cdot)$ and $v_{3}(\cdot)$ are pseudo-ordered.
	    \item[3)] If $v_{1}(\cdot)$ and $v_{2}(\cdot)$ are pseudo-ordered and $\Pi_{\vartheta^{t}(q)}v_{1}(t) = \Pi_{\vartheta^{t}(q)}v_{2}(t)$ for some $t \in \mathbb{R}$, then $v_{1}(\cdot)$ and $v_{2}(\cdot)$ coincide.
	\end{enumerate}
\end{lemma}
\begin{proof}
	1) Using \textbf{(S)} we get
	\begin{equation}
		\label{EQ: CompleteTrajLemma1}
		\int_{-\infty}^{0}e^{2\alpha(s;q)} \|v_{1}(s)-v_{2}(s)\|^{2}_{\mathbb{E}} ds \leq e^{2 \alpha^{+} \tau_{S}} \cdot C^{2}_{S} \cdot \int_{-\infty}^{0}e^{2\alpha(s;q)} |v_{1}(s)-v_{2}(s)|^{2}_{\mathbb{H}}ds.
	\end{equation}
	Applying \textbf{(H3)}, we get for all $l \leq t - \tau_{Q}$ the inequality
	\begin{equation}
		\label{EQ: CompleteTrajLemmaBeforeLimit}
		\begin{split}
			e^{2\alpha(t;q)} V_{\vartheta^{t}(q)}(v_{1}(t)-v_{2}(t)) - e^{2\alpha(l;q)}V_{\vartheta^{l}(q)}(v_{1}(l)-v_{2}(l)) \leq \\ \leq -\delta_{Q} \int_{l}^{t}e^{2\alpha(s;q)}| v_{1}(s)-v_{2}(s) |^{2}_{\mathbb{H}} ds.
		\end{split}
	\end{equation}
	If $v_{1}(\cdot)$ and $v_{2}(\cdot)$ are pseudo-ordered, then we can eliminate the second term in \eqref{EQ: CompleteTrajLemmaBeforeLimit} that leads to the estimate in \eqref{EQ: OrderedCompleteTrajEstimate}. On the other hand, if the last integral in \eqref{EQ: CompleteTrajLemma1} converges, then there exists a sequence $l=l_{k}$, where $k=1,2,\ldots$, tending to $-\infty$ such that $e^{2\alpha (l_{k};q)}\|v_{1}(l_{k})-v_{2}(l_{k})\|^{2}_{\mathbb{E}}$ tends to $0$ as $k \to +\infty$. From this and due to the uniform boundedness of $Q(q)$ we have $e^{2\alpha(l_{k};q)} V_{\vartheta^{l_{k}}(q)}(v_{1}(l_{k})-v_{2}(l_{k})) \to 0$. Putting in \eqref{EQ: CompleteTrajLemmaBeforeLimit} $l=l_{k}$ and taking it to the limit as $k \to +\infty$, we get the estimate in \eqref{EQ: OrderedCompleteTrajEstimate} and, as a corollary, the pseudo-ordering property of $v_{1}(\cdot)$ and $v_{2}(\cdot)$. Thus item 1) is proved.
	
	2) The statement follows from item 1) and the Minkowsky inequality as
	\begin{equation}
		\begin{split}
			\left(\int_{-\infty}^{0}e^{2\alpha(s;q)} |v_{1}(s)-v_{3}(s)|^{2}_{\mathbb{H}}ds\right)^{1/2} \leq \left(\int_{-\infty}^{0}e^{2\alpha(s;q)} |v_{1}(s)-v_{2}(s)|^{2}_{\mathbb{H}}ds\right)^{1/2} + \\ + \left(\int_{-\infty}^{0}e^{2\alpha(s;q)} |v_{2}(s)-v_{3}(s)|^{2}_{\mathbb{H}}ds\right)^{1/2}.
		\end{split}
	\end{equation}

    3) Since $V_{\vartheta^{t}(q)}(v_{1}(t)-v_{2}(t)) \leq 0$, from $\Pi_{\vartheta^{t}(q)}v_{1}(t) = \Pi_{\vartheta^{t}(q)}v_{2}(t)$ we have that $v_{1}(t) = v_{2}(t)$. Now the required statement follows from \eqref{EQ: OrderedCompleteTrajEstimate}.
\end{proof}

Now for a given complete trajectory $v^{*}(\cdot)$ over $q \in \mathcal{Q}$ let $[v^{*}(\cdot)]$ denote its equivalence class, i.e. the set of all complete trajectories over $q$ which are pseudo-ordered with $v^{*}(\cdot)$. For any $t \in \mathbb{R}$ we also consider the values which take these trajectories at time $t$. Namely, we define
\begin{equation}
	\label{EQ: EquivalenceSectionDefinition}
	[v^{*}(\cdot)](t) := \{ v_{0} \in \mathbb{E} \ | \ v_{0}=v(t) \text{ for some } v(\cdot) \in [v^{*}(\cdot)] \}.
\end{equation}
\begin{remark}
	Clearly, we have $\psi^{t}(\vartheta^{s}(q), [v^{*}(\cdot)](s)) = [v^{*}(\cdot)](s+t)$ for $s \in \mathbb{R}$ and $t \geq 0$. Note that this gives us the invariance along the orbit $s \mapsto \vartheta^{s}(q)$ only if the point $q$ is not stationary or periodic. For example, in the autonomous case (i.e. when $\mathcal{Q}$ is a one point set and, consequently, the cocycle is a semiflow) the set $[v^{*}(\cdot)](0)$ is not invariant in general. This is caused by the fact that the trajectories $v^{*}(\cdot)$ and $v^{*}(\cdot+s)$ (i.~e. translated by time $s$) may belong to different equivalence classes! In further sections we will show that there is a natural equivalence class which possesses the invariance. In applications, where $\alpha^{-} > 0$ it corresponds to the equivalence class of a bounded complete orbit $v^{*}(\cdot)$ and the corresponding sets (the \textit{principal leafs}) $\mathfrak{A}(q):=[v^{*}(\cdot)](0)$ form what we call the \textit{inertial manifold} for the cocycle.
\end{remark}
\begin{remark}
	A given complete trajectory $v^{*}(\cdot)$ may be a complete trajectory over different points $q$, which admit backward continuation (for example, when the semicocycle is linear and $v^{*}(\cdot)$ is the zero complete trajectory). Thus, the equivalence class $[v^{*}(\cdot)]$ and its sections strictly depend on the choice $q$. Further, when speaking about these equivalence classes, we always emphasize this choice if it is not clear from the context as it was before.
\end{remark}

 In particular, from this and \eqref{EQ: OrderedCompleteTrajEstimate} we get
\begin{equation}
	\label{EQ: CompleteTrajPiEstimate}
	\delta^{-1}_{Q} \cdot \|Q(q)\| \cdot \| \Pi_{q} v_{1}(0) - \Pi_{q} v_{2}(0) \|^{2}_{\mathbb{E}} \geq \int_{-\infty}^{0} e^{2\alpha(s,q)} | v_{1}(s)-v_{2}(s) |^{2}_{\mathbb{H}} ds
\end{equation}
for any two complete trajectories $v_{1}(\cdot)$ and $v_{2}(\cdot)$ over $q$ which are pseudo-ordered.

Now we start to describe the topological structure of the set $[v^{*}(\cdot)](0)$. By definition, the set $[v^{*}(\cdot)](0)$ is strictly negatively pseudo-ordered over $q$, i.~e. for any of its distinct points $v_{1}$ and $v_{2}$ we have $V_{q}(v_{1}-v_{2}) < 0$. This gives us that the $V_{q}$-orthogonal projector $\Pi_{q}$ is injective on this set. The following lemma strengthens this property.

\begin{lemma}
	\label{EQ: ProjectorHomeomorphismOntoImage}
	Let $\textbf{(H3)}^{-}_{w}$ and \textbf{(S)} be satisfied. Then the map $\Pi_{q} \colon [v^{*}(\cdot)](0) \to \mathbb{E}^{-}(q)$ is a homeomorphism onto its image.
\end{lemma}
\begin{proof}
	As it was noted, $\Pi_{q}$ is injective on $[v^{*}(\cdot)](0)$ and it is also clear that $\Pi_{q}$ is continuous.
	
	To show that the inverse map is continuous let $v_{k}(\cdot)$, where $k=1,2,\ldots$, be a sequence of complete trajectories over $q$ such that $v_{k}(\cdot) \in [v^{*}(\cdot)]$ and $\Pi_{q}v_{k}(0)$ converges to $\Pi_{q}v(0)$ for some complete trajectory $v(\cdot) \in [v^{*}(\cdot)]$. From \eqref{EQ: CompleteTrajPiEstimate} we have
	\begin{equation}
		\begin{split}
			\delta^{-1}_{Q} \|Q(q)\| \cdot \| \Pi_{q} v_{k}(0) - \Pi_{q} v(0) \|^{2}_{\mathbb{E}} &\geq \int_{-\infty}^{0} e^{2\alpha(s;q)} | v_{k}(s)-v(s) |^{2}_{\mathbb{H}} ds \geq \\ &\geq \int_{-\tau_{S} - 2}^{-\tau_{S} - 1}e^{2\alpha(s;q)} | v_{k}(s)-v(s) |^{2}_{\mathbb{H}} ds.
		\end{split}	
	\end{equation}
	Applying the mean value theorem to the last integral, we get a sequence $s_{k} \in [-\tau_{S}-2,-\tau_{S}-1]$ such that $| v_{k}(s_{k})-v(s_{k})|_{\mathbb{H}} \to 0$ as $k \to +\infty$. Let us suppose that $v_{k}(0)$ does not converge to $v(0)$ in $\mathbb{E}$. Then there is a subsequence (for convenience, we keep the same index) and a number $\delta>0$ such that $\|v_{k}(0)-v(0)\|_{\mathbb{E}} \geq \delta$ for all $k$. We may also assume that $s_{k}$ converges to some $\overline{s}$. By \textbf{(S)} we have that $\| v_{k}(s_{k}+\tau_{S})-v(s_{k}+\tau_{S})\|_{\mathbb{E}} \to 0$ as $k \to +\infty$. From continuity of the cocycle we get that $v_{k}(s) \to v(s)$ for all $s > \overline{s}+\tau_{S} \geq -1$. In particular, the convergence for $s=0$ contradicts the inequality $\|v_{k}(0) - v(0)\| \geq \delta$. Thus, the inverse map is continuous.
\end{proof}

Now our aim is to show that $\Pi_{q} [v^{*}(\cdot)](0) = \mathbb{E}^{-}(q)$ and here the Graph Transform plays a crucial role. However, before proceeding to this we shall present a more detailed construction which has its intrinsic interest. This will require additional assumptions, which can be relaxed for purposes of the particular case (see Exercise \ref{EXERCISE: GraphTransformPlane}), but will be anyway used throughout other sections.

Let $T_{1} < T_{2}$ be some real numbers and let $q \in \mathcal{Q}$ be fixed. Moreover, let $\mathcal{M}$ be admissible over $\vartheta^{T_{1}}(q)$. We define the map $G_{T_{1}}^{T_{2}} \colon \mathcal{M} \to \mathbb{E}^{-}(\vartheta^{T_{2}}(q))$ as
\begin{equation}
	\label{EQ: GraphTransformDefinition}
	G_{T_{1}}^{T_{2}}(v) = G_{T_{1}}^{T_{2}}(v;q,\mathcal{M}) := \Pi_{\vartheta^{T_{2}}(q)} \psi^{T_{2}-T_{1}}(\vartheta^{T_{1}}(q),v) \text{ for } v \in \mathcal{M}.
\end{equation}
From Lemma \ref{LEM: GraphTransform} and Exercise \ref{EXERCISE: GraphTransformPlane} we have the following corollary.

\begin{corollary}
	\label{COR: GraphTransformHomeo}
	Under the hypotheses of Lemma \ref{LEM: ConePerturbation} the map $G_{T_{1}}^{T_{2}} \colon \mathbb{E}^{-} \to \mathbb{E}^{-}(\vartheta^{T_{2}}(q))$ defined in \eqref{EQ: GraphTransformDefinition} is a homeomorphism for all $T_{2}-T_{1} \geq \tau_{Q}$. Moreover, for $\mathcal{M} = \mathbb{F} + v$, where $\mathbb{F}$ is a strictly pseudo-ordered (over $q$) $j_{Q}$-dimensional linear subspace, it is sufficient to require only $\textbf{(H3)}^{-}_{w}$ to be satisfied.
\end{corollary}

The reader is also encouraged to complete the following exercise to improve his understanding.
\begin{exercise}
	\label{EXER: CompactOperatorExercise}
	Let $\textbf{(H3)}_{w}^{-}$ be satisfied and suppose that $\mathbb{E} = \mathbb{H}$. Show that if for some $q \in \mathcal{Q}$ the operator $Q(q)$ is compact\footnote{It is known that $Q(q)$ is compact for certain parabolic problems \cite{Anikushin2019+OnCom,Likhtarnikov1976}.}, then the map $\psi^{r}(q,\cdot) \colon \mathbb{H} \to \mathbb{H}$ is compact if it takes bounded sets into bounded sets and $r \geq \tau_{Q}$.
	
	\textit{Hint:} let $v_{k} \in \mathbb{H}$ be a bounded sequence and consider $w_{k}(s):=\psi^{s}(q,v_{k})$ for $s \geq 0$. Suppose $v_{k}$ converges weakly in $\mathbb{H}$ to some $\overline{v}$ and let also $\psi^{r}(q,v_{k})$ converge weakly to some $\overline{w} \in \mathbb{H}$. Show that for
	\begin{equation}
		v(s):=\psi^{s}\left(q,\left(G^{r}_{0}\right)^{-1}(\Pi \overline{w}; q, \overline{v}, \mathbb{E}^{-}(q)) \right),
	\end{equation}
	which is given for $s \geq 0$, we have the convergence $v_{k}(s) \to v(s)$ for all $s > 0$.
\end{exercise}

Now we will use the map $G_{T_{1}}^{T_{2}}$ to describe the section $[v^{*}(\cdot)](0)$. The following lemma is the first place where the asymptotic compactness (via Lemma \ref{LEM: GeneralCompactnessLemma}) is utilized.
\begin{lemma}
	\label{LEM: ManifoldConstruction}
	Suppose $\textbf{(H3)}^{-}_{w}$, \textbf{(ACOM)} with $\gamma^{+} = \alpha^{+}$ and \textbf{(ULIP)} are satisfied. Let $v^{*}(\cdot)$ be a complete trajectory over some $q \in \mathcal{Q}$. Then for every $\zeta \in \mathbb{E}^{-}(q)$ the sequence of trajectories over $q$ given by
	\begin{equation}
		v_{\theta}(s):=\psi^{s-\theta}\left( \vartheta^{\theta}(q), \left( G^{0}_{\theta} \right)^{-1}\left(\zeta; q,v^{*}(\theta)+\mathbb{E}^{-}(\vartheta^{\theta}(q)) \right) \right), \text{ for } s \geq \theta 
	\end{equation}
	converges as $\theta \to -\infty$ to the unique complete trajectory $w^{*}(\cdot)$ from $[v^{*}(\cdot)]$ such that $\Pi_{q} w^{*}(0)=\zeta$.
\end{lemma}
\begin{proof}
	By Corollary \ref{COR: GraphTransformHomeo}, the trajectory $v_{\theta}(\cdot)$ is well-defined. Since, by definition, $v_{\theta}(\theta) - v^{*}(\theta) \in \mathbb{E}^{-}(\vartheta^{\theta}(q))$, we have $V_{\vartheta^{\theta}(q)}(v_{\theta}(\theta) - v^{*}(\theta)) \leq 0$ and, thanks to $\textbf{(H3)}^{-}_{w}$, $V_{\vartheta^{s}(q)}(v_{\theta}(s) - v^{*}(s)) \leq 0$ for all $s \geq \tau_{Q}+\theta$. Thus, if a subsequence of $v_{\theta}(\cdot)$ converges (pointwise on $\mathbb{R}$) to some complete trajectory $w^{*}(\cdot)$, we must have $V_{\vartheta^{s}(q)}(w^{*}(s) - v^{*}(s)) \leq 0$ for all $s \in \mathbb{R}$, that is $w^{*}(\cdot) \in [v^{*}(\cdot)]$. Since $\Pi_{q} v_{\theta}(0) = \zeta$ for all $\theta \leq 0$, we also must have $w^{*}(0)=\zeta$. Since a complete trajectory passing through a given point and belonging to the class $[v^{*}(\cdot)]$ is unique (due to item 3) of Lemma \ref{LEM: CompleteTrajectoriesOrdering}), the only possible limit point of $v_{\theta}(\cdot)$ is $w^{*}(\cdot)$. 
	
	Thus, for our purposes it is sufficient to show that for any sequence $\theta=\theta_{k}$ such that $\theta_{k} \to -\infty$ as $k \to +\infty$ the sequence $v_{k}(\cdot) := v_{\theta_{k}}(\cdot)$ has a limit point. From $\textbf{(H3)}^{-}_{w}$ we have
	\begin{equation}
		\delta^{-1}_{Q}\|Q(q)\| \cdot \| \zeta - \Pi_{q} v^{*}(0) \|^{2}_{\mathbb{E}} \geq \int_{\theta_{k}}^{0}e^{2\alpha(s;q)}|v_{k}(s)-v^{*}(s)|^{2}_{\mathbb{H}}ds.
	\end{equation}
	Applying the mean value theorem to the above integral considered on $[-\tau_{S}-l-1,-\tau_{S}-l]$ for $l=1,2,\ldots$ and using \textbf{(ULIP)} in the same way we did in \eqref{EQ: ConePerturbationLipIntegralEstimate}, we get that 
	\begin{equation}
		\|v_{k}(-l)-v^{*}(-l)\|^{2}_{\mathbb{E}} \leq D^{2} e^{-2\alpha(-l;q)} \leq D^{2} e^{2 \alpha^{+} l},
	\end{equation}
    where
    \begin{equation}
    	D^{2} = \delta^{-1}_{Q} \cdot \|Q(q)\| \cdot \| \zeta - \Pi_{q} v^{*}(0) \|^{2}_{\mathbb{E}} \cdot L^{2}_{\tau_{S}+1} \cdot e^{2\operatorname{max}\{ |\alpha^{+}|, |\alpha^{-}| \} (\tau_{S}+1)}.
    \end{equation}
	This means that	$v_{k}(-l)$ lies in the closed ball $\mathcal{B}_{l}$ of radius $De^{\alpha^{+} l}$ centered at $v^{*}(-l)$. Applying Lemma \ref{LEM: GeneralCompactnessLemma} to $t_{l}=-l$, $\mathcal{B}_{l}$ and the sequence $v_{k}(\cdot)$ with $T_{k} := \theta_{k}$, we get the existence of a limit complete trajectory $w^{*}(\cdot)$. The proof is finished.
\end{proof}

We finalize this section with the following theorem.
\begin{theorem}
	\label{TH: HorizontalFibresAlongCompleteOrbits}
	Suppose $\textbf{(H3)}^{-}_{w}$, \textbf{(ULIP)}, \textbf{(ACOM)} with $\gamma^{+} = \alpha^{+}$ and \textbf{(PROJ)} are satisfied. Let $v^{*}(\cdot)$ be a complete trajectory over some $q \in \mathcal{Q}$. Then $[v^{*}(\cdot)](s)$ is a $C_{Lip}(\vartheta^{s}(q))$-Lipschitz $\varkappa_{0}$-admissible set over $\vartheta^{s}(q)$ for all $s \in \mathbb{R}$. Here $C_{Lip}(q)$ is defined in \eqref{EQ: LipschitzConstantOverQ} and $\varkappa_{0}$ is defined in Lemma \ref{LEM: ConePerturbation}.
\end{theorem}
\begin{proof}
	It is sufficient to consider only the case $s=0$. The $\varkappa_{0}$-admissibility follows from Lemma \ref{EQ: ProjectorHomeomorphismOntoImage} and Lemma \ref{LEM: ManifoldConstruction} since any space $\mathbb{E}^{-}(q)$ is $\varkappa_{0}$-admissible over $q$. To show that $\Pi_{q} \colon [v^{*}(\cdot)](0) \to \mathbb{E}^{-}(q)$ has a Lipschitz inverse with the prescribed Lipschitz constant one can use Lemma \ref{LEM: BackwardEstimate} analogously to the corresponding part of the proof in Lemma \ref{LEM: GraphTransform}.
\end{proof}

\subsection{Vertical foliation}
\label{SUBSEC: VerticalFoliation}

During this subsection we work with a \underline{semicocycle} $\psi$. Here we also relax \textbf{(H3)} by the following assumption.
\begin{description}
	\item[$\textbf{(H3)}^{+}_{w}$] There exists a family of operators $Q(q) \in \mathcal{L}(\mathbb{E};\mathbb{E}^{*})$, which satisfies \textbf{(H1)} with $P(q)$ changed to $Q(q)$, $\mathbb{E}^{-}(q)$ and $\mathbb{E}^{+}(q)$ changed to $\mathbb{E}^{-}_{Q}(q)$ and $\mathbb{E}^{+}_{Q}(q)$ respectively and \textbf{(H2)} with $j$ changed to $j_{Q}$. Moreover, for some constants $-\infty < \beta^{-} < \beta^{+} < +\infty$, $\delta_{Q}>0$, $\tau_{Q} \geq 0$ and a Borel measurable function $\beta_{0} \colon \mathcal{Q} \to \mathbb{R}$ such that $\beta^{-} \leq \beta_{0}(\cdot) \leq \beta^{+}$ and for $\beta(t;q) := \int_{0}^{t}\beta_{0}(\vartheta^{s}(q))ds$ and $V^{Q}_{q}(v):=\langle v, Q(q)v \rangle$ we have the inequality
	\begin{equation}
		\label{EQ: H3WeakPlus}
		\begin{split}
			e^{2\beta(r;q)}V^{Q}_{\vartheta^{r}(q)}(\psi^{r}(q,v_{1})-\psi^{r}(q,v_{2})) - V^{Q}_{q} (v_{1} - v_{2}) \leq \\ \leq -\delta_{Q} \int_{0}^{r} e^{2\beta(s;q)}|\psi^{s}(q,v_{1})-\psi^{s}(q,v_{2})|^{2}_{\mathbb{H}}ds 
		\end{split}
	\end{equation}
	satisfied for all $q \in \mathcal{Q}$, $r \geq \tau_{Q}$ and $v_{1},v_{2} \in \mathbb{E}$ such that $V^{Q}_{\vartheta^{r}(q)}(\psi^{r}(q,v_{1})-\psi^{r}(q,v_{2})) \geq 0$.
\end{description}
\begin{remark}
	It should be emphasized that when we consider $\textbf{(H3)}^{-}_{w}$ and $\textbf{(H3)}^{+}_{w}$ together (as in Theorem \ref{TH: ExpTrackingTheorem} below), we always assume that the operators $Q(q)$ and the spaces $\mathbb{E}^{-}_{Q}(q)$ and $\mathbb{E}^{-}_{Q}(q)$ in these assumptions coincide. So, both assumptions differ only in the pseudo-ordering and the functions $\alpha_{0}(\cdot)$ and $\beta_{0}(\cdot)$.
\end{remark}
\begin{remark}
	As in the previous subsection, we use the notation $V_{q}$, $\mathbb{E}^{-}(q)$ and $\mathbb{E}^{+}(q)$ instead of $V^{Q}_{q}$, $\mathbb{E}^{-}_{Q}(q)$ and $\mathbb{E}^{+}_{Q}(q)$ respectively and by $\Pi_{q}$ and $\Pi^{+}_{q}$ we denote the $V^{Q}_{q}$-orthogonal projector and its complementary one respectively. In particular, when referring below to assumptions \textbf{(PROJ)} and \textbf{(CD)} introduced in Section \ref{SUBSEC: DefinitionsAndHypotheses}, we always consider them under the notation introduced just now and the operators $P(q)$ changed to $Q(q)$.
\end{remark}

We say that two points $v_{1}, v_{2} \in \mathbb{E}$ are \textit{positively equivalent} over $q \in \mathcal{Q}$ if $V_{\vartheta^{t}(q)}(\psi^{t}(q,v_{1})-\psi^{t}(q,v_{2}))\geq 0$ for all $t \geq 0$. The following lemma is analogous to Lemma \ref{LEM: CompleteTrajectoriesOrdering}.
\begin{lemma}
	\label{LEM: PositiveEquivalence}
	Suppose $\textbf{(H3)}^{+}_{w}$ and \textbf{(S)} are satisfied. Let $v_{1},v_{2},v_{3} \in \mathbb{E}$ and $q \in \mathcal{Q}$ be fixed. Then we have
	\begin{enumerate}
		\item[1)] $v_{1}$ and $v_{2}$ are positively equivalent over $q$ if and only if for some $t \geq 0$ (and, consequently, all $t$) we have finiteness (and consequently, the estimate) of the integral
		\begin{equation}
			\label{EQ: VerticalFibreIntEstimate}
			\int_{t}^{+\infty}e^{2\beta(s;q)}| \psi^{s}(q,v_{1}) - \psi^{s}(q,v_{2}) |^{2}_{\mathbb{H}} \leq \delta^{-1}_{Q} e^{2\beta(t;q)} V(\psi^{t}(q,v_{1})-\psi^{t}(q,v_{2})) < +\infty.
		\end{equation}
	    \item[2)] If $v_{1}$ and $v_{2}$ are positively equivalent over $q$ and $v_{2}$ and $v_{3}$ are positively equivalent, then $v_{1}$ and $v_{3}$ are positively equivalent.
	    \item[3)] The positive equivalence of $v_{1}$ and $v_{2}$ over $q$ is continuous in $(v_{1},v_{2},q)$.
	\end{enumerate}
\end{lemma}
\begin{proof}
	The statement in item 1) can be proved analogously to item 1) of Lemma \ref{LEM: CompleteTrajectoriesOrdering}. Moreover, the statement in item 2) is a consequence of item 1) and the Minkowski inequality used analogously to item 2) of Lemma \ref{LEM: CompleteTrajectoriesOrdering}. Item 3) directly follows from the continuity of the cocycle. The proof is finished.
\end{proof}

Thus, the above introduced positive equivalence (over $q$) is an equivalence relation in $\mathbb{E}$. For each $q \in \mathcal{Q}$ and $v_{0} \in \mathbb{E}$ let $[v_{0}]^{+}(q)$ denote the set of all $v \in \mathbb{E}$ such that $v_{0}$ and $v$ are positively equivalent over $q$. Similarly to Lemma \ref{EQ: ProjectorHomeomorphismOntoImage}, one can obtain the following lemma.
\begin{lemma}
	\label{LEM: PiPlusHomeoOntoImage}
	Let the semicocycle $\psi$ satisfy $\textbf{(H3)}^{+}_{w}$. Then the map $\Pi^{+}_{q} \colon [v_{0}]^{+}(q) \to \mathbb{E}^{+}(q)$ is a homeomorphism onto its image.
\end{lemma}

Now our aim is to prove Theorem \ref{TH: ExpTrackingTheorem} below. In particular, it states that under $\textbf{(H3)}^{-}_{w}$ and $\textbf{(H3)}^{+}_{w}$ any admissible set provides a section of the positive equivalence, i.~e. any such set contains exactly one point from each equivalence class. This immediately leads to the following corollary.
\begin{corollary}
	\label{COR: VerticalLeafStructure}
	Under the hypotheses of Theorem \ref{TH: ExpTrackingTheorem}, for any $q \in \mathcal{Q}$ and $v_{0} \in \mathbb{E}$ the map $\Pi^{+}_{q} \colon [v_{0}]^{+}(q) \to \mathbb{E}^{+}(q)$ is a homeomorphism.
\end{corollary}
\begin{proof}
	Indeed, by Lemma \ref{LEM: PiPlusHomeoOntoImage} the map $\Pi^{+}_{q}$ provides a homeomorphism between $[v_{0}]^{+}(q)$ and its image in $\mathbb{E}^{+}(q)$. But this image must coincide with $\mathbb{E}^{+}(q)$ due to the first part of Theorem \ref{TH: ExpTrackingTheorem}, which states that for $ \mathcal{M} = \mathbb{E}^{-}(q) + \zeta$ the intersection $[v_{0}]^{+}(q) \cap \mathcal{M}$ contains exactly one point for any $\zeta \in \mathbb{E}^{+}(q)$. 
\end{proof}

Now the theorem can be stated as follows.
\begin{theorem}
	\label{TH: ExpTrackingTheorem}
	Let the semicocycle $\psi$ satisfy $\textbf{(H3)}^{-}_{w}$, $\textbf{(H3)}^{+}_{w}$, \textbf{(ULIP)} and \textbf{(PROJ)} (for the $V_{q}$-orthogonal projectors $\Pi_{q}$). Let $\mathcal{M}$ be a $\varkappa$-admissible set over $q$ for some $\varkappa \in (0,1)$. Then for every $v_{0} \in \mathbb{E}$ there exists a unique $v^{*}_{0} \in \mathcal{M}$ such that $v^{*}_{0} \in [v_{0}]^{+}(q)$. Moreover, if $\mathcal{M}$ is $L(\mathcal{M})$-Lipschitz $\varkappa $-admissible, then there exists a function $R \colon \mathbb{R}^{3}_{+} \to \mathbb{R}_{+}$, which is non-decreasing in each variable, and such that
	\begin{equation}
		\label{EQ: ExponentialTrackingEstimate}
		\| \psi^{t}(q,v_{0})-\psi^{t}(q,v^{*}_{0}) \|_{\mathbb{E}} \leq R(C^{-1}_{q},(1-\varkappa)^{-1}, L(\mathcal{M})) \operatorname{dist}(v_{0},\mathcal{M}) e^{-\beta(t;q)} \text{ for all } t \geq 0.
	\end{equation}
\end{theorem}
Before giving a proof, we need some preparations. Our approach for the proof is inspired by the work of A.V.~Romanov \cite{Romanov1994}, namely, Lemma 5 therein and its further use. To adapt his ideas, we recall the quadratic forms $V^{(\varkappa)}_{q}(v)  = V_{q}(v^{+}) + \varkappa^{2}V_{q}(v^{-})$ used in the Cone Perturbation Lemma (Lemma \ref{LEM: ConePerturbation}). From this we define the cones $\mathcal{C}^{-}_{(\varkappa)}(q) := \{ V^{(\varkappa)}_{q}(v) \leq 0 \}$ for $\varkappa \in (0,1)$, which are contained in the set $\mathcal{C}^{-}(q) = \{ V_{q}(v) \leq 0 \}$ by definition.

Note that any $\varkappa$-admissible over $q$ set $\mathcal{M}$ (for $\varkappa \in (0,1)$)  lies in the cone $\mathcal{C}^{-}_{(\varkappa)}(q)$ translated to any of its point by definition. The proof of Theorem \ref{TH: ExpTrackingTheorem} will be based on this observation and the following lemma (see Fig. \ref{FIG: RomanovLemma}).
\begin{figure}
	\centering
	\includegraphics[width=1\linewidth]{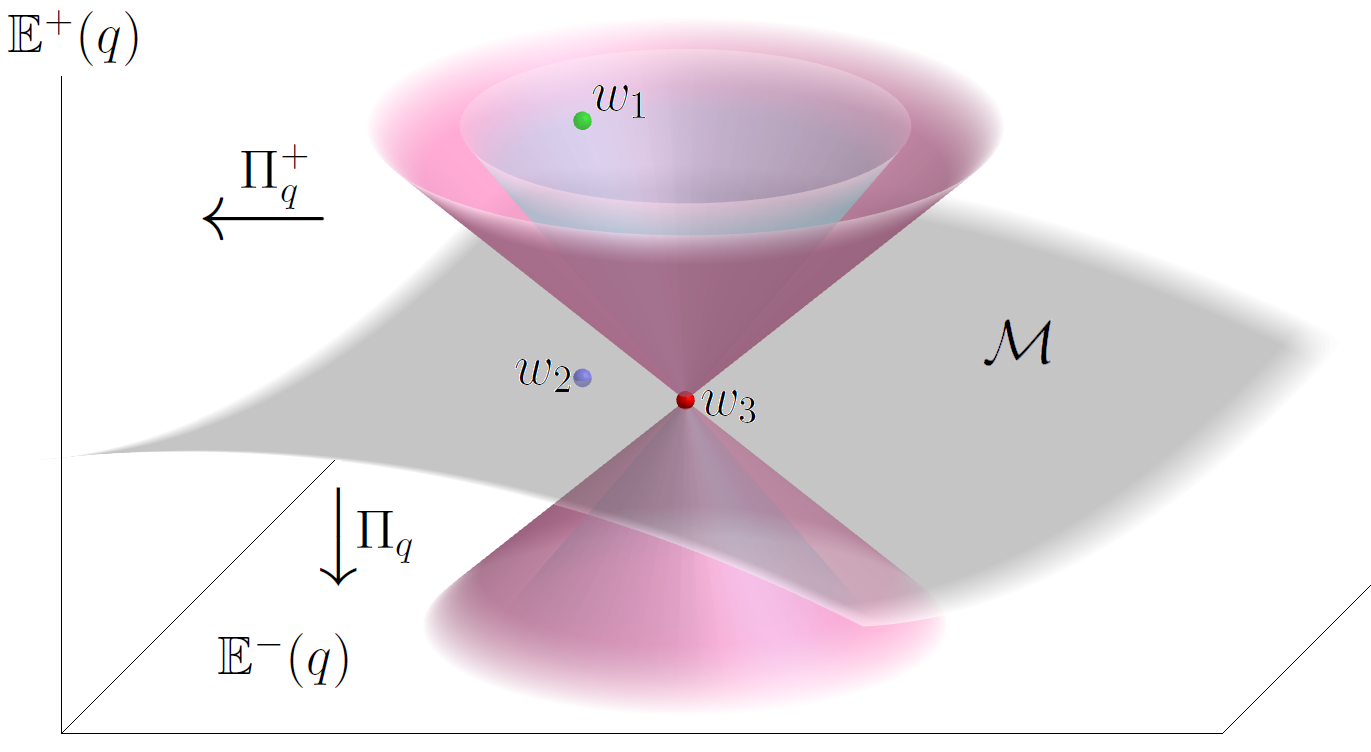}
	\caption{An illustation to Lemma \ref{LEM: RomanovLemma} and Theorem \ref{TH: ExpTrackingTheorem}. Here $\mathcal{M}$ (gray) lies in the cone $\mathcal{C}^{-}_{(\varkappa)}(q)$ translated to $w_{3}$, which is the exterior part of the pink cone. The unperturbed cone $\mathcal{C}^{-}(q)$ translated to $w_{3}$ is given by the exterior part of the cyan cone.}
	\label{FIG: RomanovLemma}
\end{figure}

\begin{lemma}[Romanov's lemma]
	\label{LEM: RomanovLemma}
	Let $q \in \mathcal{Q}$ and $\varkappa \in (0,1)$ be fixed. Suppose $w_{1},w_{2},w_{3} \in \mathbb{E}$ are given such that $V_{q}(w_{1}-w_{3}) \geq 0$, $V^{(\varkappa)}_{q}(w_{2}-w_{3}) \leq 0$ and $\Pi_{q}w_{1}=\Pi_{q}w_{2}$. Then
	\begin{equation}
		V_{q}(w^{+}_{1}-w^{+}_{3}) \leq C_{\varkappa} V_{q}(w^{+}_{1} - w^{+}_{2}),
	\end{equation}
	where $C_{\varkappa} = 1/(1-\varkappa)^{2}$ and $w_{i}=w^{+}_{i} + w^{-}_{i}$ is the unique decomposition with $w^{+}_{i} \in \mathbb{E}^{+}(q)$ and $w^{-}_{i} \in \mathbb{E}^{-}(q)$ for $i=1,2,3$.
\end{lemma}
\begin{proof}
	From $V^{(\varkappa)}_{q}(w_{2}-w_{3}) \leq 0$ we have $V_{q}(w^{+}_{2}-w^{+}_{3}) \leq - \varkappa^{2}V_{q}(w^{-}_{2}-w^{-}_{3})$. Since $\Pi_{q}w^{-}_{1}=w^{-}_{2}=\Pi_{q}w_{2}$ and $V_{q}(w_{1}-w_{3}) \geq 0$, we have $-V_{q}(w^{-}_{2}-w^{-}_{3}) \leq V_{q}(w^{+}_{1} - w^{+}_{3})$. Now this and the triangle inequality gives
	\begin{equation}
		\begin{split}
			\left( V_{q}(w^{+}_{1}-w^{+}_{3}) \right)^{1/2} &\leq \left( V_{q}(w^{+}_{1}-w^{+}_{2}) \right)^{1/2} + \left( V_{q}(w^{+}_{2}-w^{+}_{3}) \right)^{1/2} \leq \\ &\leq \left( V_{q}(w^{+}_{1}-w^{+}_{2}) \right)^{1/2} + \varkappa \left(-V_{q}(w^{-}_{2}-w^{-}_{3}) \right)^{1/2} \leq \\ &\leq 
			\left( V_{q}(w^{+}_{1}-w^{+}_{2}) \right)^{1/2} + \varkappa \left( V_{q}(w^{+}_{1} - w^{+}_{3})\right)^{1/2}.
		\end{split}
	\end{equation}
	Now the conclusion of the lemma is obvious.
\end{proof}

\begin{proof}[Proof of Theorem \ref{TH: ExpTrackingTheorem}]
	Let $v_{0} \in \mathbb{E}$ and $q \in \mathcal{Q}$ be fixed. Put $v(s):=\psi^{s}(q,v_{0})$ for $s \geq 0$. By Lemma \ref{LEM: GraphTransform}, for every $\theta \geq \max\{\tau_{Q},\tau_{S}+1\}$ there exists a unique trajectory over $q$ defined for $t \geq 0$, say $v^{*}_{\theta}(\cdot)$, such that $v^{*}_{\theta}(0) \in \mathcal{M}$ and $\Pi_{\vartheta^{\theta}(q)}v^{*}_{\theta}(\theta) = \Pi_{\vartheta^{\theta}(q)} v(\theta)$. Suppose there exists a sequence $\theta=\theta_{k}$, where $k=1,2,\ldots$, such that $v^{*}_{\theta_{k}}(0)$ converges to some $v^{*}_{0} \in \mathcal{M}$.  Put $v^{*}(s)=\psi^{s}(q,v^{*}_{0})$ for $s \geq 0$. Clearly, $v^{*}_{\theta_{k}}(s)$ converges to $v^{*}(s)$ for all $s \geq 0$. Since $V_{\vartheta^{\theta_{k}}(q)}(v^{*}_{\theta_{k}}(\theta_{k}) - v(\theta_{k}) ) \geq 0$, we must have $V_{\vartheta^{s}(q)}(v_{\theta_{k}}(s) - v(s) ) \geq 0$ for all $s \in [0,\theta_{k} - \tau_{Q}]$ due to $\textbf{(H3)}^{-}_{w}$. Thus, $V_{\vartheta^{s}(q)}( v(s) - v^{*}(s) ) \geq 0$ for all $s \geq 0$. Then, from \eqref{EQ: VerticalFibreIntEstimate} and \textbf{(ULIP)} analogously to Lemma \ref{LEM: BackwardEstimate}, we get for all $t \geq \tau_{S} + 1$ that
	\begin{equation}
		\label{EQ: ExponentialTrackingDifferenceEstimate1}
		\| v(t) - v^{*}(t) \|_{\mathbb{E}} \leq  L(Q) \cdot e^{-\beta(t;q)} \cdot \| v(0) - v^{*}(0) \|_{\mathbb{E}},
	\end{equation}
    where $L(Q) := L_{\tau_{S}+1}\sqrt{\delta_{Q} M_{Q}} \cdot \operatorname{exp}\left(\operatorname{max}\{|\beta^{+}|,|\beta^{-}|\}| (\tau_{S}+1)\right)$ and $\|Q(q)\| \leq M_{Q}$ for all $q \in \mathcal{Q}$.
	It is also clear (due to \textbf{(ULIP)}) that the above inequality will be satisfied for all $t \geq 0$ with a possible change of the constant if necessary.
	
	Now our aim is to show the existence of $\theta = \theta_{k}$ as above. Let $\Phi_{q}$ be the inverse to $\Pi_{q} \colon \mathcal{M} \to \mathbb{E}^{-}(q)$ map. Put $w_{1} := v_{0}$, $w_{2} := \Phi_{q}( \Pi_{q}v_{0} )$ and $w_{3} := v^{*}_{\theta}(0)$. We suppose that $\theta \geq \tau_{Q}$. Then $V_{q}(w_{1}-w_{3}) \geq 0$ as it was shown above. Since $\mathcal{M}$ is admissible over $q$, there exists $\varkappa \in (0,1)$ such that $V^{(\varkappa)}_{q}(v_{1}-v_{2}) \leq 0$ for any $v_{1},v_{2} \in \mathcal{M}$. Thus, we have $V^{(\varkappa)}_{q}(w_{2}-w_{3}) \leq 0$ since $w_{2},w_{3} \in \mathcal{M}$. It is also clear that $\Pi_{q}w_{1}=\Pi_{q}w_{2}$ (see Fig. \ref{FIG: RomanovLemma}). Thus, Lemma \ref{LEM: RomanovLemma} gives us the inequality
	\begin{equation}
		\label{EQ: ExponentialTrackingAfterRomanovLemma}
		\begin{split}
			V_{q}(\Pi^{+}_{q}( v(0) - v^{*}_{\theta}(0) )) \leq \frac{1}{(1-\varkappa)^{2}} V_{q}(\Pi^{+}_{q}( v(0) - \Phi_{q}(\Pi_{q}v(0))) ).
		\end{split}
	\end{equation}
	Since $V_{q}(v(0) - v^{*}_{\theta}(0) ) \geq 0$, we have
	\begin{equation}
		\begin{split}
			-V_{q}\left(\Pi_{q} (v(0)-v^{*}_{\theta}(0)) \right) \leq V_{q}(\Pi^{+}_{q}( v(0) - v^{*}_{\theta}(0) ))
		\end{split}
	\end{equation}
	and using the constant $C_{q}$ from \eqref{EQ: ConstantProjectorNorm}, we get
	\begin{equation}
		\label{EQ: ExponentialTrackingProjectedEstimate}
		\| \Pi_{q}(v(0) - v^{*}_{\theta}(0)) \|^{2}_{\mathbb{E}} \leq  C^{-2}_{q} V_{q}(\Pi^{+}_{q}( v(0) - v^{*}_{\theta}(0) )).
	\end{equation}
	Thus, \eqref{EQ: ExponentialTrackingAfterRomanovLemma} shows that $\Pi_{q}v^{*}_{\theta}(0)$ is bounded for all $\theta \geq \tau_{Q}$ and, consequently, there exists a sequence $\theta=\theta_{k}$ such that $\Pi_{q}v^{*}_{\theta_{k}}(0)$ converges as $k \to +\infty$. From this we get that $v^{*}_{\theta}(0) = \Phi_{q}(\Pi_{q} v^{*}_{\theta}(0)) \in \mathcal{M}$ also converges to some $v^{*}_{0}$. Thus, there exists at least one $v^{*}_{0}$ as it is required. Since $\mathcal{M}$ is admissible, such $v^{*}_{0}$ is unique and, in fact, $v^{*}_{\theta}(0) \to v^{*}_{0}$ as $\theta \to +\infty$.
	
	To prove \eqref{EQ: ExponentialTrackingEstimate}, we have to estimate the norm $\|v_{0} - v^{*}_{0}\|_{\mathbb{E}}$ in \eqref{EQ: ExponentialTrackingDifferenceEstimate1}. For this we pass to the limit in the first inequality from \eqref{EQ: ExponentialTrackingAfterRomanovLemma} as $\theta \to +\infty$. Thus,
	\begin{equation}
		V_{q}(\Pi^{+}_{q}( v_{0} - v^{*}_{0} )) \leq \frac{1}{(1-\varkappa)^{2}} V_{q}(\Pi^{+}_{q}( v_{0} - \Phi_{q}(\Pi_{q}v_{0})) ).
	\end{equation}
	From this and since
	\begin{equation}
		C^{2}_{q}\| \Pi_{q} (v_{0} - v^{*}_{0}) \|^{2}_{\mathbb{E}} \leq -V_{q}(\Pi_{q} (v_{0} - v^{*}_{0})) \leq V_{q}( \Pi^{+}_{q}(v_{0} - v^{*}_{0}) ),
	\end{equation}
	we get
	\begin{equation}
		\| \Pi_{q} (v_{0} - v^{*}_{0}) \|^{2}_{\mathbb{E}} \leq \frac{C^{-2}_{q}}{(1-\varkappa)^{2}}V_{q}(\Pi^{+}_{q}( v_{0} - \Phi_{q}(\Pi_{q}v_{0})) ).
	\end{equation}
	Since $\mathcal{M}$ is Lipschitz admissible, the map $\Phi_{q}$ is Lipschitz with the Lipschitz constant $L(\mathcal{M})$. Let $\widetilde{v} \in \mathcal{M}$ be any point. We have
	\begin{equation}
		\| v^{*}_{0} - \widetilde{v} \|_{\mathbb{E}} = \| \Phi_{q}(\Pi_{q}v^{*}_{0}) - \Phi_{q}(\Pi_{q} \widetilde{v})  \|_{\mathbb{E}} \leq L(\mathcal{M})\| \Pi_{q} v^{*}_{0} - \Pi_{q} v_{0} \|_{\mathbb{E}} + L(\mathcal{M}) \| \Pi_{q} v_{0} - \Pi_{q} \widetilde{v} \|_{\mathbb{E}}.
	\end{equation}
	It is clear that $\| \Pi_{q} v_{0} - \Pi_{q} \widetilde{v} \|_{\mathbb{E}} \leq M_{\Pi} \| v_{0} - \widetilde{v} \|_{\mathbb{E}}$. Moreover,
	\begin{equation}
		\|  v_{0} - \Phi_{q}(\Pi_{q}v_{0}) \|_{\mathbb{E}} \leq \|v_{0} - \widetilde{v}\|_{\mathbb{E}} + \| \widetilde{v} - \Phi_{q}(\Pi_{q}v_{0})\|_{\mathbb{E}} \leq (L(\mathcal{M}) M_{\Pi} + 1) \cdot \|v_{0} - \widetilde{v}\|_{\mathbb{E}}.
	\end{equation}
	Summarizing the above estimates, we have
	\begin{equation}
		\label{EQ: ExpTrackPointEstimateFinal}
		\| v_{0} - v^{*}_{0} \|_{\mathbb{E}} \leq \| v_{0} - \widetilde{v} \|_{\mathbb{E}} + \|v^{*}_{0} - \widetilde{v}\|_{\mathbb{E}} \leq L' \| v_{0} - \widetilde{v} \|_{\mathbb{E}},
	\end{equation}
	where 
	\begin{equation}
		L' = \frac{L(\mathcal{M}) C^{-1}_{q} \sqrt{M_{Q}}}{1-\varkappa} (1+M_{\Pi}) (L(\mathcal{M}) M_{\Pi} + 1) + L(\mathcal{M}) M_{\Pi} + 1.
	\end{equation}
	Taking $\widetilde{v} \in \mathcal{M}$ such that $\| v_{0} - \widetilde{v} \|_{\mathbb{E}}$ becomes arbitrarily close to $\operatorname{dist}(v_{0},\mathcal{M})$, we get \eqref{EQ: ExponentialTrackingEstimate} with 
	\begin{equation}
		R(C^{-1}_{q},(1-\varkappa)^{-1},L(\mathcal{M})) := L(Q) \cdot L' \cdot L_{\tau_{S}+1} \cdot e^{\max\{ |\beta^{+}|, |\beta^{-}| \} \cdot (\tau_{S}+1)}.
	\end{equation}
    Thus, the proof is finished.
\end{proof}

Now we are going to study the continuity of the point $v^{*}_{0}$ from Theorem \ref{TH: ExpTrackingTheorem} on $v_{0}$ and $q \in \mathcal{Q}$. For each $q \in \mathcal{Q}$ let $\mathcal{M}(q)$ be a Lipschitz $\varkappa(q)$-admissible subset in $\mathbb{E}$. Let us define the map $\mathbb{E} \times \mathcal{Q} \to \mathbb{E}$ by $\Pi^{c}_{q}(v_{0}) := v^{*}_{0}$, where $v^{*}_{0}$ is the unique point from Theorem \ref{TH: ExpTrackingTheorem} applied to $v_{0}$ and $q$ and $\mathcal{M} = \mathcal{M}(q)$. We say that $\Pi^{c}_{q}$ is the \textit{central projector} onto $\mathcal{M}(q)$.

We recall that under \textbf{(CD)} (note that we consider $Q(q)$ instead of $P(q)$) the $V_{q}$-orthogonal projector $\Pi_{q}$ depend continuously on $q$ (see Lemma \ref{LEM: VOrtogonalProjectorsContinuity}). Moreover, it is clear that the set $\bigcup_{q \in \mathcal{K}} \left( \mathbb{E}^{-}(q) \cap \mathcal{B} \right)$ is compact for any bounded subset $\mathcal{B} \subset \mathbb{E}$ and a compact subset $\mathcal{K} \subset \mathcal{Q}$. Thus, $C^{-1}_{q}$ are bounded from above for $q$ from a fixed compact $\mathcal{K}$.

Let us put $\mathcal{D} := \bigcup_{q \in \mathcal{Q}}\{ q \} \times \mathbb{E}^{-}(q)$ that is a closed subset of $\mathcal{Q} \times \mathbb{E}$ under \textbf{(CD)}. We will say that the family $\{ \mathcal{M}(q) \}_{q \in \mathcal{Q}}$ is \textit{continuous} if the maps $\Phi_{q}$ given by the inverse to $\Pi_{q} \colon \mathcal{M}(q) \to \mathbb{E}^{-}(q)$ are such that the map $\mathcal{D} \ni (q,\zeta) \mapsto \Phi_{q}(\zeta) \in \mathbb{E}$ is continuous.

\begin{theorem}
	\label{TH: CentralProjectorContinuityGeneral}
	Under the standing hypotheses of Theorem \ref{TH: ExpTrackingTheorem} and \textbf{(CD)}, let $\{ \mathcal{M}(q) \}$ by a continuous family of $L(\mathcal{M}(q))$-Lipschitz $\varkappa(q)$-admissible subsets, where $\varkappa(q) \in (0,1)$, and let $\Pi^{c}_{q}$ be the central projector onto $\mathcal{M}(q)$. Suppose that the values $(1-\varkappa(q))^{-1}$ and $L(\mathcal{M}(q))$ are locally bounded from above. Then the map $\mathcal{Q} \times \mathbb{E} \ni (q,v) \to \Pi^{c}_{q}(v)$ is continuous. Moreover, if $q_{k} \to q$, then $\Pi^{c}_{q_{k}}(v) \to \Pi^{c}_{q}(v)$ uniformly in $v$ from bounded sets.
\end{theorem}
\begin{proof}
	Let a sequence $q_{k}$ converge to $q$ and $v_{k} \in \mathbb{E}$ converge to $v \in \mathbb{E}$. Put $v^{*}_{k} := \Pi^{c}_{q_{k}}(v_{k})$ and $v^{*} := \Pi^{c}_{q}(v)$. We have to show that $v^{*}_{k} \to v^{*}$. Supposing the opposite, we obtain a sebsequence (we keep the same index) separated from $v^{*}$. From $\|v^{*}_{k} - v \|_{\mathbb{E}} \leq \| v^{*}_{k} - v_{k} \| + \| v_{k} - v \|$ and \eqref{EQ: ExponentialTrackingEstimate} applied for $\mathcal{M}=\mathcal{M}(q)$ and $q=q_{k}$ we get that the sequence $\|v^{*}_{k}\|$ is bounded. In particular, it is precompact since $\mathcal{M}(q)$ continuously depend on $q$ and, moreover, any of its limit points $\overline{v}$ lies in $\mathcal{M}(q)$. Since $v^{*}_{k} \in [v_{k}]^{+}(q_{k})$ and the positive equivalence is continuous (see item 3) of Lemma \ref{LEM: PositiveEquivalence}), we must have $\overline{v} \in [v]^{+}(q)$. But such a point in $\mathcal{M}(q)$ is unique and, consequently, $\overline{v} = v^{*}$ that leads to a contradiction. Note that the same arguments work in the case $v_{k}$ belong to a bounded subset $\mathcal{B}$. Thus, the proof is finished.
\end{proof}

In what follows, our main nontrivial example of a continuous family $\mathcal{M}(q)$ will be given by the principle horizontal leaves $\mathfrak{A}(q)$, which are also invariant and forms an inertial manifold for the cocycle when $\beta^{-} > 0$.

\subsection{$C^{1}$-differentiability of leaves}
\label{SUBSEC: DifferentiabilityLeaves}
In this section we are going to prove that the horizontal leaves constructed in Section \ref{SUBSEC: HorizontalLeaves} (and known to be finite-dimensional Lipschitz submanifolds in $\mathbb{E}$) are $C^{1}$-differentiable (in the Fr\'{e}chet sense) submanifolds in $\mathbb{E}$ under reasonable differentiability assumptions for the semicocycle. The key observation is that linearization of the semicocycle provides a linear semicocycle which inherits monotonicity conditions like \textbf{(H3)}. Thus, we can apply previous results to this linear cocycle and construct horizontal leaves (now being linear subspaces) around the zero trajectory. These leaves are tangent spaces at proper points and the only work to be done is to justify this.

Analogously, we show that the vertical leaves constructed in Section \ref{SUBSEC: VerticalFoliation} are $C^{1}$-differentiable (in the Fr\'{e}chet sense) Banach $\mathbb{E}^{+}(q)$-submanifolds of $\mathbb{E}$. From this we will deduce the normal hyperbolicity property and state a conjecture on differentiability of higher orders.

\subsubsection{Fr\'{e}chet differentiability of horizontal leaves}
\label{SUBSUBSEC: FrechetHorizontal}

Throughout this section we are given with a semicocycle $\psi$ in $\mathbb{E}$ over a semiflow $\vartheta$ in $\mathcal{Q}$, which satisfies $\textbf{(H3)}^{-}_{w}$, \textbf{(ULIP)} and \textbf{(PROJ)}. Let $\pi^{t}(q,v):=(\vartheta^{t}(q),\psi^{t}(q,v))$ be the skew-product semiflow on $\mathcal{Q} \times \mathbb{E}$ associated with the semicocycle. We consider a complete trajectory $v^{*}(\cdot)$ of the semicocycle over some $q \in \mathcal{Q}$. Recall that this means that we also fix a complete trajectory of the semiflow $\vartheta$ passing through $q$, which we denote (for convenience) by $\vartheta^{t}(q)$. Thus, the map $t \mapsto (\vartheta^{t}(q),v^{*}(t))$ is a complete trajectory of the semiflow $\pi$ passing through $(q,v^{*}(0))$.

Recall that we omit mentioning dependence on operators $Q(q)$ for spaces, projectors and quadratic forms according to Remark \ref{REM: NotationQomited}.

From Theorem \ref{TH: HorizontalFibresAlongCompleteOrbits} we get that the set $\mathfrak{H} := [v^{*}(\cdot)](0)$ (defined in \eqref{EQ: EquivalenceSectionDefinition}) is a $j$-dimensional Lipschitz admissible submanifold in $\mathbb{E}$. Recall that this means $V_{q}(v_{1}-v_{2}) < 0$ for any distinct $v_{1},v_{2} \in \mathfrak{H}$, the map $\Pi_{q} \colon \mathfrak{H} \to \mathbb{E}^{-}(q)$ is a homeomorphism and the inverse map $\Phi_{\mathfrak{H}} \colon \mathbb{E}^{-}(q) \to \mathfrak{H}$ is Lipschitz (with the Lipschitz constant given by \eqref{EQ: LipschitzConstantOverQ}).

Our aim is to show that the map $\Phi_{\mathfrak{H}}$ is $C^{1}$-differentiable in the sense of Fr\'{e}chet as a $\mathbb{E}$-valued map and to describe its differential and tangent spaces for $\mathfrak{H}$. This will be done under the following differentiability assumption for the semicocycle. 

\begin{description}
	\item[(DIFF)] There exists a linear semicocycle in $\mathbb{E}$ over the skew-product semiflow $( \mathcal{Q} \times \mathbb{E}, \pi )$ given by the family of linear maps $\Xi^{t}(q,v,\cdot) \colon \mathbb{E} \to \mathbb{E}$, where $q \in \mathcal{Q}$, $v \in \mathbb{E}$ and $t \geq 0$, which satisfies \textbf{(ULIP)} and \textbf{(ACOM)} with $\gamma^{+}=\alpha^{+}$, where $\alpha^{+}$ is from $\textbf{(H3)}^{-}_{w}$, and such that for any $q \in \mathcal{Q}$, $v_{0} \in \mathbb{E}$ and $\xi \in \mathbb{E}$ the limit
	\begin{equation}
		\lim_{h \to 0} \frac{\psi^{t}(q,v_{0} + h \xi) - \psi^{t}(q,v_{0}) }{h} = \Xi^{t}(q,v_{0},\xi).
	\end{equation}
    is understood in the following two senses:
    \begin{enumerate}
    	\item For $t \geq 0$ the limit exists in the norm of $\mathbb{H}$ and holds uniformly in $t$ from compact subsets of $[0,+\infty)$ and $\xi$ from bounded subsets of $\mathbb{E}$;
    	\item For $t \geq \tau_{Q}$ the limit exists in the norm of $\mathbb{E}$ and holds uniformly in $t$ from compact subsets of $[\tau_{Q},+\infty)$ and $\xi$ from bounded subsets of $\mathbb{E}$.
    \end{enumerate}
    Moreover, for any $t \geq 0$, $q \in \mathcal{Q}$ the map $\mathbb{E} \ni v \mapsto \Xi^{t}(q,v,\cdot) \in \mathcal{L}(\mathbb{E};\mathbb{H})$ is continuous.
\end{description}

\begin{remark}
	This hypothesis is the main reason why had to consider the value $\tau_{Q} \geq 0$ in previous developments of the theory. Such a case is typical for delay equations, where differentiability in the main space $\mathbb{E}$ is delayed as in item (2) of \textbf{(DIFF)}.
\end{remark}

To emphasize that the map $\Xi^{t}(q,v,\cdot)$ is linear we will usually write $L(t;q,v)\xi$ instead of $\Xi^{t}(q,v,\xi)$.

\begin{lemma}
	\label{LEM: H3WeakMinusLinearized}
	For every $q \in \mathcal{Q}$ and $v \in \mathbb{E}$ we have
	\begin{equation}
		\label{EQ: H3WeakMinusLinearized}
		\begin{split}
			e^{2\alpha(r;q)}V_{\vartheta^{r}(q)}(L(r;q,v)\xi) - V_{q}(\xi) \leq -\delta_{Q} \int_{0}^{r} e^{2\alpha(s;q)}|L(s;q,v)\xi|^{2}_{\mathbb{H}}ds 
		\end{split}
	\end{equation}
    satisfied for all $\xi \in \mathbb{E}$ such that $V_{q}(\xi) \leq 0$ and all $r \geq \tau_{Q}$.
\end{lemma}
\begin{proof}
	To prove the lemma one just have to consider \eqref{EQ: H3WeakMinus} with $v_{1}=v+h\xi$ and $v_{2} = v$, divide both sides by $h^{2}$ and pass to the limit as $h \to 0$ using \textbf{(DIFF)}.
\end{proof}

Thus, the cocycle $\Xi$ satisfies the main condition $\textbf{(H3)}^{-}_{w}$ of Section \ref{SUBSEC: HorizontalLeaves} w.~r.~t. the same operators $Q(q)$, spaces and projectors, which are independent of $v$.

Let $0^{*}(\cdot)$ be the zero complete trajectory of $\Xi$ over a given point $(q,v_{0})$, where $v_{0} \in \mathfrak{H}$. The corresponding complete trajectory of the skew-product semiflow is given by $(\vartheta^{t}(q),w^{*}(t))$, where $w^{*}(\cdot) \in [v^{*}(\cdot)]$ is the unique complete trajectory such that $w^{*}(0)=v_{0}$ (see Theorem \ref{TH: HorizontalFibresAlongCompleteOrbits}). Then the set $\mathfrak{H}'(v_{0}) := [0^{*}(\cdot)](0)$ is a $j$-dimensional admissible subspace of $\mathbb{E}$. In particular, this means that $\Pi_{q} \colon \mathfrak{H}'(v_{0}) \to \mathbb{E}^{-}(q)$ is a homeomorphism and the inverse map $\Phi'_{\mathfrak{H}}(v_{0}) \colon \mathbb{E}^{-}(q) \to \mathfrak{H}'(v_{0})$ is Lipschitz.

In the lemmas below we consider $\Phi_{\mathfrak{H}}$ and $\Phi'_{\mathfrak{H}}(v_{0})$ as $\mathbb{E}$-valued maps.

\begin{lemma}
	\label{LEM: DifferentialPhiContinuous}
	The map $\Phi'_{\mathfrak{H}}(v_{0})$ depend continuously on $v_{0} \in \mathfrak{H}$ in the norm of $\mathcal{L}(\mathbb{E}^{-}(q);\mathbb{E})$.
\end{lemma}
\begin{proof}
	Supposing the opposite, we obtain a sequence $v_{k} \in \mathfrak{H}$ converging to some $v_{0} \in \mathfrak{H}$ and a sequence $\eta_{k} \in \mathbb{E}^{-}(q)$ such that $\| \eta_{k} \| = 1$, for which the sequence
	\begin{equation}
		\Phi'_{\mathfrak{H}}(v_{k})\eta_{k} - \Phi'_{\mathfrak{H}}(v_{0}) \eta_{k}
	\end{equation}
    is uniformly separated from zero. Without loss of generality, we may assume that $\eta_{k}$ converges to some $\overline{\eta}$. Let $\xi^{*}_{k}(\cdot)$ and $\xi^{*}(\cdot)$ be complete trajectories of the linearization semicocycle $\Xi$ over $(q,v_{k})$ and $(q,v_{0})$ respectively from the equivalence class of $0^{*}(\cdot)$ and such that $\xi^{*}_{k}(0) = \Phi'_{\mathfrak{H}}(v_{k})\eta_{k}$ and $\xi^{*}(0) = \Phi'_{\mathfrak{H}}(v_{0}) \overline{\eta}$. We have
    \begin{equation}
    	\int_{-\infty}^{0}e^{2 \alpha(s;q)} |\xi^{*}_{k}(s)|^{2}_{\mathbb{H}} ds \leq -\delta^{-1}_{Q} V_{q}(\xi^{*}_{k}(s)) \leq \delta^{-1}_{Q} M_{Q} \| \eta_{k} \| = \delta^{-1}_{Q} M_{Q}
    \end{equation}
    From this, as in the proof of Lemma \ref{LEM: ManifoldConstruction}, we may obtain a subsequence (we keep the same index) such that $\xi^{*}_{k}(s) \to \overline{\xi}(s)$ in $\mathbb{E}$ for all $s \in \mathbb{R}$ for some complete trajectory $\overline{\xi}(\cdot)$ over $(q,v_{0})$. In particular,
    \begin{equation}
    	\int_{-\infty}^{0}e^{2 \alpha(s;q)} |\overline{\xi}(s)|^{2}_{\mathbb{H}} ds \leq \delta^{-1}_{Q} M_{Q} < +\infty
    \end{equation}
    and, consequently, $\overline{\xi}(0) \in \mathfrak{H}'(v_{0})$. Since $\Pi_{q}\left(\Phi'_{\mathfrak{H}}(v_{k})\eta_{k} - \Phi'_{\mathfrak{H}}(v_{0}) \eta_{k} \right) = 0$, we must have $\Pi_{q} \overline{\xi}(0) = \Pi_{q}\Phi'_{\mathfrak{H}}(v_{0}) \overline{\eta} = \overline{\eta}$ that leads to a contradiction. The lemma is proved.
\end{proof}

\begin{lemma}
	\label{LEM: DifferencePrecompactPhi}
	Consider two sequences $h_{k} \in \mathbb{R}$ and $\eta_{k} \in \mathbb{E}^{-}(q)$ such that $h_{k} \to 0$ as $k \to +\infty$ and $\eta_{k}$ is bounded. Then for any $\zeta \in \mathbb{E}^{-}(q)$ the sequence
	\begin{equation}
		\label{EQ: LemDifferencePrecompact}
		\frac{\Phi_{\mathfrak{H}}(\zeta+h_{k}\eta_{k})-\Phi_{\mathfrak{H}}(\zeta)}{h_{k}}
	\end{equation}
    is precompact in $\mathbb{E}$ and any of its limit points lies in $\mathfrak{H}'(\Phi_{\mathfrak{H}}(\zeta))$.
\end{lemma}
\begin{proof}
	Let $w^{*}_{k}(\cdot)$ and $w^{*}(\cdot)$ be two complete trajectories of the semicocycle $\psi$ over $q$ such that $w^{*}_{k}(0) = \Phi_{\mathfrak{H}}(\zeta+h_{k}\eta_{k})$ and $w^{*}(0)= \Phi_{\mathfrak{H}}(\zeta)$. From item (2) of \textbf{(DIFF)} for a fixed number $l=1,2,\ldots$ such that $l \geq \tau_{Q}$ we have
	\begin{equation}
		\begin{split}
			\label{EQ: PhiLimitPointLemma1}
			w^{*}_{k}(0)-w^{*}(0) = \psi^{l}(\vartheta^{-l}(q),w^{*}_{k}(-l))-\psi^{l}(\vartheta^{-l}(q),w^{*}(-l)) = \\ = L(l;\vartheta^{-l}(q),w^{*}(-l)) (w^{*}_{k}(-l)-w^{*}(-l)) + o(w^{*}_{k}(-l)-w^{*}(-l)).
		\end{split}
	\end{equation}
    Applying Lemma \ref{LEM: BackwardEstimate} for $q_{0} = \vartheta^{-l}(q)$, $t_{1}=-l-\operatorname{max}\{ \tau_{Q},\tau_{S}+1 \}$, $t_{0} = -l$, $t_{2} = 0$ and trajectories $w^{*}_{k}(\cdot)$ and $w^{*}(\cdot)$, for some constant $D>0$ (depending only on exterior parameters) we get
    \begin{equation}
    	\label{EQ: PhiLimitPointLemma2}
    	\| w^{*}_{k}(-l)-w^{*}(-l) \|_{\mathbb{E}} \leq D e^{\alpha^{+} l} \| h_{k}\|_{\mathbb{E}}.
    \end{equation}
    Firstly, this implies that $o(w^{*}_{k}(-l)-w^{*}_{k}(-l)) = o(h_{k})$ as $k \to +\infty$. From this and \eqref{EQ: PhiLimitPointLemma1} we get as $k \to +\infty$ uniformly in $s \in [-l,0]$
    \begin{equation}
    	\label{EQ: PhiLimitPointLemma3}
    	\frac{w^{*}_{k}(s)-w^{*}(s)}{h_{k}} = L(l+s;\vartheta^{-l}(q),w^{*}(-l)) \left(\frac{w^{*}_{k}(-l)-w^{*}(-l)}{h_{k}}\right) + o(1).
    \end{equation}
    Let $\mathcal{B}$ be the ball centered at $0$ with radius $De^{\alpha^{+}l}$. Consider also $\varepsilon>0$ and the ball $\mathcal{B}_{\varepsilon}$ centered at $0$ with radius $\varepsilon$. From \eqref{EQ: PhiLimitPointLemma3} with $s=0$ it follows that $(w^{*}_{k}(0)-w^{*}(0))/h_{k} \in L(l;\vartheta^{-l}(q);w^{*}(-l))\mathcal{B} + B_{\varepsilon}$ for all sufficiently large $k$. From this and since $\Xi$ satisfies \textbf{(ACOM)} with $\gamma^{+} = \alpha^{+}$, for some function $\gamma_{0}(t) \to 0$ as $t \to +\infty$ we have that
    \begin{equation}
    	\alpha_{K}\left( \bigcup_{k=1}^{\infty} \left\{ \frac{w^{*}_{k}(0)-w^{*}(0)}{h_{k}} \right\} \right) \leq D \cdot \gamma_{0}(l) + \varepsilon. 
    \end{equation}
    Passing to the limit as $\varepsilon \to 0+$ and then as $l \to +\infty$, we obtain the desired precompactness.
    
    Now let $\overline{\xi}$ be any limit point of \eqref{EQ: LemDifferencePrecompact}. By similar as above reasonings applying for $s < 0$ and the Cantor diagonal procedure, we may assume that for any $l=1,2,\ldots$ there is $\overline{\xi}_{l} \in \mathbb{E}$ such that
    \begin{equation}
    	\frac{w^{*}_{k}(-l)-w^{*}(-l)}{h_{k}} \to \overline{\xi}_{l} \text{ as } k \to +\infty.
    \end{equation}
    In virtue of \eqref{EQ: PhiLimitPointLemma3} we have that the relation $\xi^{*}(t) := L(t+l;\vartheta^{-l}(q),w^{*}(-l)) \overline{\xi}_{l}$ for any $t \geq -l$ and any $l=1,2,\ldots$ correctly defines a complete trajectory of the linearization cocycle $\Xi$ over $(q,w^{*}(0))$ such that for any $s$ we have the convergence $\left(w^{*}_{k}(s)-w^{*}(s) \right) / h_{k} \to \xi^{*}(s)$ as $k \to + \infty$. Passing to the limit as $k \to +\infty$ in
    \begin{equation}
    	\int_{-\infty}^{0}e^{2\alpha(s;q)} \left| \frac{w^{*}_{k}(s)-w^{*}(s)}{h_{k}} \right|^{2}_{\mathbb{H}} ds \leq -\frac{\delta^{-1}_{Q}}{|h_{k}|^{2}}V(w^{*}_{k}(0)-w^{*}(0)) \leq \delta^{-1}_{Q} \|Q(q)\| \| \eta_{k} \|_{\mathbb{E}},
    \end{equation}
    we get that $\xi^{*}(\cdot) \in [0^{*}(\cdot)]$ and, consequently, $ \overline{\xi} = \xi^{*}(0) \in \mathfrak{H}'(\Phi_{\mathfrak{H}}(\zeta))$ as it is required. The lemma is proved.
\end{proof}

\begin{lemma}
	\label{LEM: FrechetDiffPhi}
	For any $\zeta \in \mathbb{E}^{-}(q)$ the map $\Phi'_{\mathfrak{H}}(\Phi_{\mathfrak{H}}(\zeta))$ is the Fr\'{e}chet differential of $\Phi_{\mathfrak{H}}$ at $\zeta$, i.~e. the limit
	\begin{equation}
		\lim_{h \to 0}\frac{\Phi_{\mathfrak{H}}(\zeta+h\eta)-\Phi_{\mathfrak{H}}(\zeta)}{h} = \Phi'_{\mathfrak{H}}(\Phi_{\mathfrak{H}}(\zeta))\eta
	\end{equation}
    holds uniformly in $\eta$ from bounded subsets of $\mathbb{E}^{-}(q)$.
\end{lemma}
\begin{proof}
	Supposing the opposite, we obtain a bounded sequence $\eta_{k} \in \mathbb{E}^{-}(q)$ and a sequence $h_{k} \to 0$ as $k \to +\infty$ such that the sequence
	\begin{equation}
    \label{EQ: FrechetDiffLemma1}
	\frac{\Phi_{\mathfrak{H}}(\zeta+h_{k}\eta_{k})-\Phi_{\mathfrak{H}}(\zeta)}{h_{k}} - \Phi'_{\mathfrak{H}}(\Phi_{\mathfrak{H}}(\zeta))\eta_{k}
    \end{equation}
    is uniformly separated from zero. Taking a subsequence if necessary, we may assume that $\eta_{k} \to \overline{\eta}$ and, due to Lemma \ref{LEM: DifferencePrecompactPhi}, that the fraction in \eqref{EQ: FrechetDiffLemma1} tends to some $\overline{\xi} \in \mathfrak{H}'(\Phi_{\mathfrak{H}}(\zeta))$. By our assumptions, we have $\Phi'_{\mathfrak{H}}(\Phi_{\mathfrak{H}}(\zeta))\overline{\eta} \not= \overline{\xi}$. However, the vector in \eqref{EQ: FrechetDiffLemma1} vanishes after taking $\Pi_{q}$. This implies that $\Pi_{q} \overline{\xi} = \overline{\eta}$ that leads to a contradiction. The lemma is proved.
\end{proof}

\begin{theorem}
	\label{TH: HorizontalDifferentiability}
	Let the standing hypotheses along with \textbf{(DIFF)} be satisfied. Then the set $\mathfrak{H}$ is a $C^{1}$-differentiable submanifold in $\mathbb{E}$. Moreover, 
	\begin{enumerate}
		\item The map $\Phi_{\mathfrak{H}} \colon \mathbb{E}^{-}(q) \to \mathbb{E}$ is $C^{1}$-differentiable and its differential at a point $\zeta \in \mathbb{E}^{-}(q)$ is given by $\Phi'_{\mathfrak{H}}(\Phi_{\mathfrak{H}}(\zeta))$;
		\item The tangent space to $\mathfrak{H}$ at the point $v_{0}$ is given by $\mathfrak{H}'(v_{0})$;
		\item The map 
		\begin{equation}
			\label{EQ: HorzontalFiberDiffThNonLinearMap}
			\psi^{t}(\vartheta^{s}(q),\cdot) \colon [v^{*}(\cdot)][s] \to [v^{*}(\cdot)](s+t)
		\end{equation}
	    is a $C^{1}$-diffeomorphism with its differential at the point $w^{*}(s)$, where $w^{*}(\cdot) \in [v^{*}(\cdot)]$, is given by 
		\begin{equation}
			\label{EQ: HorzontalFiberDiffThLinearMap}
			L(t;\vartheta^{s}(q);w^{*}(s)) \colon [0^{*}(\cdot)](s) \to [0^{*}(\cdot)](s+t)
		\end{equation}
	    for any $s \in \mathbb{R}$ and $t \geq 0$. Here $[0^{*}(\cdot)]$ is the equivalence class of the zero complete trajectory of $\Xi$ over $(q,w^{*}(0))$.
	\end{enumerate}
\end{theorem}
\begin{proof}
	Items (1) and (2) follow directly from Lemmas \ref{LEM: FrechetDiffPhi} and \ref{LEM: DifferentialPhiContinuous}.
	
	Now put $\mathfrak{H}(s):=[v^{*}(\cdot)](s)$ and $\mathfrak{H}'(s):=[0^{*}(\cdot)](s)$, where the equivalence classes are understood in the sense of theorem's statement. Applying items (1) and (2) to these spaces, we get that $\mathfrak{H}'(s)$ is the tangent space to $\mathfrak{H}(s)$ at the point $w^{*}(s)$ and $\Phi'_{\mathfrak{H}(s)}(w^{*}(s)) \colon \mathbb{E}^{-}(q) \to \mathfrak{H}'(s)$ is the differential of $\Phi_{\mathfrak{H}(s)} \colon \mathbb{E}^{-}(q) \to \mathfrak{H}(s)$.
	
	Let $\zeta : = \Pi_{\vartheta^{s}(q)}w^{*}(s)$. From item (2) of \textbf{(DIFF)} applied for $t \geq \tau_{Q}$, $s \in \mathbb{R}$, $q \leftrightarrow \vartheta^{s}(q)$, $v_{0} = \Phi_{\mathfrak{H}(s)}(\zeta)=w^{*}(s)$ and $h\xi \leftrightarrow \Phi_{\mathfrak{H}(s)}(\zeta+h \eta) - \Phi_{\mathfrak{H}(s)}(\zeta) = h\Phi'_{\mathfrak{H}(s)}(w^{*}(s))\eta + o(\eta)$ we have
	\begin{equation}
		\begin{split}
			\Pi_{\vartheta^{s+t}(q)} \psi^{t}(\vartheta^{s}(q), \Phi_{\mathfrak{H}(s)}(\zeta + h\eta) ) = \Pi_{\vartheta^{s+t}(q)} \psi^{t}(\vartheta^{s}(q), w^{*}(s) )  + \\ + h\Pi_{\vartheta^{t+s}(q)} L(t;\vartheta^{s}(q);w^{*}(s)) \Phi'_{\mathfrak{H}(s)}(w^{*}(s))\eta+ o(h).
		\end{split}
	\end{equation}
    This shows the differentiability for $t \geq \tau_{Q}$. Now since the map in \eqref{EQ: HorzontalFiberDiffThLinearMap} is an isomorphism, the map in \eqref{EQ: HorzontalFiberDiffThNonLinearMap} is a diffeomorphism for $t \geq \tau_{Q}$. But the map corresponding to times $t \in [0,\tau_{Q}]$ can be represented as a composition of such a large time backward map and a large time forward map. The theorem is proved.
\end{proof}

\subsubsection{Fr\'{e}chet differentiability of vertical leaves}
Throughout this section along with the standing hypotheses of Section \ref{SUBSUBSEC: FrechetHorizontal} we suppose that the semicocycle $\psi$ also satisfies $\textbf{(H3)}^{+}_{w}$ (see Section \ref{SUBSEC: VerticalFoliation}) in a bit stronger form as follows.
\begin{description}
	\item[$\textbf{(H3)}^{+}_{s}$] At least one of the following is satisfied:
	\begin{enumerate}
		\item For the semicocycle $\psi$ assumptions $\textbf{(H3)}^{-}_{w}$, $\textbf{(H3)}^{+}_{w}$ and \textbf{(DIFF)} hold with $\tau_{Q} = 0$;
		\item For the semicocycle $\psi$ assumption $\textbf{(H3)}^{+}_{w}$ holds without the restriction $V^{Q}_{\vartheta^{r}(q)}(\psi^{r}(q,v_{1})-\psi^{r}(q,v_{2})) \geq 0$;
		\item The linearization semicocycle $\Xi$ satisfies $\textbf{(H3)}^{+}_{w}$ in the form of Lemma \ref{LEM: H3WeakPlusLinear} below.
	\end{enumerate} 
\end{description}
Such strengthening as in items (1) and (2) is necessary to obtain property $\textbf{(H3)}^{+}_{w}$ for the linearization cocycle in Lemma \ref{LEM: H3WeakPlusLinear}. This property is stated in item (3) and it is the only property required for further studies.

Let us fix $q \in \mathcal{Q}$ and $v_{0} \in \mathbb{E}$ and consider the positive equivalence class $\mathfrak{V}:=[v_{0}]^{+}(q)$ (see Section \ref{SUBSEC: VerticalFoliation}). By Corollary \ref{COR: VerticalLeafStructure}, the map $\Pi^{+}_{q} \colon \mathfrak{V} \to \mathbb{E}^{+}(q)$ is a homeomorphism. Let $\Phi_{\mathfrak{V}} \colon \mathbb{E}^{+}(q) \to \mathfrak{V}$ be the inverse map.

\begin{lemma}
	\label{LEM: H3WeakPlusLinear}
	Let at least one of items (1) and (2) of $\textbf{(H3)}^{+}_{s}$ be satisfied. Then item (3) of $\textbf{(H3)}^{+}_{s}$ is also satisfied, i.~e. for every $q \in \mathcal{Q}$ and $v \in \mathbb{E}$ we have
	\begin{equation}
		\label{EQ: H3WeakPlusLinearized}
		\begin{split}
			e^{2\alpha(r;q)}V_{\vartheta^{r}(q)}(L(r;q,v)\xi) - V_{q}(\xi) \leq -\delta_{Q} \int_{0}^{r} e^{2\alpha(s;q)}|L(s;q,v)\xi|^{2}_{\mathbb{H}}ds 
		\end{split}
	\end{equation}
	satisfied for all $\xi \in \mathbb{E}$ such that $V_{\vartheta^{r}(q)}(L(r;q;v)\xi) \geq 0$ and all $r \geq \tau_{Q}$.
\end{lemma}
\begin{proof}
	Let us try to consider \eqref{EQ: H3WeakPlus} with $v_{1}=v+h\xi$ and $v_{2} = v$ and divide both sides by $h^{2}$. We want pass to the limit as $h \to 0$ using \textbf{(DIFF)}. If we have item (2) of $\textbf{(H3)}^{+}_{s}$ satisfied, then there are no obstacles for this.
	
	Now let item (1) of $\textbf{(H3)}^{+}_{s}$ be satisfied. Note that for \eqref{EQ: H3WeakPlus} to hold we need to satisfy the inequality $V(\psi^{r}(q,v_{1})-\psi^{r}(q,v_{2})) \geq 0$. According to item (2) of \textbf{(DIFF)} this inequality is satisfied for all sufficiently small $h$ provided that the strict inequality $V_{\vartheta^{r}(q)}(L(r;q;v)\xi) > 0$ holds.
	
	Let us consider the case $V_{\vartheta^{r}(q)}(L(r;q;v)\xi) = 0$ and $L(r;q,v)\xi \not= 0$. We have to construct a sequence $\xi_{k} \to \xi$ such that $V_{\vartheta^{r}(q)}(L(r;q;v)\xi_{k}) > 0$. From this we can apply previous arguments to $\xi_{k}$ instead of $\xi$ and then pass to the limit as $k \to +\infty$.
	
	To construct a sequence $\xi_{k}$ as above let us consider the affine subspace $\mathcal{M}:=\mathbb{E}^{-}(q)+\xi$ and its image $\mathcal{N} := L(r;q,v) \mathcal{M}$. From Lemma \ref{LEM: H3WeakMinusLinearized} it can be deduced that the set $\mathcal{N}$ is a $j_{Q}$-dimensional admissible affine subspace and $L(r;q,v) \colon \mathcal{M} \to \mathcal{N}$ is an affine isomorphism\footnote{This statement is a particular case encountered in Lemma \ref{LEM: GraphTransform}.}. In particular, the admissibility property means that the subspace $\mathbb{E}^{-}_{\mathcal{N}}:=\mathcal{N} - L(r;q,v)\xi$ lies in the negative cone $\{V_{\vartheta^{r}(q)}(v) < 0 \}$. Let $\mathbb{E}^{+}_{\mathcal{N}}$ denote the $V_{\vartheta^{r}(q)}$-orthogonal complement (see Remark \ref{REM: Orthogonality}) to $\mathbb{E}^{-}_{\mathcal{N}}$. Let $L(r;q,v)\xi = \eta^{+}_{\mathcal{N}} + \eta^{-}_{\mathcal{N}}$ be the unique decomposition with $\eta^{+}_{\mathcal{N}} \in \mathbb{E}^{+}_{\mathcal{N}}$ and $\eta^{-}_{\mathcal{N}} \in \mathbb{E}^{-}_{\mathcal{N}}$. Then there exists $\xi_{k} \in \mathcal{M}$ such that $L(r;q,v)\xi_{k} = \eta^{+}_{\mathcal{N}} + (1-1/k) \eta^{-}_{\mathcal{N}}$. In particular, $x_{k} \to \xi$. Since $V_{\vartheta^{r}(q)}(L(r;q,v)\xi) = V_{\vartheta^{r}(q)}(\eta^{+}_{\mathcal{N}}) + V_{\vartheta^{r}(q)}(\eta^{-}_{\mathcal{N}}) = 0$ and $L(r;q,v)\xi \not= 0$, for all $k=1,2,\ldots$ we have
	\begin{equation}
		V_{\vartheta^{r}(q)}(L(r;q,v)\xi_{k}) = V_{\vartheta^{r}(q)}(\eta^{+}_{\mathcal{N}}) + \left(1-\frac{1}{k}\right)V_{\vartheta^{r}(q)}(\eta^{-}_{\mathcal{N}}) > 0
	\end{equation}
    that is required.
    
    It remains to deal only with the case $L(r;q,v) \xi = 0$ and $\xi \not= 0$. Here we highly rely on the additional restriction $\tau_{Q}=0$. Let us consider the number $\overline{r}$ given by the maximum of all numbers $0 \leq r_{0} \leq r$ for which there exists a sequence $r_{k}$ such that $r_{k} \to r_{0}$ and $L(r_{k};q,v)\xi \not= 0$. Note that for any such sequence we must have $V_{\vartheta^{r_{k}(q)}}(L(r_{k};q,v)\xi) > 0$. Indeed, if $V_{\vartheta^{r_{k}(q)}}(L(r_{k};q,v)\xi) \leq 0$ then due to $\textbf{(H3)}^{-}_{w}$ with $\tau_{Q} = 0$ satisfied for $\Xi$ we must also have 
    \begin{equation}
    	V_{\vartheta^{r}(q)}(L(r,q,v)\xi)  \leq V_{\vartheta^{r_{k}(q)}}(L(r_{k};q,v)\xi) -\delta_{Q} \int_{r_{k}}^{r}e^{2\alpha(s;q)}|L(s;q,v)\xi|^{2}_{\mathbb{H}}ds < 0
    \end{equation}
    that contradicts to $V_{\vartheta^{r}(q)}(L(r,q,v)\xi) \geq 0$. Now since $L(r_{0};q,v)\xi = 0$ for all $r_{0} \in [\overline{r},r]$, it is sufficient to show \eqref{EQ: H3WeakPlusLinearized} for $r=\overline{r}$. Previous arguments show that \eqref{EQ: H3WeakPlusLinearized} is satisfied for $r=r_{k}$ for a proper sequence $r_{k} \to \overline{r}$ such that $L(r_{k};q,v)\xi \not= 0$. Then the desired inequality follows after passing to the limit as $k \to +\infty$. The lemma is proved.
\end{proof}

Thus, under $\textbf{(H3)}^{+}_{s}$ we can apply results of Section \ref{SUBSEC: VerticalFoliation} to the linearization semicocycle $\Xi$. In particular, for any $v \in \mathfrak{V}$ there exists the positive equivalence class $\mathfrak{V}'(v):=[0]^{+}(q,v)$ of zero over $(q,v)$. By Corollary \ref{LEM: PiPlusHomeoOntoImage}, we have that $\Pi^{+}_{q} \colon \mathfrak{V}'(v) \to \mathbb{E}^{+}(q)$ is a linear homeomorphism. Let $\Phi'_{\mathfrak{V}}(v) \colon \mathbb{E}^{+}(q) \to \mathfrak{V}'(v)$ be the inverse map.

Under the standing hypotheses, from Theorem \ref{TH: ExpTrackingTheorem} applied to the linearization semicocycle $\Xi$ and $\mathcal{M} := \mathbb{E}^{-}(q)$ we have a family of linear projectors $\Pi^{c}_{L}(v) \colon \mathbb{E} \to \mathbb{E}$ given by the decomposition $\mathbb{E} = \mathbb{E}^{-}(q) \oplus \mathfrak{V}'(v)$.

\begin{lemma}
	\label{LEM: LinearFoliationProjectorCont}
	Under the standing hypotheses and \textbf{(DIFF)}, the linear projector $\Pi^{c}_{L}(v)$ depend continuously on $v \in \mathbb{E}$ in the norm of $\mathcal{L}(\mathbb{E})$.
\end{lemma}
\begin{proof}
	Supposing the opposite, we obtain a sequence $w_{k} \in \mathfrak{V}$ converging to some $w_{0} \in \mathfrak{V}$ as $k \to +\infty$ and a sequence $\xi_{k} \in \mathbb{E}$ such that $\| \xi_{k} \|_{\mathbb{E}} = 1$ and the sequence
	\begin{equation}
		\Pi^{c}_{L}(w_{k})\xi_{k} - \Pi^{c}_{L}(w_{0})\xi_{k}
	\end{equation}
    is uniformly separated from zero. Note that from \eqref{EQ: ExponentialTrackingEstimate} we get that the operators $\Pi^{c}_{L}(v)$ are bounded uniformly in $v \in \mathfrak{V}$. Let us put $\xi^{-}_{k} := \Pi^{c}_{L}(w_{k})\xi_{k}$ and $\widetilde{\xi}^{-}_{k}:=\Pi^{c}_{L}(w_{0})\xi_{k}$. Without loss of generality, we may assume that there exist $\overline{\xi}_{1}, \overline{\xi}_{2} \in \mathbb{E}^{-}(q)$ such that $\xi^{-}_{k} \to \overline{\xi}_{1}$ and $\widetilde{\xi}^{-}_{k} \to \overline{\xi}_{2}$ as $k \to +\infty$. By our assumptions, we have $\overline{\xi}_{1} \not= \overline{\xi}_{2}$. Putting $\xi^{+}_{k} := \xi_{k} - \xi^{-}_{k}$ and $\widetilde{\xi}^{+}_{k}:=\xi_{k} - \widetilde{\xi}^{-}_{k}$, for $T>0$ we have as $k \to +\infty$
    \begin{equation}
    	\label{EQ: LinearFoliationProjectorCont1}
    	\begin{split}
    	\int_{0}^{T}e^{2\beta(s;q)}|L(s;q,w_{0}) (\xi^{+}_{k} - \widetilde{\xi}^{+}_{k})|^{2}_{\mathbb{H}}ds = \\ =     	\int_{0}^{T}e^{2\beta(s;q)}|L(s;q,w_{0}) (\xi^{-}_{k} - \widetilde{\xi}^{-}_{k})|^{2}_{\mathbb{H}}ds \to	\int_{0}^{T}e^{2\beta(s;q)}|L(s;q,w_{0}) (\overline{\xi}_{1} - \overline{\xi}_{2} ) |^{2}_{\mathbb{H}}ds.
    	\end{split}
    \end{equation}
    Since the sequences $\xi^{+}_{k} \in \mathfrak{V}'(w_{k})$ and $\widetilde{\xi}^{+}_{k} \in \mathfrak{V}'(w_{0})$ are uniformly bounded, from \eqref{EQ: VerticalFibreIntEstimate} we obtain a constant $M>0$ such that for all $k=1,2,\ldots$
    \begin{equation}
    	\left(\int_{0}^{+\infty}e^{2\beta(s;q)}|L(s;q,w_{k})\xi^{+}_{k}|^{2}ds\right)^{1/2} + \left(\int_{0}^{+\infty}e^{2\beta(s;q)}|L(s;q,w_{0})\widetilde{\xi}^{+}_{k}|^{2}ds\right)^{1/2} \leq M.
    \end{equation}
    By the Minkowski inequality, we have
    \begin{equation}
    	\label{EQ: LinearFoliationProjectorCont2}
    	\begin{split}
    		\left(\int_{0}^{T}e^{2\beta(s;q)}|L(s;q,w_{0}) (\xi^{+}_{k} - \widetilde{\xi}^{+}_{k})|^{2}_{\mathbb{H}}ds \right)^{1/2} \leq \\ \leq 	\left(\int_{0}^{T}e^{2\beta(s;q)}|L(s;q,w_{0})\widetilde{\xi}^{+}_{k})|^{2}_{\mathbb{H}}ds \right)^{1/2} + 	\left(\int_{0}^{T}e^{2\beta(s;q)}|L(s;q,w_{k})\xi^{+}_{k} |^{2}_{\mathbb{H}}ds \right)^{1/2} + \\ +
    			\left(\int_{0}^{T}e^{2\beta(s;q)}|(L(s;q,w_{k}) - L(s;q,w_{0})) \xi^{+}_{k}|^{2}_{\mathbb{H}}ds \right)^{1/2} \leq \\ \leq M + \left(\int_{0}^{T}e^{2\beta(s;q)}|(L(s;q,w_{k}) - L(s;q,w_{0})) \xi^{+}_{k}|^{2}_{\mathbb{H}}ds \right)^{1/2}.
    	\end{split}
    \end{equation}
    By the last property in \textbf{(DIFF)} and the Dominated Convergence Theorem, the last integral in \eqref{EQ: LinearFoliationProjectorCont2} tends to zero as $k \to +\infty$. Thus, from this and \eqref{EQ: LinearFoliationProjectorCont1} we have for all $T > 0$ the inequality
    \begin{equation}
    \int_{0}^{T}e^{2\beta(s;q)}|L(s;q,w_{0}) (\overline{\xi}_{1} - \overline{\xi}_{2} ) |^{2}_{\mathbb{H}}ds \leq M^{2}.
    \end{equation}
    Putting $T=+\infty$, we obtain $\overline{\xi}_{1}-\overline{\xi}_{2} \in \mathfrak{V}(w_{0})$. But, since $\overline{\xi}_{1}-\overline{\xi}_{2} \in \mathbb{E}^{-}(q)$, we must have $\overline{\xi}_{1} = \overline{\xi}_{2}$. The lemma is proved.
\end{proof}

We have the following two lemmas showing that $\Phi'_{\mathfrak{V}}(\cdot)$ are continuous differentials of $\Phi_{\mathfrak{V}}$ in the sense of Fr\'{e}chet.
\begin{lemma}
	\label{LEM: VerticalDiffContDep}
	Under the standing hypotheses and \textbf{(DIFF)}, the linear map $\Phi'_{\mathfrak{V}}(v)$ depend continuously on $v \in \mathfrak{V}$ in the norm of $\mathcal{L}(\mathbb{E}^{+}(q);\mathbb{E})$.
\end{lemma}
\begin{proof} 
	For any $v \in \mathfrak{V}$ and $\eta \in \mathbb{E}^{+}(q)$, by Lemma \ref{LEM: PositiveEquivalence}, we have $V_{q}( \Phi'_{\mathfrak{V}}(v) \eta ) \geq 0$. Thus from \eqref{EQ: ConstantProjectorNorm} we get the inequality
	\begin{equation}
		\label{EQ: ContinuityVerticalDifferentialBound}
		C^{2}_{q} \| \Pi_{q} \Phi'_{\mathfrak{V}}(v) \eta \|^{2}_{\mathbb{E}} \leq -V(\Pi_{q} \Phi'_{\mathfrak{V}}(v) \eta) \leq V(\Pi^{+}_{q} \Phi'_{\mathfrak{V}}(v) \eta) \leq \|Q(q)\| \cdot \| \eta \|^{2}_{\mathbb{E}}.
	\end{equation}
	It shows that the linear maps $\Phi'_{\mathfrak{V}}(v)$ are bounded uniformly in $v \in \mathfrak{V}$. 
	
    Now let $w_{k} \in \mathfrak{V}$ converge to some $w_{0} \in \mathfrak{V}$ as $k \to +\infty$.	Using the projector $\Pi^{c}_{L}(\cdot)$ and $\Pi^{+}_{q}(\Phi'_{\mathfrak{V}}(w_{k}) - \Phi'_{\mathfrak{V}}(w_{0})) = 0$, we have
	\begin{equation}
		\label{EQ: LinearFoliationProjectorVert1}
		\begin{split}
			\Phi'_{\mathfrak{V}}(w_{k}) - \Phi'_{\mathfrak{V}}(w_{0}) =\\= \Pi^{c}_{L}(w_{0}) \left[\Phi'_{\mathfrak{V}}(w_{k}) - \Phi'_{\mathfrak{V}}(w_{0}) \right].
		\end{split}
	\end{equation}
    Since $\Pi^{c}_{L}(w_{0})\Phi'_{\mathfrak{V}}(w_{0}) = 0$, $\Pi^{c}_{L}(w_{k})\Phi'_{\mathfrak{V}}(w_{k})=0$ and $\Pi^{c}_{L}(w_{k}) \to \Pi^{c}_{L}(w_{0})$, using Lemma \ref{LEM: LinearFoliationProjectorCont} and the uniform bound from \eqref{EQ: ContinuityVerticalDifferentialBound}, we get that the right-hand side in \eqref{EQ: LinearFoliationProjectorVert1} tends to zero as $k \to +\infty$. The lemma is proved.
\end{proof}

\begin{lemma}
	\label{LEM: FrechetDifferentiabilityVertical}
	Under the standing hypotheses and \textbf{(DIFF)}, the map $\Phi_{\mathfrak{V}} \colon \mathbb{E}^{+}(q) \to \mathbb{E}$ is $C^{1}$-differentiable and its differential at a point $\zeta_{0} \in \mathbb{E}^{+}(q)$ is given by $\Phi'_{\mathfrak{V}}(\Phi_{\mathfrak{V}}(\zeta_{0}))$, i.~e. the limit
	\begin{equation}
		\lim_{h \to 0}\frac{\Phi_{\mathfrak{V}}(\zeta_{0} + h\eta) -\Phi_{\mathfrak{V}}(\zeta_{0})}{h} = \Phi'_{\mathfrak{V}}(\Phi_{\mathfrak{V}}(\zeta_{0}))\eta 
	\end{equation}
    is uniform in $\eta \in \mathbb{E}^{+}(q)$ from bounded subsets.
\end{lemma}
\begin{proof}
	Supposing the opposite, we obtain sequences $\eta_{k} \in \mathbb{E}^{+}(q)$ and $h_{k} \in \left\{\mathbb{R} \setminus \{0\} \right\}$, where $\eta_{k}$ is bounded and $h_{k} \to 0$ as $k \to +\infty$, such that
	\begin{equation}
		\label{EQ: FrechetVerticalDiffLemma1}
		\frac{\Phi_{\mathfrak{V}}(\zeta_{0} + h_{k}\eta_{k}) -\Phi_{\mathfrak{V}}(\zeta_{0})}{h_{k}} - \Phi'_{\mathfrak{V}}(\Phi_{\mathfrak{V}}(\zeta_{0}))\eta_{k} 
	\end{equation}
    is uniformly separated from zero. Let $\xi_{k}$ denote the fractional term in \eqref{EQ: FrechetVerticalDiffLemma1}. Put $\xi^{-}_{k}:= \Pi^{c}_{L}(\Phi_{\mathfrak{V}}(\zeta_{0})) \xi_{k}$ and $\xi^{+}_{k} := \xi_{k} - \xi^{-}_{k}$. Without loss of generality, we may assume that $\xi^{-}_{k}$ converges to some $\overline{\xi} \in \mathbb{E}^{-}(q)$. By our assumptions and since \eqref{EQ: FrechetVerticalDiffLemma1} vanishes after taking $\Pi^{+}_{q}$, we have $\overline{\xi} \not= 0$. On the other hand, from \textbf{(DIFF)} it follows that for $s \in [0,T]$ and $k=1,2,\ldots$ that\footnote{Here we used the relation $o(\Phi_{\mathfrak{V}}(\zeta_{0}+h_{k} \eta_{k}) - \Phi_{\mathfrak{V}}(\zeta_{0})) = o(h_{k})$. This follows from the fact that $\Phi_{\mathfrak{V}}(\cdot)$ is locally Lipschitz at $\zeta_{0}$. The latter easily follows from the boundedness of the fractional term in \eqref{EQ: FrechetVerticalDiffLemma1}.}
    \begin{equation}
    \psi^{s}(q,\Phi_{\mathfrak{V}}(\zeta_{0}+h_{k} \eta_{k})) - \psi^{s}(q,\Phi_{\mathfrak{V}}(\zeta_{0})) = h_{k}L(s;q,\Phi_{\mathfrak{V}}(\zeta_{0}))\xi_{k} + o(h_{k})
    \end{equation}
    and, consequently,
    \begin{equation}
    	\begin{split}
    		L(s;q,\Phi_{\mathfrak{V}}(\zeta_{0}))\xi^{-}_{k} = \frac{\psi^{s}(q,\Phi_{\mathfrak{V}}(\zeta_{0}+h_{k} \eta_{k})) - \psi^{s}(q,\Phi_{\mathfrak{V}}(\zeta_{0}))}{h_{k}} + \\ + L(s;q,\Phi_{\mathfrak{V}}(\zeta_{0}))\xi^{+}_{k} + o(1).
    	\end{split}
    \end{equation}
    From this we get that for some uniform constant $M > 0$ we have for all $T>0$ that
    \begin{equation}
    	\int_{0}^{T}e^{2\beta(s;q)}|L(s;q,\Phi_{\mathfrak{V}}(\zeta_{0}))\overline{\xi}|^{2}_{\mathbb{H}} \leq M.
    \end{equation}
    In particular, $\overline{\xi} \in \mathfrak{V}'(\Phi_{\mathfrak{V}}(\zeta_{0}))$ and, consequently, $\overline{\xi} = 0$ that leads to a contradiction. The lemma is proved.
\end{proof}

We collect results of the present section in the following theorem.
\begin{theorem}
	\label{TH: VertialLeavesDifferentiability}
	Let the standing hypotheses and \textbf{(DIFF)} be satisfied. Then the set $\mathfrak{V}$ is a $C^{1}$-differentiable submanifold in $\mathbb{E}$. Moreover,
    \begin{enumerate}
    	\item The map $\Phi_{\mathfrak{V}} \colon \mathbb{E}^{+}(q) \to \mathbb{E}$ is $C^{1}$-differentiable and its differential at a point $\zeta \in \mathbb{E}^{+}(q)$ is given by $\Phi'_{\mathfrak{V}}(\Phi_{\mathfrak{V}}(\zeta))$.
    	\item The tangent space at $v \in \mathfrak{V}$ is given by $\mathfrak{V}'(v)$.
    	\item The map
    	\begin{equation}
    		\psi^{t}(q,\cdot) \colon [v_{0}]^{+}(q) \to [\psi^{t}(q,v_{0})]^{+}(\vartheta^{t}(q))
    	\end{equation}
        is $C^{1}$-differentiable for any $t \geq \tau_{Q}$ and its differential at a point $v \in [v_{0}]^{+}(q)$ is given by
        \begin{equation}
        	L(t;q,v) \colon [0]^{+}(q,v) \to [0]^{+}(\vartheta^{t}(q),\psi^{t}(q,v)).
        \end{equation}
    \end{enumerate}
\end{theorem}
\begin{proof}
	Items (1) and (2) of the theorem follows from Lemma \ref{LEM: VerticalDiffContDep} and \ref{LEM: FrechetDifferentiabilityVertical}. The statement of item (3) can be proven analogously to the corresponding statement in Theorem \ref{TH: HorizontalDifferentiability}.
\end{proof}

\subsection{Admissible projectors}
\label{SUBSEC: AdmissibleProjectors}

From previous results we see that the $V_{q}$-orthogonal projectors $\Pi_{q}$ are convenient for reconstructing of foliations and studying of their properties. This is since the decomposition
\begin{equation}
	\label{EQ: VqOrthogonalProjectorFormDecomposition}
	V_{q}(v) = V_{q}(\Pi_{q} v) + V_{q}(\Pi^{+}_{q} v)
\end{equation}
simplifies various estimates and allows useful constructions such as perturbed quadratic forms $V^{(\varkappa)}_{q}$. But for some purposes (for example, to study the so-called \textit{inertial forms} that are coordinate representations of the vector field restricted to the submanifold) it may be convenient to use other types of projectors which provide different coordinate charts on constructed submanifolds.

Under \textbf{(H1)} (with $P(q)$ changed with $Q(q)$) and \textbf{(H2)} (with $j$ changed with $j_{Q}$) we say that a projector $\Pi \colon \mathbb{E} \to \mathbb{E}$ is \textit{admissible} over $q \in \mathcal{Q}$ if 
\begin{description}
	\item[\textbf{(AP1)}] For all non-zero $v \in \mathbb{E}^{-}_{\Pi}:=\operatorname{Ran}\Pi$ we have $V_{q}(v) < 0$ and $\operatorname{dim}\operatorname{Ran}\Pi = j_{Q}$;
	\item[\textbf{(AP2)}] For all non-zero $v \in \mathbb{E}^{+}_{\Pi}:=\operatorname{Ker}\Pi$ we have $V_{q}(v)>0$.
\end{description}
Note that \textbf{(AP1)} is equivalent for $\operatorname{Ran}\Pi$ to be a Lipschitz admissible set over $q$. In particular, the dichotomy projector $\Pi^{d}_{q}$ is admissible. Note that for an admissible (over $q$) projector $\Pi$ we only have the property
\begin{equation}
	\label{EQ: AdmissibleProjectorFormDecomposition}
	V_{q}(v) = V_{q}(\Pi v) + V_{q}(\Pi^{+} v) +2\langle \Pi^{+}v, Q(q) \Pi v \rangle.
\end{equation}
Here $\Pi^{+}:= I - \Pi$ is the complementary to $\Pi$ projector.

\begin{theorem}
	\label{TH: DichotomyProjectorChartHorizontal}
	Under the hypotheses and terms of Theorem \ref{TH: HorizontalFibresAlongCompleteOrbits}, any admissible (over $q$) projector $\Pi$ provides a homeomorphism of $[v^{*}(\cdot)](0)$ onto $\mathbb{E}^{-}_{\Pi}$. Moreover, the inverse map $\Phi_{\Pi} \colon \mathbb{E}^{-}_{\Pi} \to [v^{*}(\cdot)](0)$ is Lipschitz and it is $C^{1}$-differentiable under the hypotheses of Theorem \ref{TH: HorizontalDifferentiability}.
\end{theorem}
\begin{proof}
	Since $[v^{*}(\cdot)](0)$ is admissible, we have that $\Pi \colon [v^{*}(\cdot)](0) \to \operatorname{Ran}\Pi$ is an injective continuous map. Since we already know that $[v^{*}(\cdot)](0)$ is a $j_{Q}$-dimensional topological submanifold, we can apply the Brouwer theorem on invariance of domain to get that the image $\Pi([v^{*}(\cdot)](0))$ is open in $\operatorname{Ran}\Pi$. It remains to show that the image is also closed. This follows from similar estimates as in Lemma \ref{EQ: ProjectorHomeomorphismOntoImage}, but now \eqref{EQ: AdmissibleProjectorFormDecomposition} should be used instead of \eqref{EQ: VqOrthogonalProjectorFormDecomposition}. We left this as an exercise for the reader.
	
	To prove the $C^{1}$-differentiability, note that from Theorem \ref{TH: HorizontalDifferentiability} we know that $\mathfrak{H} := [v^{*}(\cdot)](0)$ is a $C^{1}$-differentiable submanifold in $\mathbb{E}$ and, consequently, the restriction of $\Pi$ to $\mathfrak{H}$ is $C^{1}$-differentiable. Moreover, arguments as above shows that its differential (that is also $\Pi$) provides an isomorphism between any tangent space $\mathfrak{H}'(v_{0})$ at $v_{0} \in \mathfrak{H}$ and $\mathbb{E}^{-}_{\Pi}$. Thus, by the Inverse Mapping Theorem, $\Pi$ is a $C^{1}$-diffeomorphism and, consequently, $\Phi_{\Pi}$ is $C^{1}$-differentiable. The proof is finished.
\end{proof}

\begin{theorem}
	\label{TH: DichotomyProjectorChartVertical}
	Under the hypotheses and terms of Theorem \ref{COR: VerticalLeafStructure}, for any admissible (over $q$) projector $\Pi$ the complementary projector $\Pi^{+}$ provides a homeomorphism of $[v_{0}]^{+}(q)$ onto $\mathbb{E}^{+}_{\Pi}$. Moreover, the inverse map $\Phi^{+}_{\Pi} \colon \mathbb{E}^{+}_{\Pi} \to [v_{0}]^{+}(q)$ is $C^{1}$-differentiable under the hypotheses of Theorem \ref{TH: VertialLeavesDifferentiability}.
\end{theorem}
\begin{proof}
	Analogously to Lemma \ref{LEM: PiPlusHomeoOntoImage} by utilizing \eqref{EQ: AdmissibleProjectorFormDecomposition}, one can show that $\Pi^{+} \colon [v_{0}]^{+}(q) \to \mathbb{E}^{+}_{\Pi}$ is a homeomorphism onto its image. Moreover, Theorem \ref{TH: ExpTrackingTheorem} applied to $\mathcal{M} := \mathbb{E}^{-}_{\Pi} + \zeta$ for each $\zeta \in \mathbb{E}^{+}_{\Pi}$ shows that the image is entire $\mathbb{E}^{+}_{\Pi}$.
	
	To show the $C^{1}$-differentiability, we again note that $\mathfrak{V}:=[v_{0}]^{+}(q)$ is a $C^{1}$-differentiable Banach $\mathbb{E}^{+}(q)$-submanifold in $\mathbb{E}$ and, consequently, the restriction of $\Pi^{+}$ to $\mathfrak{V}$ is a $C^{1}$-differentiable mapping. Moreover, arguments as above shows that its differential at any point $v_{0}$ of $\mathfrak{V}$ provides a homeomorphism between the tangent space $\mathfrak{V}'(v_{0})$ and $\mathbb{E}^{+}_{\Pi}$. This along with the Inverse Mapping Theorem shows the required statement. The proof is finished.
\end{proof}

Note that the admissibility of a projector $\Pi$ is stable under small perturbations of $\Pi$. Since the dichotomy projectors $\Pi^{d}_{q}$ in practice cannot be computed exactly (formulas usually contain certain integrals, which are computed numerically), results of the present subsection may be important for numerical simulations.
\section{Inertial and Slow manifolds}
\label{SEC: InertialManifolds}

\subsection{Principal leaves (Smith's amenable sets)}
\label{SUBSEC: AmenableSets}

During this subsection we study a cocycle $\psi$ which satisfies $\textbf{(H3)}^{-}_{w}$ defined in Subsection \ref{SUBSEC: HorizontalLeaves} and use the same notation.

We will say that a complete trajectory $v(\cdot)$ over $q$ is \textit{amenable} if 
\begin{equation}
	\label{EQ: AmenabilityH}
	\int_{-\infty}^{0}e^{2\alpha(s,q)}|v(s)|^{2}_{\mathbb{H}}ds < +\infty.
\end{equation}
From Lemma \ref{LEM: CompleteTrajectoriesOrdering} we immediately have the following lemma.
\begin{lemma}
	\label{LEM: AmenableTrajectories}
	Suppose $\textbf{(H3)}^{-}_{w}$, and \textbf{(S)} are satisfied. Then any two amenable over $q$ trajectories $v^{*}_{1}(\cdot)$ and $v^{*}_{2}(\cdot)$ belong to the same class of pseudo-ordering, i.e. $[v^{*}_{1}(\cdot)] = [v^{*}_{2}(\cdot)](\cdot)$. Moreover, if $v^{*}_{1}(t) = v^{*}_{2}(t)$ for some $t \in \mathbb{R}$, then $v^{*}_{1}(\cdot) \equiv v^{*}_{2}(\cdot)$.
\end{lemma}
\begin{proof}
	The first part follows from the Minkowsky inequality:
	\begin{equation}
		\begin{split}
		\left(\int_{-\infty}^{0}e^{2\alpha(s;q)}|v^{*}_{1}(s)-v^{*}_{2}(s)|^{2}_{\mathbb{H}}ds \right)^{1/2} \leq \\ \leq 	\left(\int_{-\infty}^{0}e^{2\alpha(s;q)}|v^{*}_{1}(s)|^{2}_{\mathbb{H}}ds \right)^{1/2} + 	\left(\int_{-\infty}^{0}e^{2\alpha(s;q)}|v^{*}_{2}(s)|^{2}_{\mathbb{H}}ds \right)^{1/2}
		\end{split}
	\end{equation}
    and item 1) of Lemma \ref{LEM: CompleteTrajectoriesOrdering}.
    
    The last part follows from the estimate \eqref{EQ: OrderedCompleteTrajEstimate}.
\end{proof}

It is not a priori known that amenable trajectories exist, so we will postulate it as in the following assumption.
\begin{description}
	\item[\textbf{(PL)}] For each $q \in \mathcal{Q}$ there exists at least one amenable trajectory over $q$.
\end{description}
It is clear that in the case $\alpha^{-}>0$ any bounded complete trajectory is amenable.

Under \textbf{(PL)}, for each $q \in \mathcal{Q}$ and an amenable over $q$ trajectory $v^{*}(\cdot)$ we define $\mathfrak{A}(q) := [v^{*}(\cdot)](0)$, i.~e. be the set of all $v_{0} \in \mathbb{E}$ such that there exists an amenable over $q$ trajectory passing through $v_{0}$. We call $\mathfrak{A}(q)$ the \textit{principal leaf} over $q$

From Theorem \ref{TH: HorizontalFibresAlongCompleteOrbits} we immediately have the following statement.
\begin{theorem}
	\label{TH: PrincipalLeaveConstruction}
	Suppose $\textbf{(H3)}^{-}_{w}$, \textbf{(ACOM)} with $\gamma^{+} = \alpha^{+}$, \textbf{(ULIP)}, \textbf{(PROJ)} and \textbf{(PL)} are satisfied. Then
	\begin{enumerate}
		\item $\mathfrak{A}(q)$ is $C_{Lip}(q)$-Lipschitz $\varkappa_{0}$-admissible, where $C_{Lip}(q)$ is defined in \eqref{EQ: LipschitzConstantOverQ} and $\varkappa_{0}$ is defined in Lemma \ref{LEM: ConePerturbation}.
		\item $\psi^{t}(q,\cdot)$ maps $\mathfrak{A}(q)$ onto $\mathfrak{A}(\vartheta^{t}(q))$ homeomorphically
		\item Under $C^{1}$ the set $\mathfrak{A}(q)$ is a $C^{1}$-differentiable submanifold 
	\end{enumerate}
\end{theorem}

To study the continuity of $\mathfrak{A}(q)$ in $q$ let us consider the map $\Phi_{q} \colon \mathbb{E}^{-}(q) \to \mathfrak{A}(q)$ given by the inverse of $\Pi_{q} \colon \mathfrak{A}(q) \to \mathbb{E}^{-}(q)$. Put $\mathcal{D}(\Phi) := \bigcup_{q \in \mathcal{Q}} \{ q \} \times \mathbb{E}^{-}(q)$ and define the map $\Phi \colon \mathcal{D}(\Phi) \to \mathbb{E}$ by $\mathcal{D}(\Phi) \ni (q,\zeta) \mapsto \Phi(q,\zeta):=\Phi_{q}(\zeta)$. Note that $\mathcal{D}(\Phi)$ is closed under \textbf{(CD)}. Recall also \textbf{(BA)} from Subsection \ref{SUBSEC: DefinitionsAndHypotheses}, which gives us for each $q \in \mathcal{Q}$ a complete trajectory $w_{q}(\cdot)$ over $q$ with the bound $\| w_{q}(t) \|_{\mathbb{E}} \leq M_{b}$, which is uniform in $q \in \mathcal{Q}$ and $t \leq 0$ (as in \eqref{EQ: BAUniformBound}).

\begin{theorem}
	\label{TH: ContinuousDependenceFibres}
	Suppose $\textbf{(H3)}^{-}_{w}$ with $\alpha^{-} > 0$ and continuous $\alpha_{0}(\cdot)$, \textbf{(ACOM)} with $\gamma^{+} = \alpha^{+}$, \textbf{(ULIP)} and \textbf{(BA)} are satisfied. Then the map $\Phi \colon \mathcal{D}(\Phi) \to \mathbb{E}$ is continuous.
\end{theorem}
\begin{proof}
	Let us consider converging sequences $q_{k} \to q$ and $\mathbb{E}^{-}(q_{k}) \ni \zeta_{k} \to \zeta \in \mathbb{E}^{-}(q)$. Let $v^{*}_{k}(\cdot)$ be the amenable trajectory over $q_{k}$ such that $v^{*}_{k}(0) = \Phi(q_{k},\zeta_{k})$ and let $v^{*}(\cdot)$ be the amenable trajectory over $q$ such that $v^{*}(0)=\Phi(q,\zeta)$. We have to show that $v^{*}_{k}(0) \to v^{*}(0)$. Let us suppose the contrary, i.e. that $v^{*}_{k}(0)$ does not converge to $v^{*}(0)$. Then for some subsequence (we keep the same index) $v^{*}_{k}(0)$ is separated from $v^{*}(0)$. Since $v^{*}_{k}(\cdot)$ and $w^{*}_{q_{k}}(\cdot)$ from \textbf{(BA)} are both amenable trajectories over $q_{k}$, from Lemma \ref{LEM: AmenableTrajectories} and item 1) of Lemma \ref{LEM: CompleteTrajectoriesOrdering} we have
	\begin{equation}
	\label{EQ: ContinuityTheoremH3Using}
	 \delta_{Q} \cdot \int_{-\infty}^{0}e^{2\alpha(s;q_{k})}|v^{*}_{k}(s)-w^{*}_{q_{k}}(s)|^{2}_{\mathbb{H}}ds \leq 	-V_{q_{k}}(v^{*}_{k}(0)-w^{*}_{q_{k}}(0)) \leq M_{Q} \| \zeta_{k} - \Pi_{q_{k}} w^{*}_{q_{k}}(0) \|^{2}_{\mathbb{E}}.
	\end{equation}
    Considering the integral from \eqref{EQ: ContinuityTheoremH3Using} on $[-l-1,-l]$ for $l=1,2,\ldots$ and applying the mean value theorem with \textbf{(ULIP)}, we get a sequence $t_{l} \to -\infty$ such that $v_{k}(t_{l}) \in \mathcal{B}_{l}$ for a proper bounded set $\mathcal{B}_{l}$ with $\operatorname{diam}\mathcal{B}_{l} \leq De^{-\alpha^{+}t_{l}}$, where $D>0$ is a uniform constant. Then Lemma \ref{LEM: GeneralCompactnessLemma} gives us a limit complete trajectory $\widetilde{v}(\cdot)$ over $q$ for $v^{*}_{k}(\cdot)$. Note that $\zeta_{k} = \Pi_{q_{k}}v^{*}(0) \to \zeta$ and, consequently, $\Pi_{q} \widetilde{v}(0) = \zeta$. Due to \textbf{(BA)} and \eqref{EQ: ContinuityTheoremH3Using} there exists a uniform constant $M'>0$ such that
    \begin{equation}
    	\int_{-\infty}^{0} e^{2 \alpha(s;q_{k})}|v^{*}_{k}(s)|^{2}_{\mathbb{H}} \leq M'
    \end{equation}
    Since $\alpha_{0}(\cdot)$ is continuous, the same estimate holds in the limit and, consequently, $\widetilde{v}(\cdot)$ is amenable over $q$. That leads to a contradiction due to the uniqueness given by Lemma \ref{LEM: AmenableTrajectories}. The proof is finished.
\end{proof}
\begin{remark}
	\label{REM: ContinuousDependenceExponentRestriction}
	The restriction $\alpha^{-} > 0$ can be omitted if we require \textbf{(BA)} to hold with all $w_{q}(\cdot)$ being identically zero. This is appropriate for the study of unstable manifolds, where $\alpha^{-} < 0$. In such applications the cocycle $\psi$ describes small deviations from a given invariant subset and, consequently, the zero is always a trajectory.
\end{remark}

\subsection{Exponential tracking}
\label{SUBSEC: ExponentialTracking}
This subsection is devoted to the \textit{exponential tracking} property (also known as the \textit{asymptotic completeness} \cite{Robinson1996AsympComl} or the \textit{existence of asymptotic phase}) of inertial manifolds. In its strongest form (as in our case) this property requires a decomposition of the cocycle trajectories into the sum of a ``fast'' part, which decays exponentially fast, and a ``slow'' part, which is given by a trajectory on the inertial manifold\footnote{Note that the ``fast'' part is not a trajectory in general due to a nonlinear nature of the cocycle.}. Most of known conditions for the existence of inertial manifolds (to be discussed in Section \ref{SEC: Applications}), in fact, imply the existence of what we called vertical leaves (see Subsection \ref{SUBSEC: VerticalFoliation}) which foliate $\mathbb{E}$. These leaves contain the decaying parts and squeeze under the dynamics towards the principal leaves $\mathfrak{A}(q)$ considered in Subsection \ref{SUBSEC: AmenableSets}. This is a nonlocal analog of the foliations constructed in a neighborhood of normally hyperbolic manifolds \cite{BatesLuZeng2000}. Note that most works have lack of this observation and the work of R.~Rosa and R.~Temam \cite{RosaTemam1996} is a rare exception.

We already have all the necessary statements, which are collected in the following theorem. See Fig. \ref{FIG: ExpTracking} for an illustration.
\begin{theorem}
	\label{TH: ExponentialTrackingIM}
	Let $\textbf{(H3)}^{-}_{w}$, $\textbf{(H3)}^{+}_{w}$, \textbf{(ULIP)}, \textbf{(ACOM)} with $\gamma^{+} = \alpha^{+}$, \textbf{(PROJ)} and \textbf{(PL)} be satisfied. Then
	\begin{enumerate}
		\item[1)] For any $q \in \mathcal{Q}$ and $v_{0} \in \mathbb{E}$ there exists a unique point $v^{*}_{0} \in \mathfrak{A}(q)$ such that $v^{*}_{0} \in [v_{0}]^{+}(q)$. Moreover, \eqref{EQ: ExponentialTrackingEstimate} is satisfied.
		\item[2)] $\Pi_{q} \colon [v_{0}]^{+}(q) \to \mathbb{E}^{+}(q)$ is a homeomorphism and it is a $C^{1}$-diffeomorphism provided that \textbf{(DIFF)} is satisfied (see Theorem \ref{TH: VertialLeavesDifferentiability} for details).
		\item[3)] The central projector $\Pi^{c}_{q}$ onto $\mathfrak{A}(q)$, i.e. the map $\Pi^{c}_{q}(v_{0}):=v^{*}_{0}$ defined in terms of item 1), is continuous in $(q,v_{0})$ provided that \textbf{(CD)}, \textbf{BA} are satisfied, $\alpha^{-}>0$ and $\alpha_{0}(\cdot)$ is continuous.
	\end{enumerate}
\end{theorem}
\begin{proof}
	Item 1) of the theorem is a direct corollary of Theorem \ref{TH: ExpTrackingTheorem}. Item 2) follows from Corollary \ref{COR: VerticalLeafStructure} and Theorem \ref{TH: VertialLeavesDifferentiability}. Item 3) is a consequence of Theorem \ref{TH: ContinuousDependenceFibres} and Theorem \ref{TH: CentralProjectorContinuityGeneral}. The proof is finished.
\end{proof}

\begin{figure}
	\centering
	\includegraphics[width=1\linewidth]{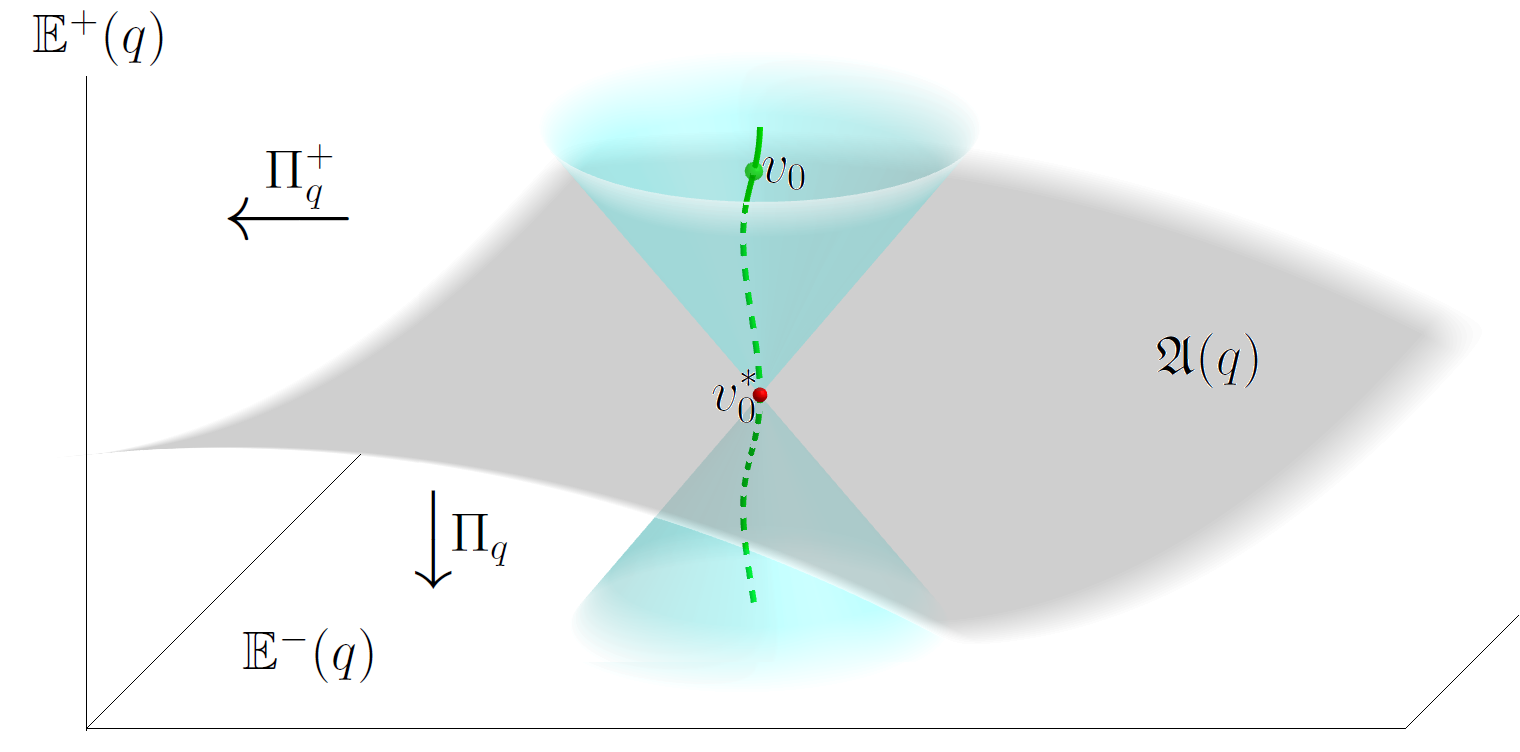}
	\caption{An illustation to Theorem \ref{TH: ExponentialTrackingIM}. Here over each point $v^{*}_{0}$ from the principal leaf $\mathfrak{A}(q)$ (gray) we have the positive equivalence class $[v^{*}_{0}]^{+}(q)$ (green), which lies in the cone $\{ V_{q}(v) \geq 0 \}$ translated to $v^{*}_{0}$ (the ``interior'' part of the cyan cone). Any point $v_{0} \in [v^{*}_{0}]^{+}(q)$ is mapped to $v^{*}_{0}$ by the central projector $\Pi^{c}_{q}$.}
	\label{FIG: ExpTracking}
\end{figure}

\subsection{Normal hyperbolicity and higher differentiability}
\label{SUBSEC: NormalHyperbolicity}

In this subsection, we discuss the normal hyperbolicity of inertial manifolds. There are abstract theories dealing with compact submanifolds and uniform normal hyperbolicity, where it is proved that such manifolds are stable under small perturbations \cite{BatesLuZeng1998}. However, in our case the manifolds are not compact and, moreover, it is much easier to prove their robustness by using the same methods as in the paper as we will do in Subsection \ref{SUBSEC: RobustnessManifolds} without appealing to the normal hyperbolicity.

We already have a family of tangent spaces $\mathfrak{A}(v) = \mathcal{T}_{v}\mathfrak{A}$ at any point $v \in \mathfrak{A}$, which forms the inertial manifold for the linearization cocycle.

Let $\mathcal{T}_{v_{0}}\mathfrak{A}(q)$ be the tangent space to $\mathfrak{A}(q)$ at $v_{0} \in \mathfrak{A}(q)$. We will understood the \textit{normal hyperbolicity} of the family $\mathfrak{A}(q)$ in the following sense.
\begin{description}
	\item[\textbf{(NH1)}] For any $q \in \mathcal{Q}$ and $v_{0} \in \mathfrak{A}(q)$ there exists a family of closed subspaces $\mathfrak{A}^{v}_{L}(q,v)$ such that $\mathbb{E} = \mathfrak{A}^{v}_{L}(q,v_{0}) \oplus \mathcal{T}_{v_{0}}\mathfrak{A}(q)$ and $L(t;q,v_{0}) \mathfrak{A}^{v}_{L}(q,v_{0}) \subset \mathfrak{A}^{v}_{L}(\vartheta^{t}(q),\psi^{t}(q,v_{0}))$ for all $t \geq 0$.
	\item[\textbf{(NH2)}] The corresponding to the decomposition $\mathbb{E} = \mathfrak{A}^{v}_{L}(q,v_{0}) \oplus \mathcal{T}_{v_{0}}\mathfrak{A}(q)$ projector $\Pi^{c}_{L}(\cdot;q,v_{0}) \colon \mathbb{E} \to \mathcal{T}_{v_{0}}\mathfrak{A}(q)$ depend continuously on $(q,v_{0})$ in the norm of $\mathcal{L}(\mathbb{E})$ and the norms are bounded uniformly in $(q,v_{0})$.
	\item[\textbf{(NH3)}] There are a constant $M_{h}>0$ and two positive bounded continuous functions $\varkappa^{h}_{0},\varkappa^{v}_{0} \colon \mathcal{Q} \times \mathbb{E} \to \mathbb{R}_{+}$ such that $\varkappa^{h}_{0}(\cdot) < \varkappa^{v}_{0}(\cdot)$ and for all $t \geq 0, q \in \mathcal{Q}, v_{0} \in \mathfrak{A}(q)$ and $\xi \in \mathbb{E}$ we have
	\begin{equation}
		\| L(t;q,v_{0}) (\xi-\Pi^{c}_{L}(\xi;q,v_{0})) \|_{\mathbb{E}} \leq M_{h}e^{-\varkappa^{v}(t;q,v_{0})} \| \xi \|_{\mathbb{E}}
	\end{equation}
	and
	\begin{equation}
		\| L(-t;q,v_{0}) \Pi^{c}_{L}(\xi;q,v_{0}) \|_{\mathbb{E}} \leq M_{h} e^{\varkappa^{h}(t; \pi^{-t}(q,v_{0}))} \| \xi \|_{\mathbb{E}},
	\end{equation}
    where $\varkappa^{h}(t;q,v_{0}) := \int_{0}^{t}\varkappa^{h}_{0}(\pi^{s}(q,v_{0}))ds$ and $\varkappa^{v}(t;q,v_{0}) := \int_{0}^{t}\varkappa^{v}_{0}(\pi^{s}(q,v_{0}))ds$.
\end{description}

To prove that the family $\mathfrak{A}(q)$ is normally hyperbolic we will assume that the linearization cocycle over $\mathfrak{A}$ satisfies $\textbf{(H3)}^{-}_{w}$ and $\textbf{(H3)}^{+}_{w}$ in two ways with distinct, but related, pairs of exponents.
\begin{description}
	\item[\textbf{(HYP)}] The linearization cocycle $\Xi$ satisfies $\textbf{(H3)}^{-}_{w}$ and $\textbf{(H3)}^{+}_{w}$ for two families of operators $Q_{i}(q,v_{0})$, $i=1,2$; with the same as $\operatorname{dim}\mathfrak{A}(q)$ dimension of the negative spaces; with positive functions $\alpha_{0}(\cdot) = \sigma^{h}_{0,i}(\cdot)$ and $\beta_{0}(\cdot) = \sigma^{v}_{0,i}(\cdot)$ such that $\sigma^{h}_{0,1}(\cdot) < \sigma^{v}_{0,2}(\cdot)$, $\sigma^{v}_{0,1}(\cdot) \leq \sigma^{v}_{0,2}(\cdot)$ (or $\sigma^{v}_{0,2}(\cdot) \leq \sigma^{v}_{0,1}(\cdot)$) and $\sigma^{h}_{0,i}(\cdot)$ being continuous. Moreover, let $C_{(q,v_{0})}=C_{(q,v_{0})}(Q_{i}(q,v_{0}))$ (defined in \eqref{EQ: ConstantProjectorNorm}) be such that $C^{-1}_{(q,v_{0})}$ are bounded from above uniformly in $(q,v_{0})$, \textbf{(ULIP)}, \textbf{(ACOM)} with $\gamma^{+} = \sup_{q \in \mathcal{Q},v_{0} \in \mathfrak{A}(q)}\sigma^{v}_{0,i}(q,v_{0})$ be satisfied and suppose
	\begin{equation}
		\label{EQ: HYPtangetSpaceRelation}
		\mathcal{T}_{v_{0}}\mathfrak{A}(q) = [0^{*}(\cdot)](0),
	\end{equation}
    where $[\cdot]$ is the class of negative pseudo-ordering (see Subsection \ref{SUBSUBSECTION: NegativePseudoOrderedOrbits}) for the zero complete trajectory $0^{*}(\cdot)$ over $(q,v_{0})$ of $\Xi$ w.~r.~t. any family of operators $Q_{i}(q)$.
\end{description}

Firstly, let us note that under \textbf{(HYP)}, the positive equivalence class $[0]^{+}(q,v_{0})$ of $0$ over $(q,v_{0})$ and, consequently, the central projector $\Pi^{c}_{L}(\cdot;q,v_{0})$ onto $\mathcal{T}_{v_{0}}\mathfrak{A}(q)$ does not depend on the family of operators (i.e. on $i$). Indeed, since $\sigma^{v}_{0,1}$ and $\sigma^{v}_{0,2}$ are related, Theorem \ref{LEM: PositiveEquivalence} gives us a relation for the positive equivalence classes. However, any such class together with the admissible (for both families due to \eqref{EQ: HYPtangetSpaceRelation}) set $\mathcal{T}_{v_{0}}\mathfrak{A}(q)$ gives a direct sum decomposition of $\mathbb{E}$ as Theorem \ref{TH: ExpTrackingTheorem} states. Thus, they must coincide.
\begin{theorem}
	\label{TH: NormalHyperbolicity}
	Let $\Xi$ satisfy \textbf{(HYP)} and \textbf{(CD)} (for at least one $i=1,2$). Then the family $\mathfrak{A}(q)$ is normally hyperbolic. Namely, in the above notation it satisfies \textbf{(NH1)} with $\mathfrak{A}^{v}_{L}(q,v_{0}) := [0]^{+}(q,v_{0})$, \textbf{(NH2)} with $\Pi^{c}_{L}(\cdot;q,v_{0})$ being the projector along $\mathfrak{A}^{v}_{L}(q,v_{0})$ onto $\mathcal{T}_{v_{0}}\mathfrak{A}(q)$ and \textbf{(NH3)} with some $M_{h}>0$, $\varkappa^{v}_{0}(\cdot) := \sigma^{v}_{0,2}(\cdot)$ and $\varkappa^{h}_{0}(\cdot) := \sigma^{h}_{0,1}(\cdot)$.
\end{theorem}
\begin{proof}
	Since $\Pi^{c}_{L}(\cdot;q,v_{0})$ is a linear projector onto $\mathcal{T}_{v_{0}}\mathfrak{A}(q) = [0^{*}(\cdot)](0)$ with $\operatorname{Ker}\Pi^{c}_{L}(\cdot;q,v_{0}) = [0]^{+}(q,v_{0})$ (see Theorem \ref{TH: ExpTrackingTheorem}), \textbf{(NH1)} is satisfied. From Theorem \ref{TH: ContinuousDependenceFibres} and Theorem \ref{TH: CentralProjectorContinuityGeneral} applied for a proper $i$, where \textbf{(CD)} is satisfied, we get that $\Pi^{c}_{L}(\cdot;q,v_{0})$ depends continuously on $q \in \mathcal{Q}$ and $v_{0} \in \mathfrak{A}(q)$. Moreover, \eqref{EQ: ExponentialTrackingEstimate} with proper terms and $t = 0$ shows that the norms $\Pi^{c}_{L}(\cdot;q,v_{0})$ are bounded uniformly in $(q,v_{0})$ (since $C^{-1}_{(q,v_{0})}(Q_{i}(q,v_{0}))$ are uniformly bounded from above) and, consequently, \textbf{(NH2)} holds.
	
	For the second family of operators ($i=2$), we use \eqref{EQ: ExponentialTrackingEstimate} and the uniform boundedness from above of $C^{-1}_{(q,v_{0})}(Q_{2}(q,v_{0}))$ to obtain the first estimate from \textbf{(NH3)} with $\varkappa^{v}_{0}(\cdot) := \sigma^{v}_{0,2}(\cdot)$ and some constant $M_{h}>0$.
	
	Now we consider the first family of operators ($i=1$). Let us show that the second estimate from \textbf{(NH3)} holds with $\varkappa^{h}_{0}(\cdot) := \sigma^{h}_{0,1}(\cdot)$. Note that there are constants $M_{1}>0$ and $M_{2}>0$ such that
	\begin{equation}
		\label{EQ: NormalHypFormNormEstimate}
		M_{1} \| \xi \|^{2}_{\mathbb{E}} \leq -V_{(q,v_{0})}(\xi) \leq M_{2} \| \xi \|^{2}_{\mathbb{E}}.
	\end{equation} 
    is satisfied for all $q \in \mathcal{Q}$, $v_{0} \in \mathfrak{A}(q)$ and $\xi \in \mathcal{T}_{v_{0}}\mathfrak{A}(q)$. The second inequality is obvious and the first one can be derived from an analog of \eqref{EQ: OrderedCompleteTrajEstimate} using the mean value theorem and \textbf{(ULIP)} as we did many times. Using $\textbf{(H3)}^{-}_{w}$ with the terms for $i=1$, we get that for all $\xi_{0} \in \mathbb{E}$
    \begin{equation}
    	-V_{(q,v_{0})}(L(t;\pi^{-t}(q,v_{0}))\xi_{0}) \geq - e^{-2\varkappa^{h}(t;\pi^{-t}(q,v_{0}))} V_{\pi^{-t}(q,v_{0})}(\xi_{0}).
    \end{equation}
    Combining this with \eqref{EQ: NormalHypFormNormEstimate} and putting $\xi_{0} := L(-t;q,v_{0})\Pi^{c}_{L}(\xi; q,v_{0})$, we get the second inequality from \textbf{(NH3)}. The proof is finished.
\end{proof}

\begin{problem}
	Suppose $\psi$ is $C^{k}$-differentiable and let $\mathfrak{A}(q)$ be a normally hyperbolic family with the exponents $\varkappa^{h}(q,v_{0})$ and $\varkappa^{v}(q,v_{0})$ as in \textbf{(NH3)}. Suppose that
	\begin{equation}
		\sup_{q \in \mathcal{Q}, v_{0} \in \mathfrak{A}(q)}\frac{\varkappa^{v}(q,v_{0})}{\varkappa^{h}(q,v_{0})} > k.
	\end{equation}
    Show that $\mathfrak{A}(q)$ is $C^{k}$-differentiable submanifold.
\end{problem}

\subsection{Stability of invariant sets and invertibility}
\label{SUBSEC: StabilityReducedDynamics}
In this subsection we starting working under the hypotheses of Theorem \ref{TH: ContinuousDependenceFibres} with $\alpha^{-} > 0$.

Let $\mathcal{K} = \{ \mathcal{K}(q) \}$, where $q \in \mathcal{K}$, be a family of bounded subsets of $\mathbb{E}$, which is invariant w.~r.~t. the cocycle, i.~e. $\psi^{t}(q, \mathcal{K}(q)) = \mathcal{K}(\vartheta^{t}(q))$ for all $q \in \mathcal{Q}$ and $t \geq 0$. We say that $\mathcal{K}$ is \textit{Lyapunov stable} (respectively \textit{amenably Lyapunov stable}) if for every $\varepsilon>0$ and any $q \in \mathcal{Q}$ there exists $\delta>0$ such that if $v_{0} \in \mathcal{O}_{\delta}(\mathcal{K}(q))$ (respectively $v_{0} \in \mathcal{O}_{\delta}(\mathcal{K}(q)) \cap \mathfrak{A}(q)$) then $\psi^{t}(q,v_{0}) \in \mathcal{O}_{\varepsilon}(\mathcal{K}(\vartheta^{t}(q)))$ for all $t \geq 0$. In the case when $\delta$ from the previous definition can be chosen independently of $q$, we say that $\mathcal{K}$ is \textit{uniformly Lyapunov stable} (respectively \textit{amenably uniformly Lyapunov stable}). A (uniformly) Lyapunov stable family $\mathcal{K}$ is called \textit{(uniformly) asymptotically Lyapunov stable} (respectively \textit{amenably (uniformly) asymptotically Lyapunov stable}) if there exists $\delta_{as}=\delta_{as}(q)>0$ (independent of $q$ in the uniform case) such that $\operatorname{dist}(\psi^{t}(q,v_{0}),\mathcal{K}(q)) \to 0$ as $t \to +\infty$ provided that $v_{0} \in \mathcal{O}_{\delta_{as}}(\mathcal{K}(q))$ (respectively $v_{0} \in \mathcal{O}_{\delta_{as}}(\mathcal{K}(q)) \cap \mathfrak{A}(q)$).

\begin{theorem}
	\label{TH: StabilityOfInvariantSets}
	Under the hypotheses of Theorem \ref{TH: ExponentialTrackingIM} with $\alpha^{-}>0$ and $\beta^{-} > 0$, let $\mathcal{K} = \{ \mathcal{K}(q) \}$ be an invariant family of uniformly bounded in $q \in \mathcal{Q}$ subsets. Then $\mathcal{K}$ is (asymptotically) Lyapunov stable if and only if $\mathcal{K}$ is amenably (asymptotically) Lyapunov stable. Moreover, if $C^{-1}_{q}$'s are uniformly in $q \in \mathcal{Q}$ bounded from above, then $\mathcal{K}$ is uniformly (asymptotically) Lyapunov stable if and only if $\mathcal{K}$ is amenably uniformly (asymptotically) Lyapunov stable.
\end{theorem}
\begin{proof}
	The only non-trivial part is ``if''. Let $q$ be fixed. For $v_{0} \in \mathbb{E}$ and $t \geq 0$ the triangle inequality gives
	\begin{equation}
		\label{EQ: StabilityTheorem1}
		\operatorname{dist}( \psi^{t}(q,v_{0}),\mathcal{K}(\vartheta^{t}(q))) \leq \| \psi^{t}(q,v_{0}) - \psi^{t}(q,v^{*}_{0}) \|_{\mathbb{E}} + \operatorname{dist}( \psi^{t}(q,v^{*}_{0}),\mathcal{K}(\vartheta^{t}(q))),
	\end{equation}
    where $v^{*}_{0}$ is given by Theorem \ref{TH: ExponentialTrackingIM}. Since $\mathcal{K}(q) \subset \mathfrak{A}(q)$ and $\beta(t;q) \geq \beta^{-}t$, from \eqref{EQ: ExponentialTrackingEstimate} we have for all $t \geq 0$ and $R(q):=R(C^{-1}_{q},(1-\varkappa_{0})^{-1},C_{Lip}(q))$
    \begin{equation}
    	\label{EQ: StabilityTheorem2}
    	\| \psi^{t}(q,v_{0}) - \psi^{t}(q,v^{*}_{0}) \|_{\mathbb{E}} \leq R(q) \operatorname{dist}(v_{0},\mathcal{K}(q)) e^{-\beta^{-} t}
    \end{equation}
    and using it with $t = 0$, we get
    \begin{equation}
    	\label{EQ: StabilityTheorem3}
    	\operatorname{dist}(v^{*}_{0},\mathcal{K}(q)) \leq \|v_{0}-v^{*}_{0}\|_{\mathbb{E}} + \operatorname{dist}(v_{0},\mathcal{K}(q)) \leq (R(q) + 1)\operatorname{dist}(v_{0},\mathcal{K}(q)).
    \end{equation}
    Let $\varepsilon>0$ be fixed and let $\delta'>0$ be such that $\psi^{t}(q,v_{0}) \in \mathcal{O}_{\varepsilon/2}$ for all $t \geq 0$ provided that $v_{0} \in \mathcal{O}_{\delta'}(\mathcal{K}(q))$. Now let $\delta>0$ be such that $\delta \cdot (R(q) + 1) < \delta'$ and $R(q) \delta < \varepsilon/2$. Then for $v_{0} \in \mathcal{O}_{\delta}(\mathcal{K}(q))$ from \eqref{EQ: StabilityTheorem1}, \eqref{EQ: StabilityTheorem2} and \eqref{EQ: StabilityTheorem3} we have
    \begin{equation}
    	\operatorname{dist}( \psi^{t}(q,v_{0}),\mathcal{K}(\vartheta^{t}(q))) < \frac{\varepsilon}{2} + \frac{\varepsilon}{2} = \varepsilon.
    \end{equation}
    Note that $\delta$ can be chosen uniformly in $q \in \mathcal{Q}$ if $C^{-1}_{q}$ are uniformly in $q \in \mathcal{Q}$ bounded from above and $\delta_{0}$ can be chosen uniformly in $q$. Moreover, the equivalence of (uniform) asymptotic Lyapunov stabilities follows from the exponential tracking given by \eqref{EQ: StabilityTheorem2}.
\end{proof}
A part of Theorem \ref{TH: StabilityOfInvariantSets} was implicitly (he was unfamiliar with the exponential tracking property) used by R.A.~Smith to show the existence of orbitally stable periodic orbits for semiflows \cite{Smith1994PB2, Smith1992, Smith1987OrbStab}.

Now let us consider the set $\mathfrak{A} = \bigcup_{q \in \mathcal{Q}} \{ q \} \times \mathfrak{A}(q)$ and the restriction of the skew-product semiflow $\pi$ to $\mathfrak{A}$. Recall that $\pi^{t}(q,v) = (\vartheta^{t}(q),\psi^{t}(q,v))$ for $(q,v) \in \mathcal{Q} \times \mathbb{E}$ and $t \geq 0$. Since the cocycle map is invertible due to item 2 of Theorem \ref{TH: PrincipalLeaveConstruction}, the map $\pi^{t} \colon \mathfrak{A} \to \mathfrak{A}$ for $t \geq 0$ is invertible. We will denote its inverse by $\pi^{-t}$. Thus the maps $\pi^{t} \colon \mathfrak{A} \to \mathfrak{A}$, where $t \in \mathbb{R}$, form a group of transformations.
\begin{theorem}
	\label{TH: InvertibilityTheoremBA}
	Under the hypotheses of Theorem \ref{TH: ContinuousDependenceFibres}, the maps $\pi^{t} \colon \mathfrak{A} \to \mathfrak{A}$, where $t \in \mathbb{R}$, form a flow.
\end{theorem}
\begin{proof}
	Due to the group property it is sufficient to prove that the map $\pi^{t} \colon \mathfrak{A} \to \mathfrak{A}$ for all $t < 0$ is continuous. Let a sequence $(q_{k},v_{k}) \in \mathfrak{A}$, where $k=1,2,\ldots$, converge to some $(q,v) \in \mathfrak{A}$ as $k \to +\infty$. Let $v_{k}(\cdot)$ be the amenable trajectory passing through $v^{*}_{k}$ over $q_{k}$ and let $v^{*}(\cdot)$ be the amenable trajectory passing through $v$ over $q$. We have to show that $v^{*}_{k}(t) \to v^{*}(t)$ as $k \to +\infty$. Supposing the contrary, we obtain a subsequence (we keep the same index) of $v^{*}_{k}(t)$ which is separated from $v^{*}_{k}(t)$ for all $k=1,2,\ldots$. Now using \textbf{(BA)} and \textbf{(ACOM)} we can obtain (in the same way as in the proof of Theorem \ref{TH: ContinuousDependenceFibres} or using its statement and similar estimates) a subsequence of $v_{k}(\cdot)$ converging to some amenable trajectory over $q$, which must coincide with $v^{*}(\cdot)$ since $v^{*}_{k}(0)=(v_{k},q_{k}) \to (v,q)=v^{*}(0)$ as $k \to +\infty$. Thus, we get a contradiction.
\end{proof}
In applications, the invertibility given by Theorem \ref{TH: InvertibilityTheoremBA} is connected to that the dynamics on $\mathfrak{A}$ can be described by an ordinary differential equation with a globally Lipschitz vector field.

\subsection{Robustness of foliations}
\label{SUBSEC: RobustnessManifolds}
During this section we work with a cocycle $\psi$ in $\mathbb{E}$ over a flow $(\mathcal{Q},\vartheta)$. Let $\psi_{\varepsilon}$, $\varepsilon \in [0,1]$, be a family of cocycles in $\mathbb{E}$ over the same driving system which to be considered as a perturbation of $\psi$ and such as $\psi_{\varepsilon} \to \psi$ as $\varepsilon \to 0$. For convenience, we put $\psi_{0} := \psi$. Our aim is to provide conditions under which $\psi_{\varepsilon}$ has an inertial manifold, which depend in some sense continuously as $\varepsilon \to 0$.

As our first assumption we will require that the conditions for the existence of inertial manifolds are satisfied for each cocycle $\psi_{\varepsilon}$ with in a sense uniform terms.
\begin{description}
	\item[(RF1)] Suppose that for every $\varepsilon \in [0,1]$ the cocycle $\psi_{\varepsilon}$ satisfies $\textbf{(H3)}^{-}_{w}$ with the same for all $\varepsilon$ terms, $\alpha_{-}>0$ and $\alpha_{0}(\cdot)$ being continuous. Moreover, suppose that \textbf{({PROJ})}, \textbf{(CD)} hold and let \textbf{(ACOM)} with $\gamma^{+} = \alpha^{+}$ be satisfied uniformly in $\varepsilon$, i.~e. \eqref{EQ: AsymptoticCompactnessEstimate} holds as
	\begin{equation}
		\alpha_{K}\left( \bigcup_{\varepsilon \in [0,1]} \psi^{t}(\mathcal{C},\mathcal{B}) \right) \leq \gamma_{0}(t)e^{-\gamma^{+}t}\alpha_{K}(\mathcal{B}).
	\end{equation}
	Finally, let \textbf{(ULIP)} also hold for each cocycle with uniform bounds for the constant $L_{T}$,  and \textbf{(BA)} with possibly different trajectories $w^{(\varepsilon)}_{q}(\cdot)$ of $\psi_{\varepsilon}$, but with a uniform in $\varepsilon$ bound $M_{b}>0$.
\end{description}

\begin{description}
	\item[(RF2)] For any $T>0$ and any compact $\mathcal{K} \in \mathcal{Q}$ and bounded subset $\mathcal{B} \subset \mathbb{E}$ we have that $\psi^{t}_{\varepsilon}(q,v) \to \psi^{t}(q,v)$ as $\varepsilon \to 0$ uniformly in $t \in [0,T]$, $q \in \mathcal{K}$ and $v \in \mathcal{B}$.
\end{description}

\begin{theorem}
	\label{TH: RobustnessIM}
	Let \textbf{(RF1)} and \textbf{(RF2)} be satisfied. Then for any compact $\mathcal{K} \subset \mathcal{Q}$ and bounded subset $\mathcal{B} \subset \mathbb{E}$ we have $\Phi^{(\varepsilon)}_{q}(\zeta) \to \Phi_{q}(\zeta)$ as $\varepsilon \to 0$ uniformly in $q \in \mathcal{K}$ and $\zeta \in \mathbb{E}^{-}(q) \cap \mathcal{B}$.
\end{theorem}
\begin{proof}
	Supposing the contrary, one can show the statement analogously to the proof of Theorem \ref{TH: ContinuousDependenceFibres}.
\end{proof}
\begin{remark}
	As in Remark \ref{REM: ContinuousDependenceExponentRestriction}, one can relax the assumption $\alpha^{-}>0$ from \textbf{(RF1)} by requiring \textbf{(BA)} to be satisfied with $w^{(\varepsilon)}_{q}(\cdot) \equiv 0$ for all $q \in \mathcal{Q}$ and $\varepsilon \in [0,1]$. This gives the robustness of unstable manifolds.
\end{remark}

Analogously, one can study continuity of the tangent spaces to $\mathfrak{A}(q)$ by assuming some sort of uniform convergence of the linearization cocycles $\Xi_{\varepsilon}$ for the perturbed cocycle $\psi_{\varepsilon}$. Since techniques are the same, we left this to the reader as an exercise.

\subsection{Slow and center manifolds}
\label{SUBSEC: SlowManifolds}
In this subsection we discuss how to further foliate the principal leaf $\mathfrak{A}(q)$ at each point under additional squeezing with lower rate w.r.t. another family of operators. In applications, this allows to construct perturbations of any spectral subspaces at least in a small neighborhood of a stationary state. In particular, this gives a center manifold theorem.

\begin{description}
	\item[\textbf{(SM)}] For $i=1,2$ the cocycle $\psi$ satisfies \textbf{(ULIP)}, $\textbf{(H3)}^{-}_{w}$ and $\textbf{(H3)}^{+}_{w}$ for two families of operators $Q(q) = Q_{i}(q)$ with $\alpha_{0}(q) = \alpha_{0,i}(q)$, $\alpha^{+} = \alpha^{+}_{i}$ $\beta_{0}(q)=\beta_{0,i}(q)$ and the dimensions of the negative spaces $\mathbb{E}(q) = \mathbb{E}_{-}(q)$ equal to $j_{Q} = j_{i}$. Moreover, $\alpha_{0,1}(\cdot) \geq \alpha_{0,2}(\cdot)$, $\beta_{0,1}(\cdot) \geq \beta_{0,2}(\cdot)$, $j_{1} > j_{2}$, \textbf{(PROJ)}  holds for each family and \textbf{(PL)} holds for the first ($i=1$) family and \textbf{(ACOM)} with $\gamma^{+} = \alpha^{+}_{i}$ is satisfied.
\end{description}

Thus, under \textbf{(SM)} we can apply Theorem \ref{TH: PrincipalLeaveConstruction} to construct the principal leaf $\mathfrak{A}(q)=\mathfrak{A}_{1}(q)$ (for the family $Q_{1}(q)$) and over each $v_{0} \in \mathfrak{A}_{1}(q)$ we get the horizontal leaf $\mathfrak{A}^{h}_{2}(q,v_{0})$ constructed from the unique complete trajectory $v(\cdot)$ passing through $v_{0}$ and such that $v(t) \in \mathfrak{A}_{1}(\vartheta^{t}(q))$ for all $t \in \mathbb{R}$. Moreover, Theorem \ref{TH: ExpTrackingTheorem} gives as over each point $v_{0} \in \mathbb{E}$ a vertical leaf $\mathfrak{A}^{v}_{i}(q,v_{0}) := [v_{0}]^{+}$, where the positive equivalence class depends on $i$. Since $\alpha_{0,1}(\cdot) \geq \alpha_{0,2}(\cdot)$, we have $\mathfrak{A}_{1}(q) \supset \mathfrak{A}_{2}(q)$ and since $\beta_{0,1}(\cdot) \geq \beta_{0,2}(\cdot)$, we have $\mathfrak{A}^{v}_{1}(q) \subset \mathfrak{A}^{v}_{2}(q)$ for all $q \in \mathcal{Q}$. Moreover, under \textbf{(DIFF)} Theorems \ref{TH: HorizontalDifferentiability} and Theorem \ref{TH: VertialLeavesDifferentiability} give us the same inclusion for the tangent spaces. Consequently, since $\mathfrak{A}_{1}(q)$ and $\mathfrak{A}^{v}_{1}(q,v_{0})$ intersect transversally, the submanifolds $\mathfrak{A}_{1}(q)$ and $\mathfrak{A}^{v}_{2}(q,v_{0})$ also intersect transversally for every $v_{0} \in \mathfrak{A}_{1}(q)$ and $q \in \mathcal{Q}$. Consequently, the intersection is a $C^{1}$-sumanifold (see, for example, Corollary 73.50 in \cite{Zeidler1988}). We summarize this in the following theorem.
\begin{theorem}
	\label{TH: SlowManifoldTheorem}
	Suppose that \textbf{(SM)} and \textbf{(DIFF)} are satisfied. Then for any $q \in \mathcal{Q}$ and $v_{0} \in \mathfrak{A}_{1}(q)$, the intersection $\mathfrak{A}_{1}(q) \cap \mathfrak{A}^{v}_{2}(q,v_{0})$ is a $(j_{1}-j_{2})$-dimensional $C^{1}$-differentiable submanifold in $\mathbb{E}$.
\end{theorem}
Let us consider an example with $\alpha_{0,1}(\cdot) \equiv \beta_{0,1}(\cdot) \equiv \nu_{1} > 0$ and $\alpha_{0,2}(\cdot) \equiv \beta_{0,2}(\cdot) \equiv \nu_{2} < 0$. Then $\mathfrak{A}^{v}_{1}(q,v_{0})$ is the strongly stable manifold, $\mathfrak{A}_{1}(q)$ is the slow-strongly-unstable manifold, $\mathfrak{A}^{v}_{2}(q)$ is the strongly-stable-slow manifold, $\mathfrak{A}^{h}_{2}(q,v_{0})$ is the strongly unstable manifold and $\mathfrak{A}_{1}(q) \cap \mathfrak{A}^{v}_{2}(q,v_{0})$ is the slow manifold at $v_{0} \in \mathfrak{A}_{1}(q)$.

\section{Abstract applications: low-dimensional dynamics}
\label{SEC: AbstractApplic}
\subsection{Exponential stability and almost automorphy in almost periodic cocycles}
\label{SUBSEC: ExpStabAndAAAP}
In this section we suppose that $(\mathcal{Q},\vartheta)$ is a minimal almost periodic (in the sense of Bohr) flow.

For almost periodic cocycles a simple way to verify \textbf{(BA)} is given by the following lemma.
\begin{lemma}
	\label{LEM: ApSectionSemiBounded}
	Suppose that for some $q_{0} \in \mathcal{Q}$ there exists $v_{0} \in \mathbb{E}$ such that the trajectory $\mathbb{R}_{+} \ni t \mapsto \psi^{t}(q,v_{0}) \in \mathbb{E}$ lies in a compact set $\mathcal{K}$. Then for every $q \in \mathcal{Q}$ there exists a bounded complete trajectory $w^{*}_{q}(\cdot)$ over $q$ lying in $\mathcal{K}$.
\end{lemma}
\begin{proof}
	Since $\vartheta$ is minimal almost periodic, for every $q \in \mathcal{Q}$ there exists a sequence $\tau_{k} \to +\infty$ such that $\vartheta^{\tau_{k}}(q_{0}) \to q$ $k \to +\infty$. Consider the trajectories $w_{k}(t):=\psi^{t+\tau_{k}}(q_{0},v_{0})$ passing through $\psi^{\tau_{k}}(q_{0},v_{0})$ over $\vartheta^{\tau_{k}}(q_{0})$ and defined for $t \geq -\tau_{k}$. Since the sequence $w_{k}(0)$ lies in $\mathcal{K}$, there exists a converging subsequence and the corresponding to it limiting complete trajectory $w^{*}_{q}(\cdot)$ over $q$ lying in $\mathcal{K}$.
\end{proof}

We have the following theorem.
\begin{theorem}
	\label{TH: AlmostPeriodicityTh}
	Let the cocycle satisfy \textbf{(H1)},\textbf{(H2)} with $j=0$, \textbf{(H3)} with $\nu_{0}(\cdot)$ being continuous and $\nu^{-} > 0$, \textbf{(ACOM)} with $\gamma^{+} = \nu^{+}$, \textbf{(ULIP)} and \textbf{(CD)}. Suppose that for some $q_{0} \in \mathcal{Q}$ there exists $v_{0} \in \mathbb{E}$ such that the trajectory $\mathbb{R}_{+} \ni t \mapsto \psi^{t}(q,v_{0}) \in \mathbb{E}$ is compact. Then for every $q \in \mathcal{Q}$ there exists a unique almost periodic trajectory $v^{*}_{q}(\cdot)$ over $q$ which continuously depend on $q$ and exponentially stable in the sense that there exists a constant $M>0$ such that for any $v_{0} \in \mathbb{E}$, $q \in \mathcal{Q}$ and $t \geq 0$ we have
	\begin{equation}
		\label{EQ: ExponentialStabilityAPcocycles}
		\| \psi^{t}(q,v_{0}) - \psi^{t}(q,v^{*}_{q}(0)) \|_{\mathbb{E}} \leq M \cdot e^{\nu(t;q)} \cdot \| v_{0} - v^{*}_{q}(0) \|_{\mathbb{E}}.
	\end{equation}
\end{theorem}
\begin{proof}
	From Lemma \ref{LEM: ApSectionSemiBounded} we get that there exists a bounded complete trajectory $v^{*}_{q}(\cdot)$ over each $q \in \mathcal{Q}$. In particular, we have \textbf{(BA)} satisfied. Moreover, since $\nu^{-}>0$ and $j=0$, such a trajectory is unique (see Lemma \ref{LEM: AmenableTrajectories}) and the uniqueness guarantees its almost-periodicity by a standard argument via the Bochner theorem (see Section 1.2 in \cite{LevitanZhikov1982}). Since \textbf{(CD)} is satisfied, Corollary \ref{COR: OrtProjectosBounded} implies that \textbf{(PROJ)} holds. Moreover, the continuity of $\nu_{0}(\cdot)$ and Theorem \ref{TH: ContinuousDependenceFibres} guarantee that $v^{*}_{q}(0)$ depend continuously on $q \in \mathcal{Q}$.
	
	Using Theorem \ref{TH: ExponentialTrackingIM} and \eqref{EQ: ExponentialTrackingEstimate}, we get \eqref{EQ: ExponentialStabilityAPcocycles} since $C^{-1}_{q}$ are uniformly bounded from above due to the compactness of $\mathcal{Q}$. The proof is finished.
\end{proof}

Conditions \textbf{(H1)}, \textbf{(H2)} with $j=0$ and \textbf{(H3)} can be considered as an abstract generalization of the strong monotonicity property, which is well-known in the theory of almost periodic differential equations (see the monographs of B.M.~Levitan and V.V.~Zhikov \cite{LevitanZhikov1982} or A.A.~Pankov). Namely, one can think that these conditions guarantee the existence of a family of positive-definite (not necessarily coercive!) bilinear forms that provide the strong monotonicity. Thus, the abstract context of Theorem \ref{TH: AlmostPeriodicityTh} unifies and extends some classical results of B.P.~Demidovich \cite{PavlovTributeDemidovich2004}, V.A.~Yakubovich \cite{Yakubovich1964,Yakubovich1988Periodic}, B.M.~Levitan and V.V.~Zhikov \cite{LevitanZhikov1982}, A.A.~Pankov \cite{Pankov1990} and their further developments. 

In Section \ref{SEC: Applications} we will provide the so-called frequency conditions for the existence of a constant operator $P$ and a constant exponent $\nu_{0}$ in the case of particular classes of equations. Our main tool is the Frequency Theorem \cite{Anikushin2020FreqDelay,Anikushin2020FreqParab}. Moreover, V.A.~Yakubovich was first who considered concrete realizations of Theorem \ref{TH: AlmostPeriodicityTh} in the case when the operators $P(q)$ depend on $q$. Namely, in \cite{Yakubovich1988Periodic} he studied a certain class of periodic ODEs via the Frequency Theorem for periodic systems in $\mathbb{R}^{n}$.

In some papers (see, for example, the one by N.~Yu.~Kalinin and V.~Reitmann \cite{KalininReitmann2012}) there were problems, concerning with generalizations of the mentioned classical results to the infinite-dimensional context, caused by the lack of coercivity of bilinear forms. As our Theorem \ref{TH: AlmostPeriodicityTh} shows, more accurate considerations allow to avoid unnecessary additional constructions (as in \cite{KalininReitmann2012}) and to obtain more concrete and natural results.

Even in the simplest case described in Theorem \ref{TH: AlmostPeriodicityTh}, almost periodic minimal sets in the corresponding skew-product flow or their projections into $\mathbb{E}$ (uniform attractors) may have complicated fractal structure due to some nonregularity of the driving system $(\mathcal{Q},\vartheta)$ and/or number-theoretic phenomena \cite{Anikushin2019Liouv,AnikushinReitRom2019}. We also do not know how to generalize the quantitative results \cite{Anikushin2019Liouv,AnikushinReitRom2019}, which are highly based on the coercivity of the quadratic forms $P(q)$ in finite-dimensions, for infinite-dimensional systems.

Now it will be convenient in this section to speak in terms of the skew-product semiflow $\pi^{t} \colon \mathcal{Q} \times \mathbb{E} \to \mathcal{Q} \times \mathbb{E}$ associated with the cocycle $\psi$ as $\pi^{t}(q,v)=(\vartheta^{t}(q),\psi^{t}(q,v))$ for $t \geq 0$, $q \in \mathcal{Q}$ and $v \in \mathbb{E}$.

The following theorem generalizes the result of W.~Shen and Y.~Yi \cite{ShenYi1998AADyn} for scalar almost periodic ODEs and scalar parabolic equations in one-dimensional domains.
\begin{theorem}
	\label{TH: AAdynamicsInAPcocycles}
	Let the cocycle satisfy \textbf{(H1)} w.~r.~t. a constant operator $P$ and a constant decomposition $\mathbb{E}=\mathbb{E}^{+} \oplus \mathbb{E}^{-}$, \textbf{(H2)} with $j=1$, \textbf{(H3)} with $\nu_{0}(\cdot)$ being continuous and $\nu^{-} > 0$, \textbf{(ACOM)} with $\gamma^{+}=\nu^{+}$ and \textbf{(ULIP)}. Then the $\omega$-limit set (w. r. t. $\pi$) of any point with a bounded positive semi-orbit consists of at most $2$ minimal sets. Moreover, every compact minimal set of $\pi$ is almost automorphic.
\end{theorem}
Some details of the proof are given in \cite{Anikushin2021DiffJ}.

\subsection{Convergence and Lyapunov stability in periodic cocycles}
\label{SUBSEC: ConvergencePerCoc}
Here we suppose that $(\mathcal{Q},\vartheta)$ is a minimal $\sigma$-periodic flow.

\begin{theorem}
	\label{TH: PerCocyclesConvergence}
	Let the cocycle satisfy \textbf{(H1)} w.~r.~t. a constant operator $P$ and a constant decomposition $\mathbb{E}=\mathbb{E}^{+} \oplus \mathbb{E}^{-}$,\textbf{(H2)} with $j = 1$, \textbf{(H3)} with $\nu_{0}(\cdot)$ being continuous and $\nu^{-} > 0$, \textbf{(ACOM)} with $\gamma^{+}=\nu^{+}$ and \textbf{(ULIP)}. Then any bounded in the future trajectory $\mathbb{R}_{+} \ni t \mapsto \psi^{t}(q,v_{0})$ converges to a $\sigma$-periodic trajectory.
\end{theorem}
\begin{proof}
	Let $v_{0} \in \mathbb{E}$, $q \in \mathcal{Q}$ and $v(t):=\psi^{t}(q,v_{0})$ be a bounded trajectory. By Lemma \ref{LEM: ApSectionSemiBounded} we have \textbf{(BA)} satisfied. Let $v^{*}(t)=\psi^{t}(q,v^{*}_{0})$ be the amenable trajectory over $q$ given by Theorem \ref{TH: ExponentialTrackingIM}. Since $\nu^{-}>0$, we have $\|v(t)-v^{*}(t) \|_{\mathbb{E}} = O(e^{-\nu(t;q)})$ as $t \to +\infty$. Let us consider the Poincar\'{e} map $T \colon \mathfrak{A}(q) \to \mathfrak{A}(q)$ given by $T w := \psi^{\sigma}(q,w)$ for $w \in \mathfrak{A}(q)$. Since $\mathfrak{A}(q)$ is homeomorphic to $\mathbb{R}$, we may assume that $\mathfrak{A}(q)$ is endowed with the natural order induced from $\mathbb{R}$ by the homeomorphism (which is provided by $\Pi_{q}$). Since $\mathfrak{A} = \bigcup_{q \in \mathcal{Q}} \mathfrak{A}(q)$ is homeomorphic to $\mathcal{Q} \times \mathbb{R}$, it is easy to see that the sequence of iterates $T^{k}(v^{*}_{0})$, $k=1,2,\ldots$, either stationary (and thus $v^{*}(\cdot)$ is $\sigma$-periodic) or monotone and bounded (since $v(\cdot)$ is bounded). Thus, there exists a limiting trajectory $w^{*}(\cdot)$ over $q$, which is, clearly, $\sigma$-periodic. It is easy to see that $v^{*}(t)-w^{*}(t) \to 0$ and, consequently, $v(t) - w^{*}(t) \to 0$ as $t \to +\infty$. The proof is finished.
\end{proof}

\begin{exercise}
	In terms of Theorem \ref{TH: PerCocyclesConvergence} show that if the operators $P(q)$ and the decompositions $\mathbb{E} = \mathbb{E}^{+}(q) \oplus \mathbb{E}^{-}(q)$ from \textbf{(H1)} are not constant and \textbf{(CD)} is satisfied, then the convergence holds to $2 \sigma$-periodic trajectories. This doubling of the period is caused by the fact that the bundle $\mathfrak{A} = \bigcup_{q \in \mathcal{Q}} \mathfrak{A}(q)$ may be homeomorphic to a M\"{o}bius strip (see \cite{Anikushin2021DiffJ}).
\end{exercise}

We say that a closed bounded set $\mathcal{S} \subset \mathbb{E}$ is a \textit{sink} for the cocycle if there is an open set $\mathcal{U} \subset \mathbb{E}$ (called a neighborhood of $\mathcal{S}$) such that $\mathcal{S} \subset \mathcal{U}$ and for every $v_{0} \in \mathcal{U}$ and $q \in \mathcal{Q}$ there exists $t_{0}=t_{0}(q,v_{0})$ such that $\psi^{t}(q,v_{0}) \in \mathcal{S}$ for all $t \geq t_{0}$.

\begin{theorem}
	\label{TH: PeriodicCocyclesStability}
	Let the cocycle satisfy \textbf{(H1)} w.~r.~t. a constant operator $P$ and a constant decomposition $\mathbb{E}=\mathbb{E}^{+} \oplus \mathbb{E}^{-}$,\textbf{(H2)} with $j = 1$, \textbf{(H3)} with $\nu_{0}(\cdot)$ being continuous and $\nu^{-} > 0$, \textbf{(ACOM)} with $\gamma^{+}=\nu^{+}$, \textbf{(ULIP)} and \textbf{(CD)}. Suppose there exists a sink $\mathcal{S}$. Then there exists at least one $\sigma$-periodic trajectory, which is Lyapunov stable.
\end{theorem}
\begin{proof}
	Let $q \in \mathcal{Q}$. From our assumptions and Theorem \ref{TH: PerCocyclesConvergence} there exist at least one $\sigma$-periodic trajectory over $q$. Let $\Gamma \subset \mathfrak{A}(q)$ denote the set of all $v_{0} \in \mathcal{S}$ such that $t \mapsto \psi^{t}(q,v_{0})$ is $\sigma$-periodic. Clearly, $\Gamma$ is compact. Let $v^{*}_{-}:= \inf \Gamma$ (w.~r.~t. the induced order on $\mathfrak{A}(q)$ as in the proof of Theorem \ref{TH: PerCocyclesConvergence}). It is easy to see that $v^{*}_{-}$ is stable w.~r.~t. perturbations from $\mathfrak{A}(q) \cap \mathcal{U}$ from below, i.~e. for every $v_{0} \in \mathfrak{A}(q) \cap \mathcal{U}$ such that $v_{0} < v^{*}$ we have that $\psi^{t}(q,v_{0}) - \psi^{t}(q,v^{*}_{-}) \to 0$ as $t \to +\infty$. Indeed, otherwise there should exists a limiting $\sigma$-periodic trajectory $w^{*}(\cdot)$ over $q$ with the property $w^{*}(0) < v^{*}_{-}$ that contradicts to the construction of $v^{*}_{-} \in \Gamma$. Thus, the set $\Gamma^{-}$ of all points $v_{0} \in \mathfrak{A}(q) \cap \mathcal{S}$ that are stable from below in $\mathfrak{A}(q) \cap \mathcal{U}$ (in the above given sense), is not empty and it is also compact. Let $v^{*}_{+} := \sup \Gamma^{-}$. Then analogous arguments show that $v^{*}_{+}$ is stable from above in $\mathfrak{A}(q) \cap \mathcal{U}$ and, consequently, it is amenably Lyapunov stable. The Lyapunov stability now follows from Theorem \ref{TH: StabilityOfInvariantSets} and the proof is finished.
\end{proof}

We left the following lemma as an exercise for the reader.
\begin{lemma}
	Under the hypotheses of Theorem \ref{TH: PerCocyclesConvergence} any isolated Lyapunov stable $\sigma$-periodic trajectory is asymptotically Lyapunov stable.
\end{lemma}

Theorem \ref{TH: PerCocyclesConvergence} generalizes the well-known convergence theorem of J.L.~Massera \cite{Massera1950} for scalar periodic ODEs. Both theorems contain and extend the results of R.A.~Smith for periodic ODEs and delay equations \cite{Smith1986Massera,Smith1990ConvDelay} and their development from our previous work \cite{Anikushin2020Red}.

For $j=2$ one can obtain an extension of the second Massera theorem on the existence of $\sigma$-periodic orbits for planar $\sigma$-periodic systems as follows.
\begin{theorem}
	Let the cocycle satisfy \textbf{(H1)},\textbf{(H2)} with $j = 2$, \textbf{(H3)}, \textbf{(ACOM)} with $\gamma^{+}=\nu^{+}$, \textbf{(ULIP)} and \textbf{(CD)}. If there exists a compact trajectory, then there exists a $\sigma$-periodic trajectory.
\end{theorem}
\begin{proof}
	Let us suppose there is a compact trajectory over some $q \in \mathcal{Q}$. Then there exists a bounded (and, consequently, amenable) complete trajectory $v^{*}(\cdot)$ over $q$. Thus, Theorem \ref{TH: PrincipalLeaveConstruction} guarantees that the set $\mathfrak{A}(q)$ is homeomorphic to $\mathbb{R}^{2}$. Thus, the iterations of $v^{*}(0) \in \mathfrak{A}(q)$ are bounded under the Poincar\'{e} map $T \colon \mathfrak{A}(q) \to \mathfrak{A}(q)$ and, consequently, $T$ has a fixed point (a complete proof can be found in the monograph of V.A.~Pliss\cite{Pliss1966}). Clearly, this point corresponds to a $\sigma$-periodic orbit. The proof is finished.
\end{proof}

\subsection{The Poincar\'{e}-Bendixson theory for semiflows}
\label{SUBSEC: PoincareBendixsonTheory}
During this subsection we work with a semiflow in $\mathbb{E}$ which we denote by $\varphi$ and its time-$t$ map by $\varphi^{t}$.

Let us start with the following theorem.
\begin{theorem}
	\label{TH: PBTrichotomy}
	Let the semiflow $\varphi$ satisfy \textbf{(H1)}, \textbf{(H2)} with $j=2$, \textbf{(H3)}, \textbf{(ACOM)} with $\gamma=\nu^{+}$ and \textbf{(ULIP)}. Then the $\omega$-limit set $\omega(v_{0})$ of any point $v_{0} \in \mathbb{E}$ with a bounded semiorbit is one of the following:
	\begin{description}
		\item[\textbf{(PBT1)}] A stationary point;
		\item[\textbf{(PBT2)}] A periodic orbit;
		\item[\textbf{(PBT3)}] A union of some set of stationary points $\mathcal{N}$ and a set of complete orbits whose $\alpha$- and $\omega$-limit sets lie in $\mathcal{N}$.
	\end{description}
\end{theorem}
Since under the assumptions of Theorem \ref{TH: PBTrichotomy} there exists a two-dimensional inertial manifold $\mathfrak{A}$, which is homeomorphic to the plane $\mathbb{R}^{2}$ (see Theorem \ref{TH: PrincipalLeaveConstruction}), and the exponential tracking property holds (see Theorem \ref{TH: ExponentialTrackingIM}), its conclusion follows from the standard plane arguments. However, for this it is required to construct transversals at nonstationary points. Such a theory is developed by O.~H\'{a}jek \cite{Hajek1968}. However, it is only useful in the abstract context, where we have no differential equation describing the dynamics of $\varphi$ on the inertial manifold (the so-called inertial form). In applications, we always have Lipschitz or $C^{1}$-differentiable inertial forms and, consequently, transversals can be constructed via more standard methods and without appealing to \cite{Hajek1968}.

The following lemmas are obvious consequences of Theorem \ref{TH: PBTrichotomy}.
\begin{lemma}
	\label{COR: Th1CorAlphaLimit}
	Under the hypotheses of Theorem \ref{TH: PBTrichotomy} the trichotomy \textbf{(PBT1)}, \textbf{(PBT2)} and \textbf{(PBT3)} holds for the $\alpha$-limit set of any complete trajectory bounded in the past.
\end{lemma}
\begin{lemma}
	\label{COR: Th1CorIsolatedOrbit}
	Under the hypotheses of Theorem \ref{TH: PBTrichotomy} any isolated orbitally stable periodic orbit is asymptotically orbitally stable. 
\end{lemma}

Let $\mathcal{D}$ be some set containing a stationary point $v^{*}_{0}$. We say that $\mathcal{D}$ is a $k$-\textit{dimensional local unstable set} for $v^{*}_{0}$ if 
\begin{description}
	\item[\textbf{(U1)}] $\mathcal{D}$ is a homeomorphic image of some open $k$-dimensional cube;
	\item[\textbf{(U2)}] For every point $w_{0} \in \mathcal{D}$ there is a unique complete trajectory $w(\cdot)$ with $w(0)=w_{0}$ and $\varphi^{t}(w_{0}):=w(t) \to v^{*}_{0}$ as $t \to -\infty$;
	\item[\textbf{(U3)}] For every $\varepsilon>0$ there exists $\delta>0$ such that if $\|w_{0} - v^{*}_{0}\|_{\mathbb{E}}<\delta$ and $w_{0} \in \mathcal{D}$ then $\|\varphi^{t}(w)-v^{*}_{0}\|_{\mathbb{E}}<\varepsilon$ for all $t \leq 0$.
\end{description}

We call a stationary point $v^{*}_{0} \in \mathbb{E}$ \textit{terminal} if either it is Lyapunov stable or there is a $2$-dimensional unstable set for $v^{*}_{0}$. In the latter case we call $v^{*}_{0}$ an \textit{unstable terminal point}. The role of terminal points is in the following.
\begin{lemma}
	\label{COR: TerminalPoints}
	Under the hypotheses of Theorem \ref{TH: PBTrichotomy} suppose that $\omega(v_{0})$ contains a terminal point $v^{*}_{0}$. Then $\omega(v_{0}) = \{ v^{*}_{0} \}$.
\end{lemma}

We call a closed bounded subset $\mathcal{A} \subset \mathbb{E}$ an \textit{attractor} if there exists an open set $\mathcal{U} \subset \mathbb{E}$ such that $\mathcal{A} \subset \mathcal{U}$ and for any $v_{0} \in \mathcal{U}$ the positive semiorbit $\gamma^{+}(v_{0})$ is compact in $\mathbb{E}$ and $\omega(v_{0}) \subset \mathcal{A}$. We call any such set $\mathcal{U}$ a \textit{neighborhood} of the attractor $\mathcal{A}$.
\begin{theorem}
	\label{TH: PeriodicOrbitExis}
	Let the hypotheses of Theorem \ref{TH: PBTrichotomy} hold. Suppose there is an attractor $\mathcal{A}$ that either contains no stationary points or the only stationary points in $\mathcal{A}$ are unstable terminal points. Then the set $\mathcal{A}$ contains at least one orbitally stable periodic orbit.
\end{theorem}
The conclusion of Theorem \ref{TH: PeriodicOrbitExis} can be obtained by combining standard arguments for flows on the plane and Theorem \ref{TH: StabilityOfInvariantSets}, which guarantees the equivalence of stability properties for invariant sets on $\mathfrak{A}$ and entire $\mathbb{E}$.

The above theorems explains the geometry behind the works of R.A.~Smith on the Poincar\'{e}-Bendixson theory \cite{Smith1994PB2,Smith1994PB1,Smith1992,Smith1987OrbStab}. In the case $\varphi$ is $C^{1}$-differentiable our approach allows to extend the results from another series of papers by R.A.~Smith concerned with the Poincar\'{e} index theorem \cite{Smith1981IndexTh1,Smith1984PBindex} and isolated periodic orbits \cite{Smith1984IsolatedOrbits}.

An interesting application of R.A.~Smith's theory for the case of ODEs is given by I.M.~Burkin \cite{Burkin2015} (see also I.M.~Burkin and N.N.~Khien \cite{Burkin2014Hidden}) for localization of hidden oscillations in control systems (see also the survey of G.A.~Leonov and N.V.~Kuznetsov \cite{LeoKuz2013Hidden}). As we have showed, these ideas have a natural generalization to the infinite-dimensional context. We refer to our work \cite{Anikushin2021SS} for an example.
\section{Applications to differential equations}
\label{SEC: Applications}

In applications given in the next subsections we consider only the case of a constant operator $P$ and a constant decomposition $\mathbb{E} = \mathbb{E}^{-} \oplus \mathbb{E}^{+}$. We also do not discuss the existence of the so-called \textit{inertial form}, i.~e. the system of ordinary differential equations in $\mathbb{E}^{-}$, which describes the dynamics on the inertial manifold, and leave this to the reader as an exercise. Let us add here some remarks concerning this for the autonomous case. We consider equations of the form
\begin{equation}
	\label{EQ: ApplicationsExampleDiffEq}
	\dot{v} = Av(t)+BF(Cv(t))
\end{equation}
To get the equations on the inertial manifold $\mathfrak{A}$ one should consider the parametrization $\Phi \colon \mathbb{E}^{-} \to \mathfrak{A}$ and the projector $\Pi \colon \mathbb{E} \to \mathbb{E}^{-}$. All the considered equations have smoothing properties and the inertial manifold consists of complete trajectories, which are the most smooth ones and, in particular, they are classical solutions. From this for any amenable trajectory $v(t)$ and its projection $\zeta(t) := \Pi v(t)$ (note that $v(t)=\Phi(\zeta(t))$) we get the inertial form
\begin{equation}
	\label{EQ: InertialForm}
	\dot{\zeta}(t) = \Pi \left[ A\Phi(\zeta(t)) + BF(C \zeta(t)) \right] =: f(\zeta(t)),
\end{equation}
Thus, for any trajectory on the inertial manifold there is a corresponding solution to \eqref{EQ: InertialForm}. To show that this correspondence is bijective one has to show that the vector field $f \colon \mathbb{E}^{-} \to \mathbb{E}^{-}$ is globally Lipschitz. Needles to say, in the case when $\mathfrak{A}$ is $C^{1}$-differentiable, the vector field in \eqref{EQ: ApplicationsExampleDiffEq} is tangent at points of $\mathfrak{A}$ and $f$ is nothing more than its coordinate representation in the chart given by $\Pi$. We left details to the reader.

Below, we give applications for ODEs, neutral delay equations and semilinear parabolic equations. Analogous results can be obtained for many other types of equations such as certain hyperbolic equations, parabolic equations with nonlinear boundary conditions, PDEs with delays and systems of coupled equations.

\subsection{Ordinary differential equations}
\label{SUBSEC: ODEs}

Let us consider the following class of ordinary differential equations in $\mathbb{R}^{n}$:
\begin{equation}
	\label{EQ: ExampleODE}
	\dot{v}(t)=Av(t)+BF(\vartheta^{t}(q),Cv(t))+W(\vartheta^{t}(q)),
\end{equation}
where $A$, $B$, $C$ are $n \times n$, $n \times m$ and $r \times n$ matrices respectively; $\vartheta$ is a flow on a complete metric space $\mathcal{Q}$, $W \colon \mathcal{Q} \to \mathbb{R}^{n}$ is bounded continuous function and $F \colon \mathcal{Q} \times \mathbb{R}^{r} \to \mathbb{R}^{m}$ is a nonlinear continuous function such that for some constant $\Lambda>0$ we have
\begin{equation}
		\label{EQ: ODELipschitz}
		|F(q,y_{1}) - F(q,y_{2})|_{\Xi} \leq \Lambda |y_{1}-y_{2}|_{\mathbb{M}} \text{ for all } y_{1},y_{2} \in \mathbb{R}^{r}, q \in \mathcal{Q}.
\end{equation}
Here $\Xi = \mathbb{R}^{m}$ and $\mathbb{M} = \mathbb{R}^{r}$ are endowed with some (not necessarily Euclidean) inner products.

Let $v(t;q,v_{0})$ be the solution with $v(0;q,v_{0})=v_{0}$. Due to \eqref{EQ: ODELipschitz} the solution $v(t;q,v_{0})$ exists for all $t \in \mathbb{R}$ and, moreover, it depends continuously on $(t,q,v_{0})$ (see, for example, Theorem 3.2, Chapter II in \cite{HartmanODE1992}). Thus, the family of maps $\psi^{t}(q,\cdot) \colon \mathbb{R}^{n} \to \mathbb{R}^{n}$ given by
\begin{equation}
	\psi^{t}(q,v_{0}):=v(t;q,v_{0}), \text{ where } v_{0} \in \mathbb{R}^{n}, q \in \mathcal{Q} \text{ and } t \geq 0.
\end{equation}
is a cocycle in $\mathbb{R}^{n}=\mathbb{E} = \mathbb{H}$ over the flow $(\mathcal{Q},\vartheta)$, which satisfies \textbf{(ULIP)} with $\tau_{S} = 0$. Since $\mathbb{E}$ is finite-dimensional and \textbf{(ULIP)} holds, we have have \textbf{(UCOM)} (and, in particular, \textbf{(ACOM)}) satisfied.

Suppose we are given with a quadratic form $\mathcal{F}(v,\xi)$ in $v \in \mathbb{E}$ and $\xi \in \Xi$ such that $\mathcal{F}(v,0) \geq 0$ and $\mathcal{F}(v_{1}-v_{2},\xi_{1}-\xi_{2}) \geq 0$ for all $v_{1},v_{2} \in \mathbb{E}$, $q \in \mathcal{Q}$ and $\xi_{1}=F(q,Cv_{1})$, $\xi_{2}=F(q,Cv_{2})$. By $\mathcal{F}^{\mathbb{C}}$ we denote its Hermitian extension to the complexifications $\mathbb{E}^{\mathbb{C}}$ and $\Xi^{\mathbb{C}}$, i.e. $\mathcal{F}^{\mathbb{C}}(v_{1}+iv_{2},\xi_{1}+i\xi_{2}) = \mathcal{F}(v_{1},\xi_{1}) + \mathcal{F}(v_{2},\xi_{2})$ for all $v_{1},v_{2} \in \mathbb{E}$ and $\xi_{1},\xi_{2} \in \Xi$.

Now we are going to provide conditions for \textbf{(H1)}, \textbf{(H2)}, \textbf{(H3)} to hold. For $\nu_{0} \in \mathbb{R}$ such $A + \nu_{0} I$ admits an exponential dichotomy, we denote by $\mathbb{E}^{s}(\nu_{0})$ and $\mathbb{E}^{u}(\nu_{0})$ the stable and unstable subspaces of $A + \nu_{0} I$ respectively.

We have the following theorem.
\begin{theorem}
	\label{TH: ODEsInertialManifolds}
	Suppose that for some $\nu_{0} \in \mathbb{R}$ the matrix $A$ has exactly $j \geq 0$ eigenvalues $\lambda$ with $\operatorname{Re} \lambda > - \nu_{0}$ and no eigenvalues lie on the line $-\nu_{0} + i\mathbb{R}$. Let the frequency inequality
	\begin{equation}
		\label{EQ: FreqInequalityODE}
		\sup_{\xi \in \Xi^{\mathbb{C}}}\frac{\mathcal{F}^{\mathbb{C}}(-(A - (\nu_{0} + i\omega) I)^{-1}B\xi,\xi)}{|\xi|^{2}_{\Xi^{\mathbb{C}}}} < 0 \text{ for all } \omega \in \mathbb{R}
	\end{equation}
    be satisfied. Then there exists a symmetric matrix $P$ such that the cocycle $(\psi,\vartheta)$ generated by \eqref{EQ: ExampleODE} satisfies \textbf{(H1)} with $\mathbb{E}^{-} = \mathbb{E}^{u}(\nu_{0})$ and $\mathbb{E}^{+} = \mathbb{E}^{s}(\nu_{0})$, \textbf{(H2)} with the given $j$ and \textbf{(H3)} with $\nu_{0}(\cdot) \equiv \nu_{0}$, $\tau_{P} = 0$ and some $\delta_{P}>0$.
\end{theorem}
The proof of Theorem \ref{TH: ODEsInertialManifolds} is standard and can be derived from the Frequency Theorem obtained in \cite{Anikushin2020FreqDelay,Anikushin2020FreqParab}.

Let us mention here a general choice of the quadratic form derived from the Lipschitz property \eqref{EQ: ODELipschitz} as $\mathcal{F}(v,\xi) = \Lambda^{2}|Cv|^{2}_{\mathbb{M}} - |\xi|^{2}_{\Xi}$. Then \eqref{EQ: FreqInequalityODE} can be described in terms of the transfer matrix $W(p) = C(A-pI)^{-1}B$ as
\begin{equation}
	|W(-\nu_{0}+i\omega)|_{\Xi^{\mathbb{C}} \to \mathbb{M}^{\mathbb{C}}} < \Lambda^{-1} \text{ for all } \omega \in \mathbb{R}.
\end{equation}
This inequality was used by R.A.~Smith \cite{Smith1986Massera,Smith1987OrbStab}. For more specific cases, more delicate choices of quadratic forms, leading to the so-called Circle Criterion and its variations, are possible (see \cite{Gelig1978,Burkin2014Hidden,Burkin2015}).

If $F$ has a continuous derivative $F'_{y}$ in $y$, then the linearization semicocycle $\Xi$ can be derived from the linearized equation
\begin{equation}
	\label{EQ: ODElinearizedEquation}
	\dot{\xi}(t)=A\xi(t) + BF'_{y}(\vartheta^{t}(q),C\psi^{t}(q,v_{0}))C\xi(t)
\end{equation}
as $\Xi^{t}(q,v_{0},\xi_{0}):=\xi(t;q,v_{0},\xi_{0})$, where $\xi(t;q,v_{0},\xi_{0})$ is the solution to \eqref{EQ: ODElinearizedEquation} with $\xi(0;q,v_{0},\xi_{0})=\xi_{0}$. Standard arguments show that \textbf{(DIFF)} is satisfied (see Theorem 3.1, Chapter V in \cite{HartmanODE1992}).

\subsection{Neutral delay equations in $\mathbb{R}^{n}$: frequency-domain inequalities and equations with small delays}
\label{SUBSEC: DelayEquations}

In this subsection we consider the following class of neutral delay differential equations in $\mathbb{R}^{n}$:
\begin{equation}
	\label{EQ: ClassicalDelayEquation}
	\frac{d}{dt}[x(t)+D_{0}x_{t}] = \widetilde{A}x_{t} + \widetilde{B}F(\vartheta^{t}(q),\widetilde{C}x_{t}) + \widetilde{W}(\vartheta^{t}(q)),
\end{equation}
where $x_{t}(\theta) := x(t+\theta)$ for $\theta \in [-\tau,0]$ denotes the history segment; $\tau>0$ is a constant; $D_{0}, \widetilde{A} \colon C([-\tau,0];\mathbb{R}^{n}) \to \mathbb{R}^{n}$, $\widetilde{B} \colon \mathbb{R}^{m} \to \mathbb{R}^{n}$ and $\widetilde{C} \colon C([-\tau,0];\mathbb{R}^{n}) \to \mathbb{R}^{r}$ are bounded linear operators and $D_{0}$ is nonatomic at zero; $\vartheta$ is a flow on a complete metric space $\mathcal{Q}$, $\widetilde{W} \colon \mathcal{Q} \to \mathbb{R}^{n}$ is a bounded continuous function and $F \colon \mathcal{Q} \times \mathbb{R}^{r} \to \mathbb{R}^{m}$ is a nonlinear continuous function such that for some constant $\Lambda > 0$ we have
\begin{equation}
	\label{EQ: DelayLipschitz}
	|F(q,y_{1}) - F(q,y_{2})|_{\Xi} \leq \Lambda |y_{1}-y_{2}|_{\mathbb{M}} \text{ for all } y_{1},y_{2} \in \mathbb{R}^{r}, q \in \mathcal{Q}.
\end{equation}
Here $\Xi =\mathbb{R}^{m}$ and $\mathbb{M} =\mathbb{R}^{r}$ are endowed with some (not necessarily Euclidean) inner products.

Let us discuss the well-posedness of \eqref{EQ: ClassicalDelayEquation} as an evolutionary equation in a proper Hilbert space. For this consider the Hilbert space $\mathbb{H} := \mathbb{R}^{n} \times L_{2}(-\tau,0;\mathbb{R}^{n})$ with the usual norm, which we denote by $|\cdot|_{\mathbb{H}}$, and the operator $A \colon \mathcal{D}(A) \subset \mathbb{H} \to \mathbb{H}$ given by
\begin{equation}
	\label{EQ: OperatorADefinition}
	(y, \phi)
	\overset{A}{\mapsto}
	\left(\widetilde{A}\phi, \frac{d}{d \theta} \phi\right),
\end{equation}
where $(y,\phi) \in \mathcal{D}(A) := \{ (x,\phi) \in \mathbb{H} \ | \ \phi(0)+D_{0}\phi=y, \phi \in W^{1,2}(-\tau,0;\mathbb{R}^{n})  \}$. It can be shown that $A$ is a generator of a $C_{0}$-semigroup in $\mathbb{H}$, which we denote by $G(t)$. Let the bounded linear operator $B \colon \Xi \to \mathbb{H}$ be defined as $B \xi := (\widetilde{B}\xi,0)$ for $\Xi \in \Xi$ and let the unbounded linear operator $C \colon \mathbb{H} \to \mathbb{R}^{r}$ be defined as $C(x,\phi):=\widetilde{C}\phi$ for $\phi \in C([-\tau,0];\mathbb{R}^{n})$. Put also $W(q):=(\widetilde{W}(q),0)$. Now \eqref{EQ: ClassicalDelayEquation} can be written as an abstract evolution equation in $\mathbb{H}$ as
\begin{equation}
	\label{EQ: AbstractDelayHilberSpace}
	\dot{v}(t) = Av(t) + BF(\vartheta^{t}(q),Cv(t)) + W(\vartheta^{t}(q)).
\end{equation}
It can be shown that its solutions generate a cocycle in $\mathbb{H}$, satisfying \textbf{(ULIP)}, as $\psi^{t}(q,v_{0}):=v(t;q,v_{0})$, where $v(t;q,v_{0})$ is a solution to \eqref{EQ: AbstractDelayHilberSpace} such that $v(0;q,v_{0}) = v_{0}$. This solution is determined from the variation of constants formula
\begin{equation}
	v(t)=G(t)v_{0} + \int_{0}^{t}G(t-s) \left[ BF(\vartheta^{s}(q),Cv(s)) + W(\vartheta^{s}(q)) \right]ds.
\end{equation}
However, it is should be emphasized in that sense $Cv(s)$ is understood since $v(s)$ may not correspond to a continuous function. We refer to \cite{AnikushinNDE2022,Anikushin2020FreqDelay} for this. 

If we embed $\mathbb{E} = C([-\tau,0];\mathbb{R}^{n})$ into $\mathbb{H}$ as $\phi \mapsto (\phi(0)+D_{0}\phi,\phi)$, then the cocycle obtained through the classical solutions (see, for example, J.K.~Hale \cite{Hale1977}) on $\mathbb{E}$ agrees with $\psi$. Moreover, when $D_{0}$ can be extended to a bounded operator in $L_{2}$, we in fact have $\psi^{\tau}(q,\mathbb{H}) \subset \mathbb{E}$ and \textbf{(ULIP)} satisfied with $\tau_{S} = \tau$. Since this seems to fail in the general case, we will study $\psi$ in $\mathbb{H}$.

When $D_{0} \equiv 0$, the cocycle satisfies \textbf{(UCOM)} with $\tau_{ucom} = 2 \tau$. However, in the general case this is not the case, but one can obtain asymptotic compactness, i.e. \textbf{(ACOM)}, under additional assumptions. Let us consider the value
\begin{equation}
	\label{EQ: NDEdecayexponent}
	a_{D} := \sup\{ \operatorname{Re}p \ | \ \operatorname{det}\left[ I + \delta(p) \right] = 0 \}.
\end{equation}
In the case the set under $\operatorname{sup}$ is empty (for example, when $D_{0} \equiv 0$), we put $a_{D}:=-\infty$. It can be shown that $\psi$ satisfies \textbf{(ACOM)} with any $\gamma^{+}>a_{D}$ in both spaces $\mathbb{E}$ and $\mathbb{H}$.

By the Riesz representation theorem, there are functions $a(\cdot), \delta_{0}(\cdot) \colon [-\tau,0] \to \mathbb{R}^{n \times n}$ and $c(\cdot) \colon [-\tau,0] \to \mathbb{R}^{r \times n}$ of bounded variation such that for any $\phi \in C([-\tau,0];\mathbb{R}^{n})$ we have
\begin{equation}
	\widetilde{A}\phi = \int_{-\tau}^{0} d a(\theta)\phi(\theta), \ D_{0}\phi = \int_{-\tau}^{0}d\delta_{0}(\theta)\phi(\theta) \text{ and } \widetilde{C}\phi = \int_{-\tau}^{0} d c(\theta) \phi(\theta).
\end{equation}
Let us also consider the matrix-valued functions of $p \in \mathbb{C}$ given by
\begin{equation}
	\label{EQ: NDE AlphaDeltaGamma}
	\alpha(p) = \int_{-\tau}^{0}e^{p\theta}d a(\theta), \ \delta(p) = \int_{-\tau}^{0}e^{p\theta}d\delta_{0}(\theta) \text{ and } \gamma(p) = \int_{-\tau}^{0}e^{p\theta} dc(\theta).
\end{equation}

Now we are going to discuss verification of \textbf{(H1)}, \textbf{(H2)} and \textbf{(H3)}, following our work \cite{Anikushin2020FreqDelay}.

Firstly, it can be shown that the spectrum of $A$ as an operator in $\mathbb{E}$ or $\mathbb{H}$ is discrete and determined by the roots $p \in \mathbb{C}$ of
\begin{equation}
	\label{EQ: SpectrumDelayODE}
	\operatorname{det} \left[ pI + p\delta(p) - \alpha(p) \right] = 0.
\end{equation}
It is known that for any $-\nu_{0} > a_{D}$ there exists only finite number of roots with $\operatorname{Re}p \geq -\nu_{0}$ and we have proper exponential dichotomies for the semigroup $G(t)$. Namely, in the case there is no roots with $\operatorname{Re}p=-\nu_{0}$, there are subspaces $\mathbb{E}^{s}(\nu_{0})$ and $\mathbb{E}^{u}(\nu_{0})$, where $\mathbb{E}^{u}(\nu_{0})$ is the $j$-dimensional spectral subspace corresponding to the roots with $\operatorname{Re}p > -\nu_{0}$ and $\mathbb{E}^{s}(\nu_{0})$ is a complementary subspace with proper growth rates determined by the spectral bound (see Theorems 10.1 and 10.3 of Chapter 12 in \cite{Hale1977}). Moreover, we have analogous decomposition in $\mathbb{H}$ with the spaces $\mathbb{H}^{s}(\nu_{0})$ and $\mathbb{H}^{u}(\nu_{0})$ (see \cite{AnikushinNDE2022}).

Let $W(p)=C(A-pI)^{-1}B$ be the transfer operator of the triple $(A,B,C)$. In terms of the functions introduced in \eqref{EQ: NDE AlphaDeltaGamma}, it can be described as $W(p) = \gamma(p) (\alpha(p) - p I -p\delta(p))^{-1}\widetilde{B}$.

Let $\mathcal{G}(y,\xi)$ be a quadratic form of $y \in \mathbb{M}$ and $\xi \in \Xi$. Put $\mathcal{F}(v,\xi):= \mathcal{G}(Cv,\xi)$ and let $\mathcal{G}$ be such that $\mathcal{F}(v,0) \geq 0$ for all $v \in \mathbb{E}$ and $\mathcal{F}(Cv_{1}-Cv_{2},F(q,Cv_{1})-F(q,Cv_{2})) \geq 0$ for all $v_{1},v_{2} \in \mathbb{E}$ and $q \in \mathcal{Q}$. We have the following theorem.
\begin{theorem}
	\label{TH: DelayODEFreqCond}
	Let $\mathcal{F}$ be as above and let $D_{0}$ have no singular part. Suppose that for some $\nu_{0} < -a_{D}$, where $a_{D}$ is given by \eqref{EQ: NDEdecayexponent}, the characteristic equation \eqref{EQ: SpectrumDelayODE} does not have roots with $\operatorname{Re}p = -\nu_{0}$ and have exactly $j$ roots with $\operatorname{Re}p > -\nu_{0}$. Let the frequency inequality 
	\begin{equation}
		\label{EQ: Frequency-domainDelay}
		\sup\limits_{\xi \in \Xi}\frac{\mathcal{G}^{\mathbb{C}}(-W(i\omega - \nu_{0})\xi,\xi)}{|\xi|^{2}_{\Xi}} < 0 \text{ for all } \omega \in \mathbb{R}
	\end{equation}
	be satisfied. Then there exist a self-adjoint operator $P \in \mathcal{L}(\mathbb{H})$ such that the cocycle $(\psi,\vartheta)$ generated by \eqref{EQ: AbstractDelayHilberSpace} in $\mathbb{H}$ satisfies \textbf{(H1)} with $\mathbb{E}^{-}(q) \equiv \mathbb{E}^{-} = \mathbb{H}^{u}(\nu_{0})$ and $\mathbb{E}^{+}(q) \equiv \mathbb{E}^{+} = \mathbb{H}^{s}(\nu_{0})$, \textbf{(H2)} with the given $j$ and \textbf{(H3)} with $\nu_{0}(\cdot) \equiv \nu_{0}$, $\tau_{P} = 0$ and some $\delta_{P}>0$.
\end{theorem}
For a proof see \cite{Anikushin2020FreqDelay}.

Again in the case of the Lipschitz property \eqref{EQ: DelayLipschitz} a general choice of the quadratic form is $\mathcal{F}(v,\xi) = \Lambda^{2}|Cv|^{2}_{\mathbb{M}} - |\xi|^{2}_{\Xi}$. Then \eqref{EQ: Frequency-domainDelay} reads as
\begin{equation}
	|W(- \nu_{0} + i\omega )|_{\Xi^{\mathbb{C}} \to \mathbb{M}^{\mathbb{C}}} < \Lambda^{-1} \text{ for all } \omega \in \mathbb{R}.
\end{equation}
This extends the frequency inequality used by R.A.~Smith \cite{Smith1990ConvDelay,Smith1992} in the case $D_{0} \equiv 0$. Below we will discuss its relation to the works of Yu.A.~Ryabov \cite{Ryabov1967}, R.D.~Driver \cite{Driver1968SmallDelays}, C.~Chicone \cite{Chicone2003} and S.~Chen and J.~Shen \cite{ChenShen2021} in the case of equations with small delays.

Analogously, in the case $F$ has a continuous derivative $F'_{y}$ in $y$, the linearization semicocycle $\Xi$ can be obtained through the linearized equation 
\begin{equation}
	\label{EQ: NDELinearized}
	\dot{x}(t) =A\xi(t) + BF'_{y}(\vartheta^{t}(q),C\psi^{t}(q,v_{0}))C\xi(t)
\end{equation}
as $\Xi^{t}(q,v_{0},\xi_{0}):=\xi(t;q,v_{0},\xi_{0})$, where $\xi(t;q,v_{0},\xi_{0})$ is the solution to \eqref{EQ: NDELinearized} with $\xi(0;q,v_{0},\xi_{0})=\xi_{0}$. It can be shown that \textbf{(DIFF)} is satisfied (see \cite{AnikushinNDE2022, Anikushin2020Semigroups, Hale1977}).

\subsubsection{Equations with small delays}

Let us consider \eqref{EQ: ClassicalDelayEquation} with $\widetilde{A} \equiv 0$ and $\widetilde{B}$ being the identity $n \times n$-matrix, i.~e. the delay equation
\begin{equation}
	\label{EQ: ConstantDelayExample}
	\frac{d}{dt}\left( x(t) + D_{0}x_{t} \right)=F(t,\widetilde{C}x_{t}),
\end{equation}
where $F$ still satisfies \eqref{EQ: DelayLipschitz} for some $\Lambda>0$. Let us assume that $\Xi = \mathbb{R}^{m}$ and $\mathbb{M}=\mathbb{R}^{r}$ are endowed with the Euclidean scalar products. We assume that $\widetilde{C} \colon C([-\tau,0]; \mathbb{R}^{n}) \to \mathbb{R}^{r}$ has the form
\begin{equation}
	\label{EQ: SmallDelaysOpC}
	\widetilde{C}\phi = \left( \int_{-\tau}^{0}\phi_{j_{1}}(\theta)d\mu_{1}(\theta),\ldots,  \int_{-\tau}^{0}\phi_{j_{r}}(\theta)d\mu_{r}(\theta) \right),
\end{equation}
where $\mu_{k}$ are Borel signed measures on $[-\tau,0]$ having total variation $\leq 1$ and $1 \leq j_{1} \leq j_{2} \leq \ldots \leq j_{r} \leq n$ are integers. In particular, taking $\mu_{k} = \delta_{-\tau_{k}}$ with some $\tau_{k} \in [0,\tau]$, we may get discrete delays since $\int_{-\tau}^{0}\phi_{j_{k}}(\theta)d\mu_{k}(\theta) = \phi_{j_{k}}(-\tau_{k})$.

Considering \eqref{EQ: ConstantDelayExample} in the context of Theorem \ref{TH: DelayODEFreqCond}, gives the following.
\begin{theorem}
	\label{TH: SmallDelayNdimTheorem}
	For \eqref{EQ: ConstantDelayExample} suppose that \eqref{EQ: SmallDelaysOpC} is satisfied and the Lipschitz constant $\Lambda$ from \eqref{EQ: DelayLipschitz} is calculated w.~r.~t. the Euclidean norms. Suppose that $D_{0}$ has no singular part and for some $\nu_{0} > 0$ we have
	\begin{equation}
		\label{EQ: SmallDelaysSqueezingKappa}
		\varkappa(\nu_{0}) := e^{ \tau \nu_{0}} \cdot \|D_{0}\| < 1.
	\end{equation}
	Let the inequality
	\begin{equation}
		\label{EQ: SmallDelaysThMainIneq}
		\sqrt{r} \cdot \frac{e^{\tau \nu_{0}}}{\nu_{0}} \cdot \frac{1}{1 - \varkappa(\nu_{0})} < \Lambda^{-1}.
	\end{equation}
	be satisfied. Then Theorem \ref{TH: DelayODEFreqCond} is applicable for \eqref{EQ: ConstantDelayExample} with $j=n$.
\end{theorem}
The proof is just simple computation (see \cite{Anikushin2020FreqDelay}).

In particular, when $D_{0} \equiv 0$ and $\nu = \tau^{-1}$ the inequality in \eqref{EQ: SmallDelaysThMainIneq} takes the form
\begin{equation}
	\label{EQ: SmallDelaysDzeroIneq}
	\sqrt{r} \cdot e \cdot \tau < \Lambda^{-1}.
\end{equation} 

In \cite{Chicone2003} C.~Chicone extended works of Yu.A.~Ryabov \cite{Ryabov1967} and R.D.~Driver \cite{Driver1968SmallDelays} on the existence of $n$-dimensional inertial manifolds for delay equations in $\mathbb{R}^{n}$ with small delays. Let us compare their results with Theorem \ref{TH: SmallDelayNdimTheorem}. First of all, the class of equations in \cite{Chicone2003} is greater than \eqref{EQ: ConstantDelayExample}, but our class covers many equations arising in practice. Secondly, in \cite{Chicone2003} it is used a Lipschitz constant $K$ of the nonlinearity as a map from $C([-\tau,0];\mathbb{R}^{n}) \to \mathbb{R}^{n}$. Thus, in our case $K = \sqrt{r} \Lambda$ since for $\phi_{1}$,$\phi_{2} \in C([-\tau,0];\mathbb{R}^{n})$ we have
\begin{equation}
	|F(q,\widetilde{C}\phi_{1})-F(q,\widetilde{C}\phi_{2})|_{\Xi} \leq \Lambda \cdot \|\widetilde{C}\| \cdot \| \phi_{1}-\phi_{2}\|_{\infty} \leq \Lambda \cdot \sqrt{r} \cdot \| \phi_{1}-\phi_{2}\|_{\infty}
\end{equation}
and one can see that the used estimates are sharp for general $\widetilde{C}$.

A theorem of Ryabov and Driver (see item 1) of Theorem 2.1 from \cite{Chicone2003}; see also \cite{Ryabov1967,Driver1968SmallDelays}) guarantees the existence of an inertial manifold if $K\tau e < 1$ or, equivalently, $\sqrt{r} e \tau < \Lambda^{-1}$. This is exactly our condition \eqref{EQ: SmallDelaysDzeroIneq}. A small strengthening of the inequality $\sqrt{r} e \tau < \Lambda^{-1}$, which as $\tau \to 0$ remains asymptotically the same, guarantees the exponential tracking (see item 2) of Theorem 2.1 from \cite{Chicone2003}), but do not provide any estimate for the exponent. Moreover, Theorem 2.2 from \cite{Chicone2003} guarantees $C^{1}$-differentiability of the inertial manifold under $F \in C^{1}$ and the rougher condition $2 \sqrt{r} \sqrt{e} \tau < \Lambda^{-1}$. Our geometric approach shows that, in fact, these properties (at least in the case of \eqref{EQ: ConstantDelayExample}) are satisfied without further strengthening of \eqref{EQ: SmallDelaysDzeroIneq}. This solves the question posed by R.D. Driver in \cite{Driver1968SmallDelays} (see Remark on p. 338 therein) in the considered class of equations. Moreover, our approach gives the exact description for the exponent of attraction $\nu = \tau^{-1}$ and allows to construct stable foliations along the manifold.

\subsection{Semilinear parabolic equations: frequency-domain inequalities, Spectral Gap Condition and the Spatial Averaging Principle}
\label{SUBSEC: SemilinearParabolicEquations}

Let $A \colon \mathcal{D}(A) \subset \mathbb{H}_{0} \to \mathbb{H}_{0}$ be a sectorial operator a real Hilbert space $\mathbb{H}_{0}$ having a compact resolvent and such that $\operatorname{Re}\operatorname{Spec}(A) > -\lambda$ for some $\lambda \in \mathbb{R}$. Then one can define the powers of $A+\lambda I$ and a scale of Hilbert spaces $\mathbb{H}_{\alpha}:=\mathcal{D}((A+\lambda I)^{\alpha})$ with proper inner products $( \cdot,\cdot)_{\alpha}$ (see Section 1.4 in D.~Henry \cite{Henry1981}). Let $\alpha \in [0,1)$ be fixed and let $\mathbb{M},\Xi$ be some Hilbert spaces. Suppose that linear bounded operators $C \colon \mathbb{H}_{\alpha} \to \mathbb{M}$ and $B \colon \Xi \to \mathbb{H}$ are given. Let $(\mathcal{Q},\vartheta)$ be a flow on a complete metric space $\mathcal{Q}$. Consider a bounded continuous function $W \colon \mathcal{Q} \to \mathbb{H}$ and let $F \colon \mathcal{Q} \times \mathbb{M} \to \Xi$ be a continuous nonlinear map such that for some constant $\Lambda>0$ we have
\begin{equation}
	\label{EQ: SPELipschitzProperty}
	|F(q,y_{1})-F(q,y_{2})|_{\Xi} \leq \Lambda |y_{1}-y_{2}|_{\mathbb{M}} \text{ for all } y_{1},y_{2} \in \mathbb{M}, q \in \mathcal{Q}.
\end{equation}
We consider the nonlinear evolution equation
\begin{equation}
	\label{EQ: ParabolicNonAutEquation}
	\dot{v}(t) = -Av(t) + BF(\vartheta^{t}(q),Cv(t)) + W(\vartheta^{t}(q)).
\end{equation}
Following the fixed point arguments of Theorem 4.2.3 as in the monograph of I.~Chueshov \cite{Chueshov2015}, it can be shown that for any $v_{0} \in \mathbb{H}_{\alpha}$ and $q \in \mathcal{Q}$ there exists a unique mild solution $v(t)=v(t;q,v_{0})$ to \eqref{EQ: ParabolicNonAutEquation} defined for $t \geq 0$ and such that $v(0)=v_{0}$. Let $G(t)$, $t \geq 0$, denote the $C_{0}$-semigroup generated by $-A$. Then this solution satisfies the variation of constants formula as
\begin{equation}
	\label{EQ: ParabVarOfConstantsAsem}
	v(t)=G(t)v_{0}+\int_{0}^{t}G(t-s)\left[BF(\vartheta^{s}(q),Cv(s))+W(\vartheta^{s}(q))\right]ds.
\end{equation}
From this it can be shown that the family of maps given by
\begin{equation}
	\psi^{t}(q,v_{0}):=v(t;q,v_{0}), \text{ where } v_{0} \in \mathbb{H}_{\alpha}, q \in \mathbb{R} \text{ and } t \geq 0
\end{equation}
is a cocycle in $\mathbb{E}=\mathbb{H}:=\mathbb{H}_{\alpha}$ over $(\mathcal{Q},\vartheta)$. From \eqref{EQ: SPELipschitzProperty} we get \textbf{(ULIP)} satisfied and due to the smoothing property of \eqref{EQ: ParabolicNonAutEquation} (see \cite{Chueshov2015,Henry1981}) we also have \textbf{(UCOM)} with any $\tau_{ucom} > 0$.

Note that $G(t)$ is a $C_{0}$-semigroup in any space from the scale and the spectra of its generators coincide. Suppose that for some $\nu_{0} \in \mathbb{R}$ there exist exactly $j$ eigenvalues of $-A$ with $\operatorname{Re}p > -\nu_{0}$ and no eigenvalues satisfy $\operatorname{Re}p = -\nu_{0}$. Then there exists a decomposition $\mathbb{H}_{\alpha} = \mathbb{H}^{s}_{\alpha}(\nu_{0}) \oplus \mathbb{H}^{u}(\nu_{0})$, where $\mathbb{H}^{u}_{\alpha}(\nu_{0})$ is the spectral subspace corresponding to the roots with $\operatorname{Re}p > -\nu_{0}$ and $\mathbb{H}^{s}_{\alpha}(\nu_{0})$ is a complementary subspace.

Consider a quadratic form $\mathcal{F}(v,\xi)$ of $v \in \mathbb{H}_{\alpha}$ and $\xi \in \Xi$ such that $\mathcal{F}(v,0) \geq 0$ and $\mathcal{F}(v_{1}-v_{2},F(q,Cv_{1})-F(q,Cv_{2})) \geq 0$ for all $v,v_{1},v_{2} \in \mathbb{H}_{\alpha}$, $q \in \mathcal{Q}$ and $\xi \in \Xi$.
\begin{theorem}
	\label{TH: ParabEqsMainAssumptionsVer}
    Suppose there is $\nu_{0} \in \mathbb{R}$ such that the operator $-A$ does not have eigenvalues with $\operatorname{Re}p = -\nu_{0}$ and there are exactly $j$ eigenvalues with $\operatorname{Re}p > -\nu_{0}$. Let the frequency inequality
	\begin{equation}
		\label{EQ: FreqParabGeneral}
		\sup_{\xi \in \Xi^{\mathbb{C}}}\frac{\mathcal{F}^{\mathbb{C}}(-(-A - (\nu_{0} + i\omega) I)^{-1}B\xi,\xi)}{|\xi|^{2}_{\Xi^{\mathbb{C}}}} < 0 \text{ for all } \omega \in \mathbb{R}
	\end{equation}
    be satisfied. Then there exists a self-adjoint operator $P \in \mathcal{L}(\mathbb{H}_{\alpha})$ such that the cocycle $(\psi,\vartheta)$ generated by \eqref{EQ: ParabolicNonAutEquation} in $\mathbb{H}_{\alpha}$ satisfies \textbf{(H1)} with $\mathbb{E}^{+}(q) \equiv \mathbb{E}^{+} := \mathbb{H}^{s}_{\alpha}(\nu_{0})$ and $\mathbb{E}^{-}(q) \equiv \mathbb{E}^{-}:=\mathbb{H}^{u}_{\alpha}(\nu_{0})$; \textbf{(H2)} with the given $j$ and \textbf{(H3)} with $\nu_{0}(\cdot) \equiv \nu_{0}$, $\tau_{P} =0$ and some $\delta_{P} > 0$.
\end{theorem}

As in the previous cases, under the Lipschitz property \eqref{EQ: SPELipschitzProperty} a general choice of the quadratic form is given by $F(v,xi):=\Lambda^{2} |v|^{2}_{\mathbb{M}} - |\xi|^{2}_{\Xi}$. In this case the frequency inequality \eqref{EQ: FreqParabGeneral} can be described in terms of the transfer operator $W(p):=C(-A-pI)^{-1}B$ as
\begin{equation}
	\label{EQ: SPESmithFreqIneq}
	|W(-\nu_{0}+i\omega)|_{\Xi^{\mathbb{C}} \to \mathbb{M}^{\mathbb{C}}} < \Lambda^{-1} \text{ for all } \omega \in \mathbb{R}.
\end{equation}
This also extends the frequency inequality used by R.A.~Smith \cite{Smith1994PB1,Smith1994PB2}.

In the case $F$ has a continuous derivative $F'_{y}$ in $y$, the linearization cocycle $\Xi$ can be derived from the linearized equation 
\begin{equation}
	\label{EQ: SPELinearized}
	\dot{\xi}(t)=-A\xi(t) + BF'_{y}(\vartheta^{t}(q),\psi^{t}(q,v_{0}))C\xi(t)
\end{equation}
as $\Xi^{t}(q,v_{0},\xi_{0}):=\xi(t;q,v_{0},\xi_{0})$, where $\xi(t;q,v_{0},\xi_{0})$ is the solution to \eqref{EQ: SPELinearized} with $\xi(0;q,v_{0},\xi_{0})=\xi_{0}$. It can be shown that \textbf{(DIFF)} is satisfied (see, for example, Theorem 3.4.4 in \cite{Henry1981}).

Now we are going to discuss some particular cases of the frequency inequality in \eqref{EQ: FreqParabGeneral} and its modifications.

\subsubsection{Spectral Gap Condition}

Let us consider the special case of \eqref{EQ: SPESmithFreqIneq} for $A$ being a self-adjoint positive operator; $\Xi := \mathbb{H}_{\beta}$ for some $\beta \in [0,\alpha]$; $\mathbb{M}:=\mathbb{H}_{\alpha}$ and $C$, $B$ being the identity operators. Let $0 < \lambda_{1} \leq \lambda_{2} \leq \ldots$ be the eigenvalues of $A$. Let us fix $j$ such that $\lambda_{j+1}-\lambda_{j}>0$ and find $\nu_{0} \in (\lambda_{j},\lambda_{j+1})$ such that \eqref{EQ: SPESmithFreqIneq} is satisfied, that is
\begin{equation}
	\label{EQ: FrequencyConditionParabBefore}
	\| (A+(-\nu_{0} + i \omega)I)^{-1} \|_{\mathbb{H}_{\beta} \to \mathbb{H}_{\alpha}} < \Lambda^{-1}, \text{ for all } \omega \in \mathbb{R}.
\end{equation}
Using the orthonormal basis corresponding to the eigenvalues, one can easily show that
\begin{equation}
	\| (A - (\nu_{0} + i \omega)I)^{-1} \|_{\mathbb{H}_{\beta} \to \mathbb{H}_{\alpha}} = \sup_{k} \frac{\lambda^{\alpha-\beta}_{k} }{|\lambda_{k} - \nu - i\omega|} \leq \sup_{k}\frac{\lambda^{\alpha-\beta}_{k} }{|\lambda_{k} - \nu|}.
\end{equation}
Moreover, due to some monotonicity, we have
\begin{equation}
	\label{EQ: ParabSpectralNormMax}
	\sup_{k}\frac{\lambda^{\alpha-\beta}_{k}}{ |\lambda_{k} - \nu| } = \max\left\{ \frac{\lambda^{\alpha-\beta}_{j}}{\nu - \lambda_{j}}, \frac{\lambda^{\alpha-\beta}_{j+1}}{\lambda_{j+1} - \nu} \right\}
\end{equation}
and the norm will be the smallest possible if the values under the maximum in \eqref{EQ: ParabSpectralNormMax} coincide. So, for this one should take
\begin{equation}
	\nu_{0} = \frac{\lambda^{\alpha-\beta}_{j+1}}{\lambda^{\alpha-\beta}_{j}+\lambda^{\alpha-\beta}_{j+1}} \cdot \lambda_{j} + \frac{\lambda^{\alpha-\beta}_{j}}{\lambda^{\alpha-\beta}_{j}+\lambda^{\alpha-\beta}_{j+1}} \cdot \lambda_{j+1} \in ( \lambda_{j},\lambda_{j+1} ).
\end{equation}
For such a choice of $\nu_{0}$ the frequency condition \eqref{EQ: FrequencyConditionParabBefore} takes the form
\begin{equation}
	\label{EQ: SpectralGapConditionParab}
	\frac{ \lambda_{j+1} - \lambda_{j} }{\lambda^{\alpha-\beta}_{j} + \lambda^{\alpha-\beta}_{j+1} } > \Lambda,
\end{equation}
known as the Spectral Gap Condition. Its first nonoptimal version appeared in the pioneering paper of C.~Foias, G.R.~Sell and R.~Temam \cite{FoiasSellTemam1988}. The optimal version of the Spectral Gap Condition was firstly discovered by M.~Miklav\v{c}i\v{c} \cite{Miclavcic1991} and A.V.~Romanov \cite{Romanov1994}.

In fact, condition \eqref{EQ: SpectralGapConditionParab} is too rough to obtain concrete results for particular equations. For $-A$ being the Laplace operator in a domain with Dirichlet boundary conditions, the Weyl law guarantees that \eqref{EQ: SpectralGapConditionParab} is satisfied for a sufficiently large $j$ provided that $\alpha = \beta = 0$ and the domain is one-dimensional. For $\alpha > 0$ and/or $-A$ being a vector-valued Laplacian, direct applications of the Spectral Gap Condition do not guarantee the existence of inertial manifolds since it is known to be violated or it is unknown whether it holds. Although, there indeed may be obstacles for their existence (see J.~Mallet-Paret, G.R.~Sell, Z.~Shao \cite{MalletParetSellShao1993}; A.~Eden, S.V.~Zelik, V.K.~Kalantarov \cite{EdenZelikKalantarov2013}; A.V.~Romanov \cite{Romanov2016}), violation of the Spectral Gap Condition or, more generally, the frequency inequality \eqref{EQ: FreqParabGeneral} does not necessarily prevent inertial manifolds to exist for particular subclasses of equations.

In \cite{KostiankoZelik20171dI,KostiankoZelik20181dII} A.~Kostianko and S.~Zelik studied the existence of inertial manifolds for reaction-diffusion-advection systems. For such equations we have $\alpha = 1/2$ and $-A$ is a vector-valued Laplacian in $1D$-domain with Dirichlet, Neumann or periodic boundary conditions. It is shown that in the case of Dirichlet or Neumann boundary conditions a proper nonlocal change of variables transforms the system into a system, for which a modified Spectral Gap Condition (which can be derived from \eqref{EQ: FreqParabGeneral}) is satisfied. For periodic boundary conditions, the existence of inertial manifolds is guaranteed only in the scalar case and for systems they may not exist in general \cite{KostiankoZelik20181dII}.

Moreover, in applications, a given system can be written in the form \eqref{EQ: ParabolicNonAutEquation} by different ways, resulting in distinct frequency inequalities. In the case of \eqref{EQ: SPESmithFreqIneq} having variying operators $A$, $B$ and $C$ give a lot of flexibility. This allowed R.A.~Smith to obtain nontrivial results concerned with the existence of two-dimensional inertial manifolds for the FitzHugh-Nagumo system and the diffusive Goodwin system \cite{Smith1994PB1,Smith1994PB2}.

\subsubsection{Spatial Averaging Principle}
\label{EQ: SPEspatialaveraginge}

For scalar equations in $2D$ and $3D$ domains and $-A$ being the Laplace operator, the Spectral Gap Condition is violated in general. For such problems J.~Mallet-Paret and G.R.~Sell \cite{MalletParetSell1988IM} suggested a method to construct inertial manifolds known as the Spatial Averaging Principle. As it is shown by A.~Kostianko and S.~Zelik \cite{KostiankoZelik2015}, A.~Kostianko et al. \cite{KostiankoZelikSA2020}, assumptions of the Spatial Averaging Principle also lead to a squeezing condition w.~r.~t. a certain quadratic functional. We will show that it leads to the existence of nonuniformly normally hyperbolic inertial manifolds. Moreover, our results allow to construct also a $C^{1}$-differentiable vertical foliation and show that these manifolds are stable under perturbations of the vector field.

Let us consider here the autonomous semilinear parabolic equations of the form:
\begin{equation}
	\label{EQ: ParabSpatialAvEquation}
	\partial_{t}A^{-2\gamma}v + Au + F(u) = 0.
\end{equation}
Here $A \colon \mathbb{H}_{0} \to \mathbb{H}_{0}$ is a positive self-adjoint operator such that $A^{-1} \colon \mathbb{H}_{0} \to \mathbb{H}_{0}$ is compact, $\gamma \geq 0$ is fixed and $F \colon \mathbb{H}_{0} \to \mathbb{H}_{0}$ is a bounded and globally Lipschitz nonlinearity.

It is shown in \cite{KostiankoZelikSA2020} (see Proposition 2.1 therein) that \eqref{EQ: ParabSpatialAvEquation} under the imposed conditions generates a dissipative semiflow $\varphi$ in $\mathbb{E} = \mathbb{H} := \mathbb{H}_{-\gamma}$, which satisfies \textbf{(ULIP)} and \textbf{(UCOM)}.

Let $0<\lambda_{1}\leq \lambda_{2} \leq \ldots$ be the eigenvalues of $A$ and $e_{1},e_{2},\ldots$ be the corresponding orthonormal basis of eigenvectors. Let $j \geq 0$ be fixed and let $\operatorname{Pr}_{j} \colon \mathbb{H}_{0} \to \mathbb{H}_{0}$ be the orthogonal projector onto $\mathbb{E}^{-}:=\operatorname{Span}\{ e_{1},\ldots,e_{j} \}$. Note that $\operatorname{Pr}_{j}$ can be extended to the orthogonal operator in $\mathbb{H}_{-\gamma}$. Let $\mathbb{E}^{+}$ be the orthogonal in $\mathbb{H}_{-\gamma}$ complement to $\mathbb{E}^{-}$. Let us consider the quadratic form in $\mathbb{H}_{-\gamma}$ given by
\begin{equation}
	\label{EQ: SpatialAveragingQForm}
	V(v):=|(I-\operatorname{Pr}_{j})v|^{2}_{-\gamma} - |\operatorname{Pr}_{j}v|^{2}_{-\gamma}.
\end{equation}
it is clear that $V(\cdot)$ is positive on $\mathbb{E}^{+}$ and negative on $\mathbb{E}^{-}$.

We assume in addition that $F \colon \mathbb{H} \to \mathbb{H}$ is Fr\'{e}chet continuously differentiable (this can be relaxed a bit, see \cite{KostiankoZelik2015,KostiankoZelikSA2020}). Then the difference $\xi(t)=v_{1}(t)-v_{2}(t)$ of any two solutions to \eqref{EQ: ParabSpatialAvEquation} satisfies
\begin{equation}
	\label{EQ: LinearSpatialAveraging}
	\partial_{t}A^{-2\gamma}\xi(t) + A\xi(t) + l(t) \xi(t) = 0,
\end{equation}
where $l(t) = \int_{0}^{1}F'(s v_{1}(t)+(1-s)v_{2}(t)) ds$. It is shown in \cite{KostiankoZelikSA2020,KostiankoZelik2015} that the Spatial Averaging Principle leads to the so-called strong cone property in a differential form , which can be stated as follows.
\begin{description}
	\item[\textbf{(SCP1)}] There are constants $0 < \alpha_{-} \leq \alpha_{+} < +\infty$ and $\delta>0$ such that for any pair of solutions $v_{1}(t), v_{2}(t)$ to \eqref{EQ: ParabSpatialAvEquation} there is a Borel measurable function $\alpha \colon \mathbb{R}_{+} \to \mathbb{R}$ such that $\alpha^{-} \leq \alpha(t) \leq \alpha^{+}$ and we have
	\begin{equation}
		\label{EQ: SCPinequality}
		\frac{d}{dt}V(\xi(t)) + 2\alpha(t)V(\xi(t)) \leq -\delta |\xi(t)|^{2}_{-\gamma}.
	\end{equation}
    satisfied for all $t > 0$ and any solution $\xi(\cdot)$ to \eqref{EQ: LinearSpatialAveraging} with $l(t) = \int_{0}^{1}F'(s v_{1}(t)+(1-s)v_{2}(t)) ds$.
\end{description}
We have the following lemma.
\begin{lemma}
	Under \textbf{(SCP1)} the semiflow $\varphi$ in $\mathbb{H}_{-\gamma}$ generated by \eqref{EQ: ParabSpatialAvEquation} satisfies $\textbf{(H3)}^{-}_{w}$ (see Subsection \ref{SUBSEC: HorizontalLeaves}) with $\alpha_{0}(\cdot) := \alpha^{+}$ and $\textbf{(H3)}^{+}_{w}$ (see Subsection \ref{SUBSEC: VerticalFoliation}) with $\beta_{0}(\cdot) := \alpha^{-}$ and the quadratic form $V(\cdot)$ from \eqref{EQ: SpatialAveragingQForm}.
\end{lemma}
\begin{proof}
	Let us consider \eqref{EQ: SCPinequality} with $\xi(t) = v_{1}(t)-v_{2}(t)$, where $v_{1}(\cdot),v_{2}(\cdot)$ are two solutions to \eqref{EQ: ParabSpatialAvEquation}.
    
    If $V(v_{1}(0)-v_{2}(0)) \leq 0$, then from \eqref{EQ: SCPinequality} we have that $V(v_{1}(s)-v_{2}(s)) \leq 0$ for all $s \geq 0$ and, consequently, for $s > 0$ we have
    \begin{equation}
    	\frac{d}{dt}V(\xi(s)) + 2\alpha_{+}V(\xi(s)) \leq -\delta |\xi(s)|^{2}_{-\gamma}.
    \end{equation}
    and we get $\textbf{(H3)}^{-}_{w}$ after integrating the inequality
    \begin{equation}
    	\frac{d}{ds}\left[e^{2\alpha_{+}s}V(v_{1}(s)-v_{2}(s))\right] \leq -\delta e^{2\alpha_{+} s}|v_{1}(s)-v_{2}(s)|^{2}_{-\gamma}.
    \end{equation}
   
     If $V(v_{1}(r)-v_{2}(r)) \geq 0$, then we must have $V(v_{1}(s)-v_{2}(s)) \geq 0$ for all $s \in [0,r]$ and analogous arguments show that $\textbf{(H3)}^{+}_{w}$ with $\beta_{0}(\cdot) := \alpha^{-}$ is satisfied.
\end{proof}
Thus, under \textbf{(SCP1)}, Theorem \ref{TH: PrincipalLeaveConstruction} gives a $C^{1}$-differentiable inertial manifold $\mathfrak{A}$ for the semiflow $\varphi$ (if there exists at least one amenable trajectory). Moreover, Theorem \ref{TH: ExponentialTrackingIM} establishes the exponential tracking property and a $C^{1}$-differentiable vertical foliation given by the positive equivalence w.~r.~t. the form $V(\cdot)$.

Let us consider the linearized along the trajectory of $v_{0} \in \mathbb{H}_{-\gamma}$ equation
\begin{equation}
	\label{EQ: ParabSAPlinearized}
	\partial_{t}A^{-2\gamma}\xi(t) + A\xi(t) + F'(\varphi^{t}(v_{0}))\xi(t)=0
\end{equation}
and define the linearization semicocycle $\Xi$ as $\Xi^{t}(v_{0},\xi_{0}):=\xi(t;v_{0},\xi_{0})$, where $\xi(t;v_{0},\xi_{0})$ is the solution to \eqref{EQ: ParabSAPlinearized} with $\xi(0;v_{0},\xi_{0}) = \xi_{0}$.

In fact, \textbf{(SCP1)} is verified under the following assumption.
\begin{description}
	\item[\textbf{(SCP2)}] There exists a Borel measurable function $\alpha_{0} \colon \mathbb{H}_{0} \to \mathbb{R}^{+}$ such that $0 < \alpha_{-} \leq \alpha_{0}(\cdot) \leq \alpha_{+} < +\infty$ for some $\alpha^{-}$ and $\alpha_{+}$. Moreover, there exists $\delta>0$ such that for any $v_{0} \in \mathbb{H}_{-\gamma}$ and any solution $\xi(t)=\xi(t;v_{0},\xi_{0})$ to \eqref{EQ: ParabSAPlinearized} we have for all $t > 0$ the inequality
	\begin{equation}
		\label{EQ: SCP2inequality}
		\frac{d}{dt}V(\xi(t)) + 2\alpha(\vartheta^{t}(v_{0}))V(\xi(t)) \leq -\delta |\xi(t)|^{2}_{-\gamma}.
	\end{equation}
	satisfied.
\end{description}
It is convenient to consider $\Xi$ and $\phi$ in the abstract context with $\mathbb{E}:=\mathbb{H}_{0}$ and $\mathbb{H}:=\mathbb{H}_{-\gamma}$, where \textbf{(ULIP)} with any $\tau_{S} > 0$ and \textbf{(UCOM)} with any $\tau_{ucom} > 0$ (see Proposition 2.3 in \cite{KostiankoZelikSA2020}). Then \textbf{(SCP2)} implies that $\Xi$ as a cocycle over the semiflow $\varphi$ in $\mathcal{Q} := \mathbb{H}_{0}$ satisfies \textbf{(H3)} with $\nu_{0}(\cdot):=\alpha_{0}(\cdot)$ and the constant quadratic form $V(\cdot)$. Moreover, it is clear that for any $v_{0} \in \mathfrak{A}$, we have
\begin{equation}
	\begin{split}
		\mathcal{T}_{v_{0}}\mathfrak{A} = \left\{ \xi_{0} \in \mathbb{H}_{0} \ \vline \ \int_{0}^{\infty}e^{2\alpha^{+}s}|\xi(s)|^{2}_{-\gamma}ds < +\infty \right\} =\\= 
		\left\{ \xi_{0} \in \mathbb{H}_{0} \ \vline \ \int_{0}^{\infty}e^{2\alpha(s;q)}|\xi(s)|^{2}_{-\gamma}ds < +\infty \right\},
	\end{split}
\end{equation}
where the first equality is given by definition and the second comes from the inequality $\alpha_{0}(\cdot) \leq \alpha^{+}$ and the fact that both sets are $j$-dimensional linear subspaces of $\mathbb{H}_{0}$.

Moreover, from \eqref{EQ: SCP2inequality} it is clear \textbf{(SCP2)} will be satisfied for $\alpha_{0}(\cdot)$ changed to $\alpha_{0}(\cdot) + \varepsilon$ for any $\varepsilon<\delta/2$. Thus, Theorem \ref{TH: NormalHyperbolicity} guarantees that $\mathfrak{A}$ is normally hyperbolic with $\varkappa^{h}(\cdot) := \alpha_{0}(\cdot)$ and $\varkappa^{v}(\cdot) := \alpha_{0}(\cdot) + \varepsilon$ provided that $\alpha_{0}$ is continuous on $\mathfrak{A}$.

To show that $\mathfrak{A}$ is robust under perturbations, let $P = \operatorname{Id} - 2\operatorname{Pr}_{j}$ denote the self-adjoint operator corresponding to $V(\cdot)$ and note that \eqref{EQ: SCP2inequality} is equivalent to that for all $v_{0} \in \mathbb{H}_{0}$ and $\xi_{0} \in \mathbb{H}_{1}$ we have (see Lemma 4.2 in \cite{KostiankoZelik2015})
\begin{equation}
	2(-A\xi_{0} - F'(v_{0})\xi_{0} + \alpha_{0}(v_{0}) \xi_{0}, P\xi_{0})_{0} \leq -\delta |\xi_{0}|^{2}_{-\gamma}.
\end{equation}
Then it is clear that any $C^{1}$-perturbation $\widetilde{F}$ of $F$ such that $\sup_{v \in \mathbb{H}_{0}}|F'(v) - \widetilde{F}'(v)| < \delta / 2$ preserves \textbf{(SCP2)} for the perturbed system. This allows to apply Theorem \ref{TH: RobustnessIM} to get the robustness of $\mathfrak{A}$.

One can verify \textbf{(SCP2)} in the case when the nonlinearity $F$ satisfies the Spatial Averaging Principle \cite{KostiankoZelikSA2020,KostiankoZelik2015}. Since its definition and applications require a lot of technical preparations, we refer the interested reader to the paper of A.~Kostianko et al. \cite{KostiankoZelikSA2020}, where applications to scalar diffusion equations, Cahn-Hilliard type equations with periodic boundary conditions and modified Navier-Stokes equations are given. Moreover, in the recent paper of A.~Kostianko, C.~Sun and S.~Zelik \cite{KostiankoSunZelik2021IMGinzburgLandau}, the Spatial Averaging Principle is applied to 3D complex Ginzburg-Landau equation after a change of variables, leading to the so-called spatio-temporal averaging.
\section*{Acknowledgements}
The author is thankful to Sergei Zelik for many useful discussions on the topic.

\bibliographystyle{amsplain}

\end{document}